\documentclass[12pt,onecolumn,draftcls]{IEEEtran}

\usepackage{cite}
\usepackage{amsmath,amssymb,amsfonts}
\usepackage{graphicx}
\usepackage{textcomp}
\usepackage{gensymb}
\usepackage{color}         % required for colored fonts
\usepackage{amsmath} % assumes amsmath package installed
\usepackage{amssymb}  % assumes amsmath package installed
\usepackage{mathtools}
\mathtoolsset{showonlyrefs}
 
\usepackage{amsthm}

\usepackage{diagbox}

\usepackage{silence}
\usepackage{caption} 
\newcommand{\newnew}[1]{#1}

\usepackage{comment}
\usepackage{ifthen}
\usepackage{amsmath}
\usepackage[bb=dsserif]{mathalpha} % MathBB 1
\usepackage{wrapfig}
\usepackage{lipsum} % For dummy text
\usepackage[export]{adjustbox} % For image alignment
\usepackage{bbm}
\usepackage{hyperref}
\usepackage{verbatim}
\usepackage{stackrel}
\PassOptionsToPackage{normalem}{ulem}
\usepackage{ulem}
\usepackage{algorithm,algpseudocode}
\usepackage[linesnumbered,ruled,vlined,algo2e,noend]{algorithm2e}
\usepackage{tikz}
\usepackage{pgfplots}
\pgfplotsset{compat=1.18}

\usepackage{etoolbox}

\DeclareMathOperator*{\argmin}{arg\,min}

\def\BibTeX{{\rm B\kern-.05em{\sc i\kern-.025em b}\kern-.08em
    T\kern-.1667em\lower.7ex\hbox{E}\kern-.125emX}}
\begin{document}
\newcommand{\gb}[1]{#1}
\newcommand{\gbb}[1]{\textcolor{purple}{#1}}
\newcommand{\mh}[1]{\textcolor{red}{#1}}
\newcommand{\mhmargin}[1]{\marginpar{\textcolor{red}{\tiny #1}}}

\newcommand{\T}{\mathcal{T}}
\newcommand{\Ts}{\mathcal{T}_s}
\newcommand{\X}{\mathcal{X}}
\newcommand{\R}{\mathbb{R}}
\newcommand{\N}{\mathbb{N}}
\newcommand{\C}{\mathcal{C}}
\newcommand{\D}{\mathcal{D}}
\newcommand{\K}{\mathcal{K}}
\newcommand{\B}{\mathcal{B}}
\newcommand{\HH}{\mathcal{H}}
\newcommand{\V}{\mathcal{V}}
\newcommand{\Z}{\mathcal{Z}}

\newcommand{\at}[2]{\alpha ( {#1} ; t_{#2} )}
\newcommand{\bt}[1]{\beta ( {#1} )}
\newcommand{\pt}[2]{\rho^{#1}_{#2}}
\newcommand{\rhoz}{\rho_{z}}
\newcommand{\rhozz}{\rho_{z+1}}
\newcommand{\rz}{r_{z}}
\newcommand{\Dz}{D_{z}}
\newcommand{\Ez}{E_{z}}
\newcommand{\Fz}{F_{z}}
\newcommand{\Gz}{G_{z}}

\newcommand{\rzz}{r_{z+1}}
\newcommand{\Dzz}{D_{z+1}}
\newcommand{\Ezz}{E_{z+1}}
\newcommand{\Fzz}{F_{z+1}}
\newcommand{\Gzz}{G_{z+1}}

\newcommand{\diam}{\textnormal{diam}}
\newcommand{\dx}{d_{\X}}

\newcommand{\gammaz}{\gamma_z}
\newcommand{\gammazz}{\gamma_{z+1}}

\newcommand{\az}{a_{z}}
\newcommand{\bz}{b_{z}}
\newcommand{\cz}{c_{z}}
\newcommand{\dz}{d_{z}}
\newcommand{\azz}{a_{z+1}}
\newcommand{\bzz}{b_{z+1}}
\newcommand{\czz}{c_{z+1}}
\newcommand{\vz}{\bar{v}_{z}}
\newcommand{\kz}{\hat{k}_{z}}
\newcommand{\kzz}{\hat{k}_{z+1}}
\newcommand{\kzero}{\hat{k}_{0}}
\newcommand{\kone}{\hat{k}_{1}}
\newcommand{\ktwo}{\hat{k}_{2}}

% \newcommand{\az}[1]{
% \ifthenelse{\equal{#1}{}}{a_z}{a_{#1}}
% }
% \newcommand{\bz}[1]{
% \ifthenelse{\equal{#1}{}}{b_z}{b_{#1}}
% }
% \newcommand{\cz}[1]{
% \ifthenelse{\equal{#1}{}}{c_z}{c_{#1}}
% }

% Convergence Section
\newcommand{\s}[3]{
\ifthenelse{\equal{#3}{}}{s(#2;\ind{#3})}{s_{#1}(#2)}
}
\newcommand{\Lx}{L_{z}}
\newcommand{\Lxx}{L_{z+1}}
\newcommand{\Mx}{M_{z}}
\newcommand{\Mxx}{M_{z+1}}
\newcommand{\Lg}[1]{L_{g,#1}}

\newcommand{\gmax}{\bar{g}(\tz)}

\newcommand{\Lq}[1]{L^{[#1]}_Q}
\newcommand{\Lr}[1]{L^{[#1]}_r}

% \newcommand{\Ly}{L_{y,\ell}}
% \newcommand{\Lj}{L_{J,\ell}}
% \newcommand{\Lt}{L_{\ell}}

% Problem Statement Section
\newcommand{\J}[2]{J \left(#1,#2\right) }
\newcommand{\f}[2]{f \left(#1;#2\right) }
\newcommand{\fl}[2]{f_l \left(#1;#2\right) }
\newcommand{\g}[2]{g (#1;#2) }
\newcommand{\gu}[2]{g_u \left(#1;#2\right) }
\newcommand{\gc}[2]{ \bar{g} \left(#1;#2\right) }
\newcommand{\h}[2]{h \left(#1;#2\right) }
\newcommand{\hu}[2]{h_u \left(#1;#2\right) }

\newcommand{\ind}[2]{t^{#1}_{#2}}
\newcommand{\tl}{t_{\ell}}
\newcommand{\tll}{t_{\ell+1}}
\newcommand{\tz}{t_{z}}
\newcommand{\tzz}{t_{z+1}}

\newcommand{\Q}[2]{
\ifthenelse{\equal{#1}{}}{Q\left(#2\right)}{Q^{[#1]}\left(#2\right)}
}
\newcommand{\rr}[2]{
\ifthenelse{\equal{#1}{}}{r\left(#2\right)}{r^{[#1]}\left(#2\right)}
}

\newcommand{\Qmap}[1]{
\ifthenelse{\equal{#1}{}}{\Upsilon \left( \Q{}{\cdot} \right)}{\Upsilon \left( \Q{}{#1} \right)}
}

\newcommand{\rmap}[1]{
\ifthenelse{\equal{#1}{}}{\Xi \left( \rr{}{\cdot} \right)}{\Xi \left( \rr{}{#1} \right)}
}

\newcommand{\solSeries}{\left\{\x{*}{}{\ind{}{\ell}} \right\}_{\ind{}{\ell} \in \T}}

\newcommand{\x}[3]{
\ifthenelse{\equal{#3}{}}{x^{#1}_{#2}}{x^{#1}_{#2}(#3)}
}

%\newtheorem{dfn}{Definition}
%\newtheorem{Prop}{Proposition}

% Define a new counter with letters
% \newcounter{lettercounter}
% \renewcommand{\thelettercounter}{\alph{lettercounter}}

\newtheorem{dfn}{Definition}
\newtheorem{theorem}{Theorem}
\newtheorem{assumption}{Assumption}
\newtheorem{lemma}{Lemma}
\newtheorem{appendice}{Appendix}
\newtheorem{problem}{Problem}
\newtheorem{remark}{Remark}
\newtheorem{corollary}{Corollary}

\renewcommand\theappendice{\Alph{appendice}}

\title{Distributed Asynchronous Time-Varying Quadratic Programming with Asynchronous Objective Sampling}
\author{Gabriel Behrendt, Zachary I. Bell, Matthew Hale
\thanks{Gabriel Behrendt and Matthew Hale were supported by 
AFRL under grant FA8651-23-F-A006,
AFOSR under grants FA9550-19-1-0169 and FA9550-23-1-0120,
and ONR under grant N00014-24-1-2331. 
 }
\thanks{Gabriel Behrendt and Matthew Hale are
with the School of Electrical and Computer Engineering
at the Georgia Institute of Technology, Atlanta, GA, USA. 
Emails: \texttt{\{gbehrendt3,mhale30\}@gatech.edu} 
Zachary I. Bell is with AFRL/RW at Eglin AFB. Email: \texttt{zachary.bell.10@us.af.mil} }}

\maketitle

\begin{abstract} \label{sec:abstract}
\newnew{
Existing works on multi-agent time-varying optimization 
allow agents to asynchronously communicate
and/or compute, but do not allow 
asynchronous sampling of objectives. 
Sampling can be difficult to synchronize, and we 
therefore present a multi-agent 
optimization framework that allows asynchrony in
sampling, communications, and computations
for time-varying quadratic programs. 
We show that agents have bounded error when
tracking the solution to the asynchronously sampled problem, 
which solves an open problem for quadratic programs. 
Simulations validate these results. 
}
\end{abstract}
\section{Introduction} \label{sec:intro}
\newnew{
Time-varying optimization problems 
have been formulated in 
robotics~\cite{arslan2019sensor}, signal processing~\cite{jakubiec2012d}, machine learning~\cite{yang2016online}, 
and many other applications~\cite{simonetto2020time}. 
Such problems can, for example, model time-varying demands in power distribution systems~\cite{tang2017real} or robot navigation in cluttered dynamic environments~\cite{arslan2016exact}, among others~\cite{dall2016optimal,cortes2017coordinated}. 
These applications often involve multi-agent systems, and it is common
for agents to sample a 
continuously varying objective and minimize the objective function
that they sample~\cite{simonetto2016class}. 
Existing work has assumed that all agents take samples at the same times, but
it can be difficult to synchronize
agents' sampling,
just as it can be difficult to synchronize their communications
and computations~\cite{bertsekas1989parallel}. 
}

\newnew{
Accordingly, 
in this work we present and analyze a distributed gradient optimization algorithm 
in which agents can 
asynchronously compute, communicate, and sample their objective. 
Asynchronous sampling can cause agents to disagree about the objective function that they minimize, which presents
a new technical challenge that is not typically faced in multi-agent optimization. Indeed, 
to the best of our knowledge, this work is the first
to allow sampling, computing, and communicating to all be executed
asynchronously. 
}

\newnew{
We consider a distributed asynchronous
block coordinate descent algorithm.
Algorithms of this kind have been shown
to be robust to asynchrony in computations and communications
in many time-invariant optimization problems~\cite{bertsekas1989parallel,yazdani2021asynchronous,cannelli2020asynchronous,tseng1991rate,luo1992error,ubl21,hendrickson23,hendrickson23b}. 
Recently, such algorithms have been applied to
time-varying problems as well~\cite{behrendt2023totally,behrendt2023distributed}, though
always with synchronous objective sampling.
We show for the first time that this algorithm is robust to asynchronous objective sampling as well. 
We apply it to time-varying quadratic programs (QPs), which
have been used in applications such as motion planning~\cite{miao2015solving, mohammed2015dynamic}, controller design~\cite{hamed2020quadrupedal, johansen2004constrained}, signal processing~\cite{mattingley2010real}, and machine learning~\cite{zhang2021barrier}, among others.
}

\newnew{
Asynchronous objective sampling causes agents to have different objectives
onboard, which has required several new technical developments to analyze. 
This paper makes the following contributions, all of which
are new to the literature: 
\begin{itemize}
    \item \newnew{We show that asynchronous sampling 
    can cause agents to
    solve a time-varying nonconvex QP, even when the original problem
    is a strongly convex QP (Theorem~\ref{thm:agg})}.
    \item \newnew{We prove that a distributed asynchronous block coordinate descent algorithm
    tracks the solutions of time-varying nonconvex QPs
    with bounded error (Theorem~\ref{thm:1}); 
    this result may be of independent interest}.
    \item \newnew{We bound the distance between (i) the solutions to the original
    QPs and (ii) the solutions to the nonconvex QPs that result
    from asynchronous sampling. We use this bound to 
    show that agents track the solutions of the 
    original QPs to within an error ball of known 
    size~{(Theorem~\ref{thm:2})}}.
    \item We show that by synchronizing agents' computations, communications, and objective sampling we recover 
    (up to constant factors) existing linear convergence results for time-varying optimization (Corollaries~\ref{cor:synch1} and~\ref{cor:synch2}).
    \newnew{This paper therefore generalizes several existing works.}
    \item We demonstrate bounded tracking error for two asynchronously sampled time-varying QPs~{(Section~\ref{sec:simulation})}.
\end{itemize}
}

\newnew{
%Numerous algorithms have been developed for centralized time-varying optimization~\cite{popkov2005gradient,tang2022running,bastianello2019prediction}.
Distributed time-varying optimization has been studied in both continuous time~\cite{rahili2016distributed,wu2022distributed} and discrete time~\cite{bernstein2018asynchronous,simonetto2017decentralized,ling2013decentralized,yi2020distributed}.
Specifically,~\cite{bernstein2018asynchronous} proposed an algorithm 
that uses asynchronous communications for discrete-time time-varying convex optimization problems. 
Developments in~\cite{simonetto2017decentralized} provide a decentralized prediction-correction method for 
the same class of problems. Work in
\cite{ling2013decentralized} developed a decentralized alternating direction method of multipliers algorithm to track the solution of such problems, while 
\cite{yi2020distributed} considered distributed online convex optimization with time-varying constraints.
}

\newnew{
Those and other existing works require agents to sample their objective simultaneously, 
though it can be difficult to synchronize sampling, especially when communications
and computations are asynchronous. 
We differ from existing works
by allowing agents to asynchronously sample their objective, 
which we show can require agents to solve a time-varying nonconvex optimization problem. 
New technical challenges come both from 
(i) quantifying how agents track the solutions of such problems and 
(ii) relating the solutions of those nonconvex problems to the solutions of the original QPs. 
}

\newnew{
Distributed asynchronous sampling was named
in~\cite{simonetto2020time} as an open problem in time-varying optimization, and
we solve it for a class of quadratic programs.
To the best of our knowledge we are the first to address asynchronous sampling of this kind. 
This paper extends our previous work in~\cite{behrendt2023totally,behrendt2023distributed}, which
required synchronous sampling of objectives. 
 }

\newnew{
The rest of the article is organized as follows. Section~\ref{sec:problem} gives 
problem statements, Section~\ref{sec:algorithm} presents the main algorithm, Section~\ref{sec:convergence} analyzes convergence, 
Section~\ref{sec:simulation} presents simulations, 
and Section~\ref{sec:conclusion} concludes.
}

\noindent\textbf{Notation} 
We use~$\R$ and~$\R_{+}$ to denote the reals and non-negative reals, respectively.
We use~$\N$ to denote the naturals and~${\N_0 = \N \cup \{ 0 \}}$.
The symbol~$\| \cdot \|$ denotes the Euclidean norm.
The diameter of a compact set $\X \subset \R^n$ is denoted $\dx \coloneqq \sup_{x,y \in \X} \| x-y \|$. 
For~$a, b \in \N$, the set~$\{a, \ldots, b\}$ is empty if~$b < a$. 
We use~$\Pi_\mathcal{Z}$ for the Euclidean projection onto a non-empty, closed, convex set~$\mathcal{Z}$, i.e.,~$\Pi_\mathcal{Z}[v] = \argmin_{z \in \mathcal{Z}} \| v - z \|$.
We  use~$[d]=\{1,\dots,d\}$ for~$d \in \N$ and~$I_n$ for the~$n \times n$ identity matrix. 
For~$\X \subseteq \R^n$,~$C^2( \X )$ is the set of twice continuously differentiable functions from~$\X$ to $\R$. 
For~$\bar{u} > 0 $ and~${0 < \varphi < \psi}$, we define the sets of functions 
$\D = \{f \in C^2( \X ) : \| \nabla f(x) \| 
\leq \bar{u}, \nabla^2 f(x) \preceq \psi I_n \,\, \forall x \in \X \}$ and
$\mathcal{S} = \{f \in \D : \varphi I_n \preceq \nabla^2f(x) \,\, \forall x \in \X\}$.

\section{Problem Formulation} \label{sec:problem}

This section states the problem that is the focus of this paper. 

\addtocounter{problem}{-1}

\subsection{Problem Statement}
\begin{problem}[Preliminary; continuous-time] \label{prob:continuous}
    Given~$J : \R^{n} \times \R_+ \rightarrow \R$ over the time horizon~$t \in [0,t_f]$, using~$N \in \N$ agents that asynchronously 
    compute and communicate, solve
    \begin{equation}
        \underset{x\in \X}{\textnormal{minimize}} \  \J{x}{t} := x^T Q(t) x + r(t)^T x,
    \end{equation}
    where~$Q(t) \in \R^{n \times n}$,~$r(t) \in \R^n$, 
    $\X \subseteq \R^n$, and~$t_f \in \R_+$.
    \hfill $\Diamond$
\end{problem}
% \red{We may not need this to be over a finite time horizon, i.e.,~$t\in [0,t_f]$. Let's discuss.}
We make the following assumptions about Problem~\ref{prob:continuous}.
\begin{assumption} \label{ass:contTime}
    $Q(t)  =  Q(t)^T  \succeq  \xi I_n$ for all~$t$, where~$\xi  >  0$.  \hfill $\lozenge$
\end{assumption}
% \gb{Note we can be non-symmetric}
\begin{assumption} \label{ass:setConstraint}
    There exist sets~$\X_1, \ldots, \X_N$ such that~$\X$ can be decomposed via ${\X = \X_{1} \times \dots \times \X_{N}}$, 
    where~$N$ is the number of agents and~$\X_{i} \subseteq \R^{n_i}$ is non-empty, compact, and polyhedral for all~$i \in [N]$, with~$\sum_{i \in [N]} n_i = n$. 
    \hfill $\lozenge$ 
\end{assumption}

Assumption~\ref{ass:contTime} ensures that~$\J{\cdot}{t}$ is~$\xi$-strongly convex for all~$t\in [0,t_f]$. 
\gb{It is easily satisfied in applications, e.g., when minimizing the squared distance to a moving target.
If~$Q(t) \not \succeq \xi I_n$, then 
the forthcoming algorithm may still track minimizers, though the analysis would be
more complex due to the possible non-uniqueness of the solution to Problem~\ref{prob:discrete}.
}
For each~$i \in [N]$, agent~$i$ will compute new values of~$x_i$, and 
Assumption~\ref{ass:setConstraint} allows each agent to project its block of decision variables onto a constraint set independently, 
which ensures that the overall set constraint is satisfied. 
\gb{ 
Without it, there is in general no way to ensure the satisfaction
of set constraints under asynchrony.}
Assumptions~\ref{ass:contTime} and~\ref{ass:setConstraint} guarantee that the solution trajectory of Problem~\ref{prob:continuous} exists and is unique for all~$t \in [0, t_f]$. 

Given a vector~$v \in \R^n$, where~$n= \sum^N_{i=1} n_i$, 
$v^{[1]}$ denotes the first~$n_1$ entries of~$v$,~$v^{[2]}$ denotes the next~$n_2$ entries, etc. 
For an~$n \times n$ matrix~$A$,~$A^{[1]}$ denotes the first~$n_1$ rows of~$A$,~$A^{[2]}$ denotes the next~$n_2$ rows, etc. 
We make the following assumption on the evolution of~$\Q{i}{\cdot}$ and~$\rr{i}{\cdot}$. 

 \begin{assumption} \label{ass:lipTime}
     For all~$t_1,t_2 \in [0,t_f]$ and~$i \in [N]$ there exist~$\Lq{i},\Lr{i}>0$ such that~$\| \Q{i}{t_1} - \Q{i}{t_2} \| \leq \Lq{i} |t_1 - t_2 |$ and~$\| \rr{i}{t_1} - \rr{i}{t_2} \| \leq \Lr{i} |t_1 - t_2 |$. \hfill $\lozenge$ 
 \end{assumption}
 
 \gb{
 Assumption~\ref{ass:lipTime} says that~$\Q{i}{\cdot}$ and~$\rr{i}{\cdot}$ 
 are Lipschitz functions of time
 for all~$i \in [N]$, and it is satisfied, for example, by any~$\Q{i}{\cdot}$
 and~$\rr{i}{\cdot}$ that are continuously differentiable. 
 Without it, movement of minimizers could be arbitrarily fast over time, making them impossible to track.
 }

Agents asynchronously sample the objective in Problem~\ref{prob:continuous}. 
They are capable of sampling it every~$t_s > 0$ seconds, 
but they are not guaranteed to take a sample at every possible sample time.
Agents only take samples at times
of the form~$t = \ell t_s$ for~$\ell \in \N_0$.
Define~$\mathcal{T}_{all} = \{\ell t_s\}_{\ell \in \N_0}$. 
The following problem includes all objectives that an agent might sample.  

\begin{problem}[Discrete-Time] \label{prob:discrete}
     Given~$\mathcal{S} \ni f : \R^{n} \times \R_+ \rightarrow \R $ and~$\ell \in \N_0$, using~$N \in \N$ agents that asynchronously compute, communicate, and sample the objective, drive~$x$ to track 
    \begin{equation}
        \argmin_{x \in \X} \f{x}{\ell t_s} \coloneqq x^T \Q{}{\ell t_s} x + \rr{}{\ell t_s}^T x. \tag*{\text{$\Diamond$}}
    \end{equation}    
\end{problem}

We denote the minimum cost of Problem~\ref{prob:discrete} at time~$\ell t_s$ as 
    $f^*(\ell t_s) \coloneqq \underset{x \in \X}{\min} \ \f{x}{\ell t_s} \text{ for } \ell \in \N_0$.
\gb{Time-varying quadratic programs have been used in control~\cite{hours2015parametric}, robotics~\cite{arslan2016exact}, and machine learning~\cite{balavoine2015discrete}.
Multi-agent systems arise in these applications, 
e.g., by using parallel computing. 
Multi-agent algorithms are thus needed, and that is
what we develop. 
}
 
Agents are modeled as sampling Problem~\ref{prob:discrete} asynchronously. 
Let~$\T^i$ be the set of discrete time indices at which agent~$i$ samples Problem~\ref{prob:discrete}. 
Define~$\Ts \coloneqq \cup_{i=1}^N\T^i = \{t_0, t_1, \dots, t_T \}$ with~$T \in \N$, which contains all times at which at least one agent samples Problem~\ref{prob:discrete}. 
Then for~$\tz \in \Ts$, 
let~$\theta_i \left( \tz \right) \coloneqq \max \{ t \in \T^i : t \leq \tz\}$, which is the time index of agent~$i$'s latest sample that is not
taken after time~$\tz \in \Ts$. 
 
\begin{assumption} \label{ass:sample}
    For all~$i \in [N]$ and~$\tz\in \Ts$ there exists~$\Delta > 0$ such that~$| \theta_i (\tzz) - \theta_i(\tz) | \leq \Delta$. \hfill $\lozenge$ 
\end{assumption}

\gb{
Assumption~\ref{ass:sample} holds as long as no agent ever permanently stops sampling the objective.  
It is easily satisfied in practice because~$\Delta$ can be arbitrary.
}
When agent~$i$ computes a new value for~$x_i$, it 
uses a gradient descent law with the 
gradient~$\nabla_{x_i} \f{\x{}{}{}}{\cdot} = Q^{[i]}(\cdot) \x{}{}{} + r^{[i]}(\cdot)$. 
\gb{Due to asynchronously sampling Problem~\ref{prob:discrete}, agents will usually 
compute gradients with blocks 
of~$Q(\cdot)$ and~$r(\cdot)$ that were sampled at different times. 
%and agent~$i$'s gradient at time~$\tz$ and a point~$x \in \X$ is~$\Q{i}{\theta_i \left( \tz \right)} \x{}{}{} + r^{[i]}(\theta_i \left( \tz \right))$.
%Each agent is minimizing the local objective function that is has onboard, 
%and assembling their gradients shows their updates 
These computations are equivalent to 
minimizing an aggregate objective function that no agent has onboard. 
}

\begin{problem} [Aggregate Problem] \label{prob:aggregate}
    Given~$g : \R^{n} \times \R_+ \rightarrow \R$, using~$N \in \N$ agents that asynchronously 
    compute and communicate, for all~$\tz \in \Ts$, drive~$x$ to track the solution of
    \begin{equation}
        \underset{x\in \X}{\textnormal{minimize}}  \  \g{x}{\tz} \coloneqq x^T \Qmap{\tz} x + \rmap{\tz}^T x,
    \end{equation}
    where
    \begin{equation}
        \Qmap{\tz} = \big[ \Q{1}{\theta_1 \left( \tz \right)}^T, \dots, \Q{N}{\theta_N \left( \tz \right)}^T \big]^T
    \end{equation}
    and
    \begin{equation}
        \rmap{\tz} = \big[\rr{1}{\theta_1 \left( \tz \right)}^T, \dots, \rr{N}{\theta_N \left( \tz \right)}^T \big]^T. 
    \end{equation}    
    $\lozenge$
    % ~$\Qmap{\tz} \coloneqq \big[ \Q{1}{\theta_1 \left( \tz \right)}^T, \dots, \Q{N}{\theta_N \left( \tz \right)}^T \big]^T$ and~$\rmap{\tz} \coloneqq \big[\rr{1}{\theta_1 \left( \tz \right)}^T, \dots, \rr{N}{\theta_N \left( \tz \right)}^T \big]^T$. \hfill $\lozenge$
    % \begin{equation}
    %     \begin{matrix}
    %     \Qmap{\tz}  \coloneqq 
    %     \begin{bmatrix} 
    %          \Q{1}{\theta_1 \left( \tz \right)} \\             
    %          \vdots \\ 
    %          \Q{N}{\theta_N  \left( \tz \right)} 
    %     \end{bmatrix}, 
    %      \rmap{\tz}  \coloneqq 
    %     \begin{bmatrix} 
    %         \rr{1}{\theta_1 \left( \tz \right)} \\             
    %         \vdots \\ 
    %         \rr{N}{\theta_N  \left( \tz \right)}
    %     \end{bmatrix}.
    %     \end{matrix}
    % \end{equation} \hfill $\lozenge$ 
\end{problem}

\gb{Problem~\ref{prob:aggregate}'s objective simultaneously produces
each agent's local gradient, which indicates that agents are indeed minimizing it.}
%Problem~\ref{prob:aggregate} does not necessarily change every~$\ell t_s$ seconds for~$\ell \in \N_0$ because it only changes when an agent samples Problem~\ref{prob:discrete}. 
While Problem~\ref{prob:discrete} has a strongly convex objective at all times,
Problem~\ref{prob:aggregate} can be nonconvex. 
%\newnew{
%which we prove next by constructing a choice of Problem~\ref{prob:discrete}
%that satisfies all assumptions and leads to a Problem~\ref{prob:aggregate}
%that is nonconvex. 
%}

\begin{theorem} \label{thm:agg}
% \textcolor{red}{We'll add this in.}
There exist choices of Problem~\ref{prob:discrete} such that, under Assumptions~\ref{ass:contTime}-\ref{ass:sample}, 
asynchronous objective function sampling leads to a nonconvex aggregate problem in Problem~\ref{prob:aggregate}.
\end{theorem}
\begin{proof}
 \newnew{We prove that such choices of Problem~\ref{prob:discrete} exist by explicitly
 constructing one.}
 Consider
    \begin{equation}
        Q(t) = \begin{bmatrix}
            1.2+\cos(t) & \sin(t) \\
            \sin(t) & 1.2
        \end{bmatrix}, \quad r(t) \equiv 0.
    \end{equation} 
    Then~$Q(t) \succeq \frac{1}{5}I_2$ for all~$t$
    and Assumption~\ref{ass:contTime} is satisfied. 
    Suppose there are~$N = 2$ agents that each update a scalar
    and~$\X = [0, 1]^2$.
    Then Assumption~\ref{ass:setConstraint} is satisfied. 
    %The eigenvalues of~$Q(t)$ are positive at all times. 
    %\begin{align}
    %    \lambda_{1,2}(t) %= \frac{(a+d) \pm \sqrt{(a+d)^2 - 4(ad-bc)}}{2}
    %    % &= \frac{2.4 + \cos(t) \pm \sqrt{(2.4+\cos(t))^2 - 4(1.2(1.2+\cos(t))-\sin^2(t))}}{2} \\
    %    % &= \frac{2.4 + \cos(t) \pm \sqrt{5.76 + 4.8\cos(t) +\cos(t)^2 - (5.76+4.8\cos(t)-4\sin^2(t))}}{2} \\
    %    = \frac{2.4 + \cos(t) \pm \sqrt{\cos^2(t) + 4\sin^2(t)}}{2} > 0.
    %\end{align}
    %Then~$J(\cdot,t)$ is strongly convex for all~$t \geq 0$. 
    Suppose agent~$1$ 
samples the objective at times in~$\T^1 \coloneqq \{(2\xi + 1)\pi + \frac{\pi}{4} \}$
and agent~$2$ samples it
at times in~$\T^2 = \{ (2\xi + 1)\pi + \frac{\pi}{2} \}$ for all~$\xi \in \N_0$. 
It is easy to verify that Assumptions~\ref{ass:lipTime} and~\ref{ass:sample} hold. 
Then~$\T_s = \T^1 \cup \T^2$, and for all~$\tz \in \Ts$ we have
    %Furthermore, the resulting~$\Qmap{\tz}$ is
    \begin{equation}
        \Qmap{\tz} 
        %= \begin{bmatrix}
        %    1.2+\cos(\theta_1 ( \tz )) & \sin(\theta_1 ( \tz )) \\
        %    \sin(\theta_2 ( \tz )) & 1.2
        %\end{bmatrix} 
        = \begin{bmatrix}
            1.2 - \frac{\sqrt{2}}{2}  & - \frac{\sqrt{2}}{2} \\
            -1 & 1.2
        \end{bmatrix}. 
    \end{equation}     
    When agents minimize~$x^T\Qmap{\tz}x$, only the symmetric
    part~$Q_s(\tz) = \frac{1}{2}(\Qmap{\tz} + \Qmap{\tz}^T)$ contributes
    to its value. That minimization is equal
    to minimizing~$x^TQ_s(\tz)x$, and~$Q_s(\tz)$ 
    %Every square matrix can be uniquely decomposed as~$Q=Q_s+Q_a$ where~$Q_s = Q_s^T = \frac{Q+Q^T}{2}$ and~$Q_a=-Q_a^T=\frac{Q-Q^T}{2}$. Moreover, only the symmetric part~$Q_s$ contributes to the convexity of the problem because~$x^T Q_a x = 0$ for all~$x\in \R^n$ so we can examine the eigenvalues of~$\Upsilon(Q_s(\tz))$ to determine the convexity of~$\g{x}{\tz}$. We observe
    %\begin{align}
    %    Q_s(\tz) 
    %                = \begin{bmatrix}
    %                    1.2 - \frac{\sqrt{2}}{2}& -\frac{\sqrt{2}}{4}-\frac{1}{2} \\
    %                    -\frac{\sqrt{2}}{4}-\frac{1}{2} & 1.2
    %                \end{bmatrix}        
    %\end{align}
    has eigenvalues~$\lambda_{1,2} \approx 1.77,-0.08$. 
    The negative eigenvalue shows that agents' objective
    is nonconvex. 
    %Therefore, in this example, agents asynchronously sampling Problem~\ref{prob:discrete} according to~$\T^1,\T^2$  results in Problem~\ref{prob:aggregate} being nonconvex since~$\Upsilon(Q_s(\tz)) \nsucc 0$ for all~$\tz \in \Ts$.   
\end{proof}

\gb{Nonconvexity of Problem~\ref{prob:aggregate} is due to agents' 
disagreement about the objective, which comes from asynchronous sampling.
Making it convex would require all agents to agree on 
the objective, which 
is difficult to enforce under asynchrony.}
Problem~\ref{prob:aggregate} may have several local solutions, and 
we use
    $\X^*(\tz) \coloneqq \{ x \in \R^n \ | \ x = \Pi_\X [ x - \nabla_x \g{x}{\tz}] \}$
to denote its solution set for each~$\tz \in \Ts$.

\begin{assumption} \label{ass:sigma}
    For all~$\tz \in \Ts$, and for any~$\x{*}{}{\tz} \in \X^*(\tz)$ and any~$\x{*}{}{\tzz} \in \X^*(\tzz)$, there exists~$\sigma_{z+1} > 0$ such that~$\| \x{*}{}{\tzz} - \x{*}{}{\tz} \| \leq \sigma_{z+1}$. \hfill $\lozenge$ 
 \end{assumption}

    \gb{
        Assumption~\ref{ass:sigma} 
        ensures that successive minimizers are not arbitrarily far apart.
        It is always satisfied with~$\sigma_{z} = \dx$ for all~$\tz \in \Ts$, but
        smaller values can be used if they are known. 
        %Since~$\X^*(\tz)$ may not be a singleton, Assumption~\ref{ass:sigma} generalizes the typical assumption made when the objective function is strongly convex, e.g.,~\cite[Assumption 1]{simonetto2020time}, which is that the unique minimizer of a time-varying strongly convex objective function has bounded movement over time. 
    }

\begin{assumption} \label{ass:lowerBound}
    For all~$\tz \in \Ts$ and~$x \in \X$,~$\g{x}{\tz} \geq 0.$ \hfill $\lozenge$ 
\end{assumption}

    \gb{
    All objectives are bounded below due to our earlier assumptions, and 
    one can thus add a constant to each objective function to make it non-negative
    to satisfy Assumption~\ref{ass:lowerBound}. 
    }

\section{Asynchronous Update Law} \label{sec:algorithm}
% This section presents the proposed asynchronous algorithm.
\newnew{
Now we present the setup for asynchronous multi-agent optimization
of time-varying objectives with asynchronous sampling.
While one
could envision developing a  new algorithm for this setting, we show
that is not necessary. Instead, 
in Algorithm~\ref{alg:myAlg} below
we revisit and adapt the classic 
asynchronous block coordinate descent algorithm, which was shown 
in~\cite[Chapter 7]{bertsekas1989parallel} to be robust
to asynchronous communications and computations.
We show for the first time that it is robust to asynchronous
objective sampling as well. 
}

\subsection{Timescale Separation}
% Problem~\ref{prob:discrete} is indexed by the discrete time index~$\ind{}{\ell}$. We start the index~$\ell = 0$ and increment by~$1$, i.e.~$\ell \to \ell + 1$, every~$t_s$ seconds. 
% Similarly, 
Problem~\ref{prob:aggregate} is indexed by~$\ind{}{z}$.
It starts with~$z = 0$ and increments~$z$ by~$1$ at
the times at which at least one agent samples Problem~\ref{prob:discrete}.
We index agents' computations and communications over the time index~$k$ 
because these operations need not occur at the same times as their samples. 
Similar to~\cite{hauswirth2020timescale, colombino2019online}, 
we use a ``correction-only'' algorithm that does not predict changes in agents' objectives 
since it would be difficult to do so under asynchronous sampling. 
Some timescale separation is therefore required~\cite{simonetto2017decentralized}
between the changes in objectives and the agents' communications and computations.

\begin{assumption} \label{ass:timescale}
     In Problem~\ref{prob:aggregate}, for each~$\tz \in \Ts$, there are~${\kappa_z > 0}$ ticks of~$k$ when minimizing~$\g{\cdot}{\tz}$.      
     \hfill $\lozenge$ 
\end{assumption}

%\begin{remark}
   
    \gb{Eliminating Assumption~\ref{ass:timescale} would not necessarily force Algorithm~\ref{alg:myAlg} 
    to fail     to converge, though it would make our theoretical guarantees no longer hold.}
    We use~$\eta_z \in \N$ to denote the total number of ticks of~$k$ that occur before~$\ind{}{z}$ increments to~$\ind{}{z + 1}$, i.e.,~$\eta_z = \sum^{z}_{i=0} \kappa_i$.
We set~$\kappa_{-1}=0$ and~$\eta_{-1}=0$.

\subsection{Asynchronous Block Coordinate Descent (BCD)}
We consider a block-based gradient projection algorithm with asynchronous computations and communications to track the solutions of Problem~\ref{prob:aggregate}. 
Each agent updates only a subset of the entire decision vector, and 
agents' locally updated values are then communicated to other agents over time.
Asynchrony implies that agents receive different values at different times, and thus they have different values of decision variables onboard. 
At any time~$k$, agent~$i$ has a local copy of the decision vector, denoted~$\x{i}{}{k}$.
Due to asynchrony, we allow~$\x{i}{}{k} \neq \x{j}{}{k}$ for~$i \neq j$. 
Within the vector~$\x{i}{}{k}$, agent~$i$ computes updates only to its own sub-vector of decision variables,~$\x{i}{i}{k} \in \R^{n_i}$. 
At any time~$k$, agent~$i$ has (possibly old) values for agent~$j$'s sub-vector of the decision variables, denoted by~$\x{i}{j}{k} \in \R^{n_j}$. These are updated only by agent~$j$ and then communicated by agent~$j$ to agent~$i$ and others. 

At time~$k$, if agent~$i$ computes an update to~$\x{i}{i}{k}$, then its computations use the decision vector~$\x{i}{}{k}$. 
The entries of~$\x{i}{j}{k}$ for~$j  \neq  i$ are obtained by communications from agent~$j$, which may be subject to delays. Therefore, agent~$i$ may (and often will) compute updates to~$\x{i}{i}{k}$ based on outdated information from other agents.

Let~$\K^i$ be the set of times at which agent~$i$ computes an update to~$\x{i}{i}{}$. 
Let~$\C^i_j$ be the set of times at which agent~$i$ receives a transmission of~$\x{j}{j}{}$ from agent~$j$; due to communication delays, these transmissions can be received long after they are sent, and they can be received at different times by different agents. The sets~$\K^i$ and~$\C^i_j$ are only defined to simplify discussion and analysis, and
agents do not need to know them. 
We define~$\tau^i_j(k)$ to be the time at which agent~$j$ originally computed the value of~$\x{i}{j}{k}$ that agent~$i$ has onboard at time~$k$. 
At time~$k$, agent~$i$ stores the vector
\begin{equation}
    \x{i}{}{k} = \left( \x{1}{1}{\tau^i_1(k)}^T,\dots,\x{i}{i}{k}^T,\dots, \x{N}{N}{\tau^i_N(k)}^T \right)^T. \label{eq:1000} 
\end{equation}
We will also analyze the true state of the network, namely 
\begin{equation}
    \x{}{}{k} = \left( \x{1}{1}{k}^T,\x{2}{2}{k}^T,\dots,\x{N}{N}{k}^T \right)^T . \label{eq:1001}
\end{equation} 
We assume that computation and communication delays are bounded, which has been called \emph{partial asynchrony}~\cite{bertsekas1989parallel}. 

\begin{assumption} \label{ass:partialAsynch}
There exists~$B \in \N$ such that
    \begin{enumerate}
    \item For all $i \in [N]$ and all $k \geq 0$, at least one element
    of the set $\left\{ k,k+1,\dots,k+B-1\right\} $ belongs to $\K^{i}$
    \item The bound
    %\begin{equation}
    $\max\left\{ 0,k-B+1\right\} \leq\tau_{j}^{i}\left(k\right)\leq k$
    %\end{equation}
     holds for all $i \in [N]$, $j \in [N]$, and~$k \geq 0$. \hfill $\lozenge$
    \end{enumerate} 
\end{assumption}

\gb{Assumption~\ref{ass:partialAsynch} only needs
to hold for some~$B \in \N$, and that~$B$ does not need to be small. 
It cannot be removed since unbounded delays
can cause divergence~\cite[Section 6.3]{bertsekas1989parallel}. 
}
%We will use it to prove that for each~$\tz \in \Ts$ our algorithm makes progress toward some element of~$\X^*(\tz)$ across intervals of~$B$ time steps (or multiples of~$B$ time steps). 
% \gb{We could reference back to our timescale assumption here and say the partial asynchrony assumption here requires the specific form of the iterates, i.e.,~$\kappa_z = \chi_z + r_zB$. I can also see how talking about that assumption again could raise concerns in the reviewer if we defend it too much. Let me know what you think.}
% For simplicity, for each~$\tz \in \T$ and for some~$v_z \in \N$, we consider~$\kappa_z = v_z B$. Furthermore, we adopt the convention that~$\kappa_{-1}=0$ and~$\eta_{-1}=0$.

%We require a  algorithm that is robust to asynchrony and can be distributed.  
%Therefore, 
%To do so, we propose the projected-gradient update law
We propose to use the projected gradient update law 
\begin{align}
\x{i}{i}{k+1}= & \begin{cases}
\Omega_i \big( \x{i}{}{k} \big) & k\in \K^{i}\\
\x{i}{i}{k} & k\notin \K^{i}
\end{cases}\label{eq:update}  \\
\x{i}{j}{k+1}= & \begin{cases}
\x{j}{j}{\tau_{j}^{i}(k)} & k\in \C_{j}^{i}\\
\x{i}{j}{k} & k\notin \C_{j}^{i}
\end{cases},
\end{align}
where
\begin{equation}
    \Omega_i \big(\x{i}{}{k}\big) 
    = \Pi_{\X_{i}} \Big[\x{i}{i}{k} 
     - \gamma_z \big( Q^{[i]} (\theta_i ( \tz ) ) \x{i}{}{k} + r^{[i]}(\theta_i ( \tz ))  \big) \Big]
\end{equation}
and~$\gammaz > 0$ is the step size. 
Agents' communications
are \emph{not} all-to-all. Instead, two agents communicate only if
one’s updates rely upon the other’s sub-vector of decision
variables. 
Algorithm~\ref{alg:myAlg} provides pseudocode for agents' updates.

\begin{algorithm}[t!]
    \caption{Asynchronous Projected BCD}
    \label{alg:myAlg}
	\SetAlgoLined
	\KwIn{$\x{i}{j}{0}$ for all~$i,j \in [N]$}
	\For{$\tz \in \Ts$}
	{
		\For{$k=\eta_{z-1}+1:\eta_{z}$}
		{
			\For{i=1:N}
			{
			    \For{j=1:N}
			    {
    				\If{$k \in \C^{i}_j$}{    			
                        ~\phantom{} $\x{i}{j}{k+1} = \x{j}{j}{\tau^i_j(k)}$ 
    				}
    				\Else{
    				    ~\phantom{} $\x{i}{j}{k+1} = \x{i}{j}{k}$
    				}
    			}
				\If{k $\in \K^{i}$ }{
					~$\x{i}{i}{k+1} = \Omega_i \left( \x{i}{}{k} \right) $
				}
				\Else{
				~$\x{i}{i}{k+1} = \x{i}{i}{k}$ 
				}
			}			
		}	
	}
\end{algorithm}

\section{Convergence of Algorithm~\ref{alg:myAlg}} \label{sec:convergence}
This section analyzes the convergence of Algorithm~\ref{alg:myAlg}. 
 
 \subsection{Analyzing Objectives in Problem~\ref{prob:aggregate}} \label{ss:p2objectives}
First, we establish key properties of the objective in Problem~\ref{prob:aggregate}. 
For all~$\tz \in \Ts$, from Assumptions~\ref{ass:contTime} and~\ref{ass:setConstraint} 
we see that~$\|\nabla^2_x \g{\cdot}{\tz}\|$ is continuous and~$\X$ is compact. Then $\|\nabla^2_x \g{\cdot}{\tz}\|$ is bounded and~$\nabla_x \g{\cdot}{\tz}$ is Lipschitz on~$\X$. 
For each~$\tz \in \Ts$, we use~$\Lx$ to denote the upper bound on~$\| \nabla^2_x \g{\cdot}{\tz} \|$,
and~$\Lx$ is the Lipschitz constant of~$\nabla_x \g{\cdot}{\tz}$ over~$\X$. 
For each~$\tz \in \Ts$ and for all~$x_1,x_2 \in \X$ we have
%\begin{equation}
    $\| \nabla_x \g{x_1}{\tz}  - \nabla_x \g{x_2}{\tz} \| \leq \Lx \| x_1 - x_2 \|$. %\label{eq:92}
%\end{equation}
Again using continuity and compactness, for all~$x \in \X$ and~$\tz \in \Ts$ there exists~$\Mx > 0$ such that
        $\| \nabla_x \g{x}{\tz} \| \leq \Mx$. 
        There also exists~$M_{g, z} > 0$ such that~$\sup_{x \in \X} g(x; t_z) \leq M_{g, z}$. 
Define~$L_x \coloneqq \max_{\tz \in \Ts} \ \max_{x \in \X}  \| \nabla_x^2 \g{x}{\tz} \|$. 
Then for~$\tz \in \Ts$ and~$x \in \X$, we have 
    $\nabla^2_x \g{x}{\tz} \preceq L_x I_n$. 
Then~$\g{\cdot}{\tz} \in \D$ for each~$\tz \in \Ts$
with~$\bar{u} = M_z$ and~$\psi = L_x$. And the function $\g{\cdot}{\tz}$ is Lipschitz, so 
for each~$\tz \in \Ts$
there exists~$\Lg{z}>0$ such that
%\begin{equation} \label{eq:1004}
        $| \g{x_1}{\tz} - \g{x_2}{\tz} | \leq \Lg{z} \| x_1 - x_2 \|$. 
%\end{equation}

To keep track of computations, we define~$s_i(k) \in \R^{n_i}$ as 
\begin{equation}
    \s{i}{k}{\ell}  \coloneqq \begin{cases}
                    \x{i}{i}{k+1} - \x{i}{i}{k} & k\in \K^{i}\\
                    0 & k \notin \K^{i}. \label{eq:1005}
\end{cases}
\end{equation} 
We also form~$\s{}{k}{\ell} \coloneqq \big( \s{1}{k}{\ell}^T, \dots, \s{N}{k}{\ell}^T \big)^T \in \R^n$.
For all~$\tz \in \Ts$ and for all~$k$ such that~$\eta_{z -1} \leq k \leq \eta_z$, we define 
\begin{align}
    \at{k}{z} &\coloneqq \g{\x{}{}{k}}{\tz} - \g{\x{*}{k}{\tz}}{\tz} \label{eq:1003}\\
    \bt{k} &\coloneqq \sum_{\tau=k-B }^{k-1} \|\s{}{\tau}{1}\|^2, \label{eq:1032} 
\end{align}
where\footnote{
The output of~$\argmin_{\x{*}{}{} \in \X^*(\tz)} \| \x{*}{}{} - \x{}{}{k}\|$
may contain points with different costs, and for~$\alpha$
the point with the lowest cost should be used. Below, we will 
evaluate~$\alpha$ at times 
at which each point in
$\argmin_{\x{*}{}{} \in \X^*(\tz)} \| \x{*}{}{} - \x{}{}{k}\|$ has the same cost,
and thus there will be no ambiguity. 
} 
$\x{*}{k}{\tz} \in \argmin_{\x{*}{}{} \in \X^*(\tz)} \| \x{*}{}{} - \x{}{}{k}\|$. 
The term~$\at{k}{z}$ is the suboptimality gap between the cost of the true state and the 
locally optimal cost, 
and~$\bt{k}$ is the sum of the magnitudes of the last~$B$ updates before iteration~$k$. 

\subsection{Convergence of Algorithm~\ref{alg:myAlg}} \label{ss:phases}
\gb{
After an agent samples the objective, there are four
behaviors that occur. First, the cost~$\g{\x{}{}{k}}{\tz}$
can increase for~$k \in \{\eta_{z-1}, \dots, \eta_{z-1} + B\}$. Delays of up to length~$B$ mean that agents may still communicate and 
compute using iterates from up to~$B$ timesteps before the objective 
was sampled, and those old iterates may not help minimize the current
objective function. 
%This behavior is analyzed in Lemma~\ref{lem:phase1}. 
Second, 
there exists a time~$\hat{k}_z$ such that 
the cost~$\g{\x{}{}{k}}{\tz}$ must decrease
for~$k \in  \{ \eta_{z-1} + B, \dots,  \eta_{z-1} + \hat{k}_z \}$. 
%This behavior is analyzed in Lemma~\ref{lem:designGamma}. 
Third, after~$\hat{k}_z$ the value of~$\g{\x{}{}{k}}{\tz}$
decreases with linear rate toward~$\g{\x{*}{k}{\tz}}{\tz}$
until~$k$ reaches~$\eta_z$. 
%This behavior is analyzed in Lemmas~\ref{lem:kHat} and~\ref{lem:phase3}. 
Fourth, there is a potential abrupt increase in cost when an agent samples
the objective of Problem~\ref{prob:discrete} because agents' iterates
are constant while their objective changes. 
%This behavior is analyzed in Lemma~\ref{lem:jump}. 
Figure~\ref{fig:behaviors} depicts this process. 
}

Below we use the constants
\begin{align}       
    \Dz &= \frac{ 2 - \gammaz \Lx ( 1 + B  + N B ) }{2}, \\
    \Ez &= \frac{\Lx N B}{2}, \\
    \Fz &=\frac{N \Lx^2}{2} \big( N ( 7 \Lx^2 + 6 \Lx + 3 ) + 3 \big) + \frac{ 3 }{2}   \Big(  BN ( 6 \Lx^4  + 12 \Lx^3  + 14 \Lx^2 )  + 3\Lx^2   + 6\Lx + 7  \\
    &\qquad\qquad\qquad\qquad\qquad\qquad\qquad\qquad\qquad+  N \lambda_z^2 ( \Lx^2 ( BN (6 \Lx^2  + 8 ) +  3 )  + 4  ) \Big) + 1 \\
    \Gz &=  \frac{N \Lx^2}{2} \big( N ( 7 \Lx^2 + 6 \Lx  + 3 ) + 3 \big) + \frac{\Lx B N}{2} + N B \Lx^2 \big(  9 \Lx^2 + 18 \Lx + 21   + N \lambda_z^2 (9 \Lx^2  + 12  ) \big),  \\
         a_0 &= \max \left\{ \Lg{0} \dx, 8 E_0 \Big( \frac{G_0}{F_0} + \frac{E_0}{D_0} \Big) F_0 B \dx^{2} D_0^{-1} \right\},  \label{eq:3037}\\
     b_0 &= \frac{D_0}{8 E_0 \Big( \frac{G_0}{F_0} + \frac{E_0}{D_0} \Big) F_0} a_0, \\
     \az &=  a_{z-1} \pt{r_{z-1} -1}{z-1} + 2 L_t \Delta + B \dx  \Mx  + \Lg{z} \sigma_{z} + 4 \Lx B^2 \dx^2 \Ez\left(\frac{\Gz}{\Fz} + \frac{\Ez}{\Dz}\right)\Fz \Dz^{-1} \\
     &\qquad\qquad\qquad\qquad\qquad\qquad\qquad\qquad\qquad\qquad\qquad\qquad\qquad\qquad\textnormal{ for } \tz 
    \in \Ts \backslash \{0\} \\
    \bz &= B \dx^2 \textnormal{ for } \tz  \in \Ts \backslash \{0\} \\
         \pt{}{z} &= 1 - \gammaz \cz \in (0,1) \textnormal{ for } \tz \in \Ts, \\
    \cz &= \frac{\Dz}{2 \Fz + 2 \Dz}, \\
         \gb{K_z} &= \gb{2 L_t \Delta + B \dx  \Mx  + \Lg{z} \sigma_{z}}  + 4 B^2  \Lx  \dx^2 \Ez\left(\frac{\Gz}{\Fz} + \frac{\Ez}{\Dz}\right)\Fz \Dz^{-1}.
        \label{eq:3038} 
\end{align}

\begin{lemma} \label{lem:phase3}
    Let {Assumptions~\ref{ass:contTime} -~\ref{ass:partialAsynch}} hold with~$\kappa_z = \kz + r_z B$ for some~$r_z \in \N_0$. For each~$\tz \in \Ts$, set
    \begin{multline}
        \gamma_{\max,z} \coloneqq 
        \min \Bigg\{ \frac{1}{2}, \frac{2}{ \Lx  ( 1 + B + 2 N B) }, 
                    \left( \frac{\Gz}{\Fz} + \frac{\Ez}{\Dz} \right)^{-1}, \\
                     \frac{ \frac{\az}{\bz} + 2\Ez + \Dz \cz - \sqrt{\big( \frac{\az}{\bz} + 2\Ez + \Dz \cz\big)^2 - 4 \Dz \Ez \cz}}{2 \Ez \cz},
                     \\
                    \frac{\Dz}{\Ez}, 
                    \frac{1}{2 \cz},
                    \frac{\Dz}{8 \Fz \Big( \frac{\Gz}{\Fz} + \frac{\Ez}{\Dz} \Big) \cz},
                    \frac{2}{3 \Lx B N + \Lx B}
                     \Bigg\}. \label{eq:gammaMax}
    \end{multline}
    Then for~$\gammaz \in (0,\gamma_{\max,z})$ and~$\x{}{}{k}$ from~\eqref{eq:1001}, the sequence~$\{ \x{}{}{k} \}_{k \in  \{ \eta_{z-1} + \kz,\dots,\eta_{z} \}}$ generated by~$N$ agents executing Algorithm~\ref{alg:myAlg} obeys
    $\at{\eta_{z-1} + \kz + \rz B}{z} \leq \az \pt{\rz-1}{z}$. 
    %and $\bt{\eta_{z-1} + \kz + \rz B} \leq \bz \pt{\rz-1}{z}$.     
\end{lemma}

\begin{proof}
    See Appendix~\ref{lem:linearConverge}.
\end{proof}

\begin{figure}[!t] 
    \centering
    \begin{tikzpicture}
    \begin{axis}[
        xlabel={$k$},
        ylabel={$\g{x(k)}{\tz}$},
        axis x line=bottom,      % Draw x-axis at the bottom
        axis y line=left,        % No y-axis
        width=9cm,              % Adjust the width of the plot
        height=6cm,              % Adjust the height of the plot
        x axis line style={->},   % Arrow style for x-axis
        x label style={at={(axis description cs:1.03,0.04)},anchor=north},
        y axis line style={-},
        ytick=\empty,
        xtick=\empty,
        extra x ticks={1, 2, 3, 4, 5,6,7},
        extra x tick labels={$\eta_{z-1}$, $\eta_{z-1}+B$, $\eta_{z-1} + \kz$, $\eta_{z}$, $\eta_{z}+B$, $\eta_{z} + \hat{k}_{z+1}$, $\eta_{z+1}$},
       extra x tick style={
           grid=major,
           tick label style={rotate=315, anchor = north west } % <-- this is added
           }
    ]   
    \addplot[color = blue, ultra thick] coordinates {(1, 0) (2, 0.25)}; %\addlegendentry{Lemma~\ref{lem:phase1}}
    \addplot[color = red, ultra thick] coordinates {(2, 0.25) (3,0)}; %\addlegendentry{Lemma~\ref{lem:designGamma}}
    \addplot[domain=3:4, green,  ultra thick,opacity=1] {1/(x-2)-1}; %\addlegendentry{Lemma~\ref{lem:kHat},\ref{lem:phase3}}
    \addplot[color = purple, ultra thick,dotted] coordinates {(4,-0.5) (4,0.5)}; %\addlegendentry{Lemma~\ref{lem:jump}}
    \addplot[color = blue, ultra thick]
            coordinates {(4, 0.5) (5, 0.75)};
    \addplot[color = red, ultra thick]
            coordinates {(5, 0.75) (6,0)};
    \addplot[domain=6:7, green,  ultra thick,opacity=1] {1/(x-5.25)-1.33};
    \addplot[color = purple, ultra thick,dotted]
            coordinates {(7,-0.75) (7,0.0)};
    \addplot[domain=0:10, blue, thick,opacity=0] {sin(deg(x))};
    % \addlegendentry{$f(x) = \sin(x)$}

     \addplot[color = blue, ultra thick]
             coordinates {(1, 0) (2, 0.25)}; 
    \addplot[color = red, ultra thick]
             coordinates {(2, 0.25) (3,0)}; 
     \addplot[domain=3:4, green,  ultra thick,opacity=1] {1/(x-2)-1}; 
     \addplot[color = purple, ultra thick,dotted]
             coordinates {(4,-0.5) (4,0.5)}; 
     \addplot[color = blue, ultra thick]
             coordinates {(4, 0.5) (5, 0.75)};
     \addplot[color = red, ultra thick]
             coordinates {(5, 0.75) (6,0)};
     \addplot[domain=6:7, green,  ultra thick,opacity=1] {1/(x-5.25)-1.33};
     \addplot[color = purple, ultra thick,dotted]
             coordinates {(7,-0.75) (7,0.0)};
     \addplot[domain=0:10, blue, thick,opacity=0] {sin(deg(x))};

    \end{axis}
    
    \begin{axis}[
        axis x line=top,      % Draw x-axis at the bottom
        axis y line=none,        % No y-axis
        width=9cm,              % Adjust the width of the plot
        height=6cm,  % Adjust the height of the plot
        ytick=\empty,
        axis line style={->},   % Arrow style for x-axis
        xtick={1, 4, 7},
        xticklabels={$t_{z}$, $t_{z+1}$, $t_{z+2}$}
    ]

        \addplot[domain=0:10, blue, thick,opacity=0] {sin(deg(x))};
        % \addplot[color = blue, ultra thick]
        %     coordinates {(0, -1) (1, 0)};
        % \addlegendentry{$f(x) = \sin(x)$}
    
    \end{axis}
\end{tikzpicture}
    \vspace{-3mm}
    \caption{
    \gb{The four behaviors that Algorithm~\ref{alg:myAlg} can exhibit. 
    First, the blue lines depict the possible increase
    in the objective that can occur for~$k$
    from~$\eta_{z-1}$ to~$\eta_{z-1}+B$. 
    Second, the red lines depict 
    the guaranteed decrease in~$\g{\cdot}{\tz}$ for~$k$ 
    from~$\eta_{z-1} + B$ to~$\eta_{z-1} + \hat{k}_z$. 
    Third, the green curves depict the guaranteed linear convergence of~$\at{k}{z}$ for~$k$ 
    from~$\eta_{z-1} + \hat{k}_z$ to~$\eta_z$. 
    Fourth, the dotted purple lines show the potential increase in cost when~$\tz$ 
    is incremented to~$\tzz$.
    }
    }
    \label{fig:behaviors}
\end{figure}

Next we quantify Algorithm~\ref{alg:myAlg}'s performance
on Problem~\ref{prob:aggregate}. 

\begin{theorem} \label{thm:1}
    \gb{Let {Assumptions~\ref{ass:contTime}-\ref{ass:partialAsynch}} hold with~$\kappa_z = \kz + r_z B$ for some~$r_z \in \N_0$ and
    consider~$N$ agents executing Algorithm~\ref{alg:myAlg}. 
    Then for~$\tz \in \Ts$ and~$\rz \in \N_0$, 
    with stepsize~$\gamma_z \in (0, \gamma_{\max,z})$, 
    \begin{equation} \label{eq:thm2}
        \at{\eta_z}{z} \leq a_0 \prod^{z}_{i=0} \pt{r_i -1}{i} + \sum^{z}_{j=1} K_j \prod^{z}_{k=j} \pt{r_k - 1}{k}. 
    \end{equation}
    %where~$a_0$ is from~\eqref{eq:3037},~$\pt{}{z}$ is from~\eqref{eq:3038}, and~$\at{\cdot}{z}$ is from~\eqref{eq:1003}.
    }
\end{theorem}
\begin{proof}
   From Lemma~\ref{lem:phase3}, for~$t_0$ and~$k=\kzero + r_0 B = \eta_0$ we have~$\at{\eta_0}{0}  \leq a_0 \pt{r_0-1}{0}$.
    % By Assumption~\ref{ass:timescale}, $\eta_0 = \kzero + r_0 B$ so we have~$\at{\eta_0}{0}  \leq a_0 \pt{r_0-1}{0}$.
    Then for~$t_1$ and~$k=\eta_0+\kone+r_1 B = \eta_1$,  we have~$\at{\eta_1}{1} \leq a_1 \pt{r_1-1}{1}$.
    From the definition of~$a_1$, we reach
    %\begin{equation}
        $\at{\eta_1}{1} %& \leq \left(a_0 \pt{r_0 -1}{0} + K_1 \right) \pt{r_1-1}{1} \\
         \leq a_0 \pt{r_0 -1}{0} \pt{r_1-1}{1} + K_1 \pt{r_1-1}{1}$. \label{eq:3036}
    %\end{equation}
    %where we define~$K_z \coloneqq 2 L_t \Delta + B \dx  \Mx  + \Lg{z} \sigma_{z} +  \frac{4 \Lx B^2 \dx^2 \Ez\left(\frac{\Gz}{\Fz} + \frac{\Ez}{\Dz}\right)\Fz}{\Dz}$. 
    Iterating this bound gives the result. 
    %This bound can be expanded to any~$\tz \in \Ts$ to yield the result.     
\end{proof}

Theorem~\ref{thm:1} shows the long-run dependence of
sub-optimality on the amount of time that agents have optimized
each objective function, which for~$g(\cdot; \tz)$
is encoded in~$r_{z}$. 
It incorporates the temporal variation of agents' objective through~$L_t$, and it incorporates agents' sampling through~$\Delta$. 

\begin{remark} \label{rem:iss1}
    \gb{Theorem~\ref{thm:1} is essentially a uniform ultimate boundedness (UUB) result. The first term on the right-hand side of~\eqref{eq:thm2} goes to zero as~$z$ increases since~$\rho_i \in (0,1)$ for all~$i \in \N_0$, and the second term can be bounded by a constant.
    Theorem~\ref{thm:2} below is a UUB result for similar reasons.}
\end{remark}

\subsection{Relating Problem~\ref{prob:discrete} and Problem~\ref{prob:aggregate}}
%In Theorem~\ref{thm:1} we showed that agents can track the minimizers of Problem~\ref{prob:aggregate}, which comes from the asynchronous sampling of Problem~\ref{prob:discrete}, with bounded error.

We next seek to bound how closely agents track the minimizer of Problem~\ref{prob:discrete} by 
tracking the minimizers of Problem~\ref{prob:aggregate},
which is our main result. 
First, we define the constants~$K_1 (n, \nu_\X, r_\X, d_\X ) = \frac{ \varphi^2 \pi^{n/2} }{ n 2^{n+3} \Gamma \left( \frac{n}{2} \right)  } \nu_\X \bigg( \frac{r_\X}{d_\X} \bigg)^n$
and
\gb{
    \begin{equation}
        K_2 \left( f,g,h \right)  = \max_{\tz \in \Ts} \max \Big\{ \| \f{\cdot}{\tz} - \h{\cdot}{\tz} \|_{L^2}^2, 
         \| \f{\cdot}{\tz} - \g{\cdot}{\tz} \|_{L^2}^2 \Big\}
    \end{equation}
    }    
    for~$f, g, h : \R^n \times \R_{+} \to \R$.
    We begin with the following, which
    is based on~\cite[Proposition A.2]{nozari2016differentially}  but derives
    a new bound that does not depend on a
    function inverse. 

\begin{lemma} \label{lem:argmin}
    For all~$\tz \in \Ts$ and any two functions~$f,h \in \mathcal{S}$,
    \begin{equation}
       \Big\| \underset{x\in \X}{\argmin} \ \f{x}{\tz} - \underset{x\in \X}{\argmin} \ \h{x}{\tz} \Big\| \leq
        \left( \frac{ 4 \bar{u}   }{ \varphi } \right)^{\frac{n}{2n+4}} 
       \left( \frac{ \| \f{\cdot}{\tz} - \h{\cdot}{\tz} \|_{L^2}^2  }{ K_1 \left(n, \nu_\X, r_\X, d_\X  \right) } \right)^{\frac{1}{2n+4}},
    \end{equation}
    where~$\nu_\X \in (0,1)$,~$r_\X > 0$ is the radius of the largest ball contained in~$\X$,~$d_\X$ is the diameter of~$\X$,~$\varphi >0$,
    and~$\bar{u}>0$. 
    %\begin{equation}
    %    K_1 (n, \nu_\X, r_\X, d_\X ) = \frac{ \varphi^2 \pi^{n/2} }{ n 2^{n+3} \Gamma \left( \frac{n}{2} \right)  } \nu_\X \bigg( \frac{r_\X}{d_\X} \bigg)^n.
    %\end{equation}
\end{lemma}
\begin{proof}
    Let~$a = \argmin_{x \in \X} \f{x}{\tz}$ and $b = \argmin_{x \in \X} \h{x}{\tz}$.
    If~$a = b$, then the bound is trivial, so consider~$a \neq b$. 
    Let~$m_a = \f{a}{\tz}$,~$m_b=\h{b}{\tz}$,~$m = m_a - m_b$,~$u_a = \nabla_x \f{a}{\tz}$, and~$u_b = \nabla_x \h{b}{\tz}$. Without loss of generality take~$m \geq 0$. 
    Then~\cite[Prop. A.2]{nozari2016differentially} gives
    \begin{multline}
        \| \chi_{\tz} \|_{L^2}^2  \geq 
        \frac{ \varphi^2 \pi^{n/2} }{ n 2^{n+3} \Gamma \left( \frac{n}{2} \right)  } \nu_\X \bigg( \frac{r_\X}{d_\X} \bigg)^n  \| a - b \|^4 
         \left( \frac{ \varphi \| a - b \|^2 }{  2 \sqrt{ \varphi \psi \| a - b \|^2 + 2 ( \psi  + \varphi ) \bar{u} \| a - b \| + 4 \bar{u}^2  }  } \right)^n,
    \end{multline}
    where~$\chi_{\tz} \coloneqq \f{\cdot}{\tz} - \h{\cdot}{\tz}$. 
    We now extend that result to 
    derive an explicit bound on~$\| a - b \|$. 
    We write~$K_1$ for~$K_1(n, \nu_{\X}, r_{\X}, d_{\X})$. 
    %First, we define
    %$K_1 = \frac{ \varphi^2 \pi^{n/2} }{ n 2^{n+3} \Gamma ( \frac{n}{2} )  } \nu_\X \big( \frac{r_\X}{d_\X} \big)^n$.
    Since~$\varphi < \psi$ by definition of~$\mathcal{S}$, we have    
    %\begin{equation}
        $\| \chi \|_{L^2}^2  \geq 
        K_1    \| a - b \|^4  
        \left( \frac{ \varphi \| a - b \|^2 }{  2 \big(\psi \| a - b \| + 2 \bar{u}\big)  } \right)^n$, 
    %\end{equation}
    where we have also used
    $\psi^2 \| a - b \|^2 + 4 \psi  \bar{u} \| a - b \| + 4 \bar{u}^2 = \left( \psi \| a - b \| + 2 \bar{u} \right)^2$. 
    Taking the~$n^{th}$ root yields
    \begin{equation}    
     \| a - b \|^{2 + \frac{4}{n}} - \frac{ 2 \psi \sqrt[n]{\| \chi \|_{L^2}^2 }  }{ \varphi \sqrt[n]{K_1}}  \| a - b \| 
     - \frac{ 4 \bar{u} \sqrt[n]{\| \chi \|_{L^2}^2 }  }{ \varphi \sqrt[n]{K_1}} \leq 0. \label{eq:1021}
    \end{equation}
    A sufficient condition for~\eqref{eq:1021} is 
    %\begin{align}
    %\begin{equation}
         $\| a - b \|^{2 + \frac{4}{n}} - \frac{ 4 \bar{u} \sqrt[n]{\| \chi \|_{L^2}^2 }  }{ \varphi \sqrt[n]{K_1}} 
         \leq 0$.
    %\end{equation}
    %\end{align}
    Solving for~$\| a - b \|$ gives the result. 
    %\begin{multline}
         %\| a - b \|^{2 + \frac{4}{n}}   \leq \frac{ 4 \bar{u} \sqrt[n]{\| \f{\cdot}{\tz} - \h{\cdot}{\tz} \|_{L^2}^2 }  }{ \varphi \sqrt[n]{K_1}} \\
         %\| a - b \| \leq \left( \frac{ 4 \bar{u} \sqrt[n]{\| \f{\cdot}{\tz} - \h{\cdot}{\tz} \|_{L^2}^2 }  }{ \varphi \sqrt[n]{K_1}} \right)^{\frac{n}{2n+4}} \\
        %\| a - b \| \leq \left( \frac{ 4 \bar{u}   }{ \varphi } \right)^{\frac{n}{2n+4}} \left( \frac{ \| \f{\cdot}{\tz} - \h{\cdot}{\tz} \|_{L^2}^2  }{  K_1 } \right)^{\frac{1}{2n+4}}.
    %\end{multline}
    %Replacing~$a = \argmin_{x \in \X} \f{x}{\tz}$ and $b = \argmin_{x \in \X} \h{x}{\tz}$ yields the result.
\end{proof}

The main result of this article is as follows. 
    
\begin{theorem} \label{thm:2}
\gb{Let {Assumptions~\ref{ass:contTime}-\ref{ass:partialAsynch}} hold with~$\kappa_z = \kz + r_z B$ for some~$r_z \in \N_0$, 
    consider~$N$ agents executing Algorithm~\ref{alg:myAlg}, 
    and define~$h(x; t_z) = \frac{1}{2}\|x\|^2 + M_{g, z}$.
    For~$\tz \in \Ts$, with stepsize~$\gamma_z \in (0, \gamma_{\max, z})$, we have
    \begin{equation}
       | \g{\x{}{}{\eta_z}}{\tz} - f^*(\tz) | \leq \az \pt{\rz-1}{z} 
       +\left( \frac{ 4 \bar{u}   }{ \varphi } \right)^{\frac{n}{2n+4}}\left( \frac{ K_2 \left( f,g,h \right) }{ K_1 \left(n, \nu_\X, r_\X, d_\X  \right) } \right)^{\frac{1}{2n+4}}. \label{eq:1024}
    \end{equation}
}
\end{theorem}
\begin{proof}
    %Let~$a = \argmin_{x \in \X} \f{x}{\tz}$ and $b = \argmin_{x \in \X} \h{x}{\tz}$ where~$a \neq b$. 
%Let~$u_a = \nabla_x \f{a}{\tz}$. 
Define~$a$, $b$, $m_a$, $m_b$, $m$, and~$u_a$ as in Lemma~\ref{lem:argmin}. 
For each~$\tz \in \Ts$ 
and all~$k \in \{\eta_{z-1} + \hat{k}_z, \ldots, \eta_z\}$ we have 
    \begin{equation}
        | \g{\x{}{}{\eta_z}}{\tz} - f^*(\tz) | 
        = | \g{\x{}{}{\eta_z}}{\tz} - \g{\x{*}{k}{\tz}}{\tz} 
        + \g{\x{*}{k}{\tz}}{\tz} - f^*(\tz) |.
    \end{equation}
    Applying the triangle inequality and Lemma~\ref{lem:phase3} gives
    %yields
    %\begin{multline}
    %    | \g{\x{}{}{\eta_z}}{\tz} - f^*(\tz) |
    %    \leq | \g{\x{}{}{\eta_z}}{\tz} - \g{\x{*}{k}{\tz}}{\tz} | \\
    %    + | \g{\x{*}{k}{\tz}}{\tz} - f^*(\tz) |. \label{eq:1031}
    %\end{multline}
    %Since~$\eta_{z} \in \{\eta_{z-1} + \kz,\dots,\eta_{z}   \}$ we can apply the result of Lemma~\ref{lem:phase3} which yields
    \begin{equation}
        | \g{\x{}{}{\eta_z}}{\tz} - f^*(\tz) |
        \leq \az{} \rhoz^{\rz-1} + | \g{\x{*}{k}{\tz}}{\tz} - f^*(\tz) |. \label{eq:1020}
    \end{equation}
    \textbf{CASE 1: } For~$\g{\cdot}{\tz} \in \mathcal{S}$
    Lemma~\ref{lem:argmin} gives the result. \\
    %\begin{multline}
    %    | \g{x(\eta_z)}{\tz} - f^*(\tz) | 
    %     \leq \az \rhoz^{\rz-1} \\
    %     + \left( \frac{ 4 \bar{u}   }{ \varphi } \right)^{\frac{n}{2n+4}} \left( \frac{ \| \f{\cdot}{\tz} - \g{\cdot}{\tz} \|^2  }{ K_1 %\left(n, \nu_\X, r_\X, d_\X  \right) } \right)^{\frac{1}{2n+4}}. 
    %\end{multline}
    \textbf{CASE 2: } Suppose~$\g{\cdot}{\tz} \notin \mathcal{S}$ and~$ f^*(\tz) \geq \g{\x{*}{k}{\tz}}{\tz}$. For all~$\tz \in \Ts$, the function~$\g{\cdot}{\tz} \in C^2(\X)$ since it is quadratic and it has a bounded Hessian and gradient from Section~\ref{ss:p2objectives}. Therefore,~$\g{\cdot}{\tz} \in \D$.
    %In this case, we can still apply the result of Lemma~\ref{lem:argmin} to bound~$| \g{\x{*}{k}{\tz}}{\tz} - f^*(\tz) |$. 
    Note that 
    %\begin{equation}
        $m = m_a - m_b = f^*(\tz) - \g{\x{*}{k}{\tz}}{\tz}  \geq 0$,
    %\end{equation}
    which holds by hypothesis. Also, since $g \in \D$ 
    we have 
    %\begin{align}
        $\nabla^2_x \g{x}{\tz} \preceq L_x I_n$. 
    %\end{align}
    Due to the convexity of~$\f{\cdot}{\tz}$ we have
    %\begin{align}
        $u_a^T (b - a) \geq 0$,
    %\end{align}
    and all steps used to derive Lemma~\ref{lem:argmin} still hold.
    We apply Lemma~\ref{lem:argmin} to~\eqref{eq:1020} to achieve the same bound as Case~1. 
    %\begin{multline}
    %    | \g{x(\eta_z)}{\tz} - \f{\x{*}{f}{\tz}}{\tz} | 
    %     \leq \az \rhoz^{\rz-1} \\
    %     + \left( \frac{ 4 \bar{u}   }{ \varphi } \right)^{\frac{n}{2n+4}} \left( \frac{ \| \f{\cdot}{\tz} - \g{\cdot}{\tz} \|^2  }{ K_1 \left(n, \nu_\X, r_\X, d_\X  \right) } \right)^{\frac{1}{2n+4}}. 
    %\end{multline}
    
    \noindent \textbf{CASE 3: } Suppose~$\g{\cdot}{\tz} \notin \mathcal{S}$ and~$f^*(\tz) < \g{\x{*}{k}{\tz}}{\tz}$. We have~$\g{\cdot}{\tz} \in \D$ using the same argument from Case 2.
    By inspection, the function~$h(x; t_z) = \frac{1}{2}\|x\|^2 + M_{g, z}$
    satisfies~$h(x; \tz) \geq g(x; \tz)$ for each~$\tz \in \Ts$ and for all~$x \in \X$. Thus,
    %\begin{equation}
         $| \g{\x{*}{k}{\tz}}{\tz} - f^*(\tz) |
         \leq | \h{\x{*}{k}{\tz}}{\tz} - f^*(\tz) |$.
    %\end{equation}
    Since~$f,h \in \mathcal{S}$ for this choice of~$h$, we can apply Lemma~\ref{lem:argmin} to~\eqref{eq:1020}, which yields the same result as the prior cases. 
    %\begin{multline}
    %     | \g{x(\eta_z)}{\tz} - \f{\x{*}{f}{\tz}}{\tz} | 
    %     \leq \az \rhoz^{\rz-1} \\
    %     + \left( \frac{ 4 \bar{u}   }{ \varphi } \right)^{\frac{n}{2n+4}} \left( \frac{ \| \f{\cdot}{\tz} - \h{\cdot}{\tz} \|^2  }{ K_1 \left(n, \nu_\X, r_\X, d_\X  \right) } \right)^{\frac{1}{2n+4}} . 
    %\end{multline}
     %We exclude the case the case where~$\g{x(\eta_z)}{\tz} = f^*(\tz)$ since the proof is trivial. 
     %This concludes the proof since we have considered all possible cases for~$\g{\cdot}{\tz} \in \mathcal{S}$ and~$\g{\cdot}{\tz} \notin \mathcal{S}$. 
\end{proof}

%Theorem~\ref{thm:2} shows that agents track the solution of Problem~\ref{prob:discrete} with  bounded error. 
The first term in the Theorem~\ref{thm:2} bound accounts for agents' tracking error of Problem~\ref{prob:aggregate}, which can be reduced by 
computing and communicating more often. The second term is partly dictated by the distance between~$f$,~$g$, and~$h$, and it can be reduced by agents sampling Problem~\ref{prob:discrete} more often. 
%Thus, agents can track the minimizer of Problem~\ref{prob:discrete} more closely by completing more computations, communications, and objective samples, and Theorem~\ref{thm:2} makes precise how each individual operation contributes to reducing error. 

\subsection{Synchronization Corollaries}
Here we show how this work generalizes prior work. 

\begin{corollary} \label{cor:synch1}
        Synchronous sampling, i.e.,~$\T^i = \T^j$ for all~$i, j \in [N]$, gives
        %\begin{equation}
            $|\g{x(\eta_z)}{\tz} - f^*(\tz) | \leq \az \rhoz^{\rz-1}$  for all~${\tz \in \Ts}$. 
        %\end{equation}
\end{corollary}
\begin{proof}
   From~\eqref{eq:1020} we have
\begin{equation}
    | \g{\x{}{}{\eta_z}}{\tz} - f^*(\tz) |
    \leq \az{} \rhoz^{\rz-1}  + | \g{\x{*}{k}{\tz}}{\tz} - f^*(\tz) |. 
\end{equation}
Since~$\g{\cdot}{\tz} = \f{\cdot}{\tz}$ here, we have~$\g{\x{*}{k}{\tz}}{\tz} - f^*(\tz) = 0$ for all~$\tz \in \Ts$, which gives the result. 
\end{proof}

    Corollary~\ref{cor:synch1} shows that if agents sample Problem~\ref{prob:discrete} at 
    the same times, then we recover (up to constants) the linear convergence rate for distributed 
    asynchronous minimization of
    time-varying strongly convex functions from~\cite{behrendt2023totally}. 

\begin{corollary} \label{cor:synch2}
    If sampling, computations, and communications are synchronous, i.e.,~$\T^i = \T^j$ for all~$i, j \in [N]$ and $B=1$, then 
        %\begin{equation}
            $|\g{x(\eta_z)}{\tz} - f^*(\tz) | \leq \az \rhoz^{\kappa_z -1}$
        %\end{equation}
        for all~$\tz \in \Ts$. 
\end{corollary}
\begin{proof}
   For all~$\tz \in \Ts$ and~$k \in \{ \eta_{z-1},\dots,\eta_{z} \}$, we have~$\g{\cdot}{\tz} = \f{\cdot}{\tz} \in \mathcal{S}$. For all~$k$, we have~$\x{*}{k}{\tz} = \x{*}{k+1}{\tz}$ from strong convexity of~$\g{\cdot}{\tz}$. Then~$\kz = 0$ for all~$\tz \in \Ts$. Using~$B=1$, Lemma~\ref{lem:phase3} then gives
    %\begin{align}
        $\at{ \eta_{z-1} + \rz }{z} \leq \az \pt{\rz-1}{z}$
    %\end{align}
    for~$\rz \in \{0,\dots, \kappa_z \}$ and for all~$\tz \in \Ts$. 
    Setting~$\rz = \kappa_z $ yields the result. 
\end{proof}

    Corollary~\ref{cor:synch2} shows that if agents 
    sample Problem~\ref{prob:discrete} at the same times
    and compute and communicate at all times, 
    then we recover (up to constants) 
    the linear convergence rate for centralized gradient descent on 
    time-varying
    strongly convex functions from~\cite{simonetto2020time}.

\section{Numerical Results} \label{sec:simulation}
This section presents two examples using Algorithm~\ref{alg:myAlg}. 
%First, we consider a general time-varying quadratic program. Second, we consider using a network of agents to track a time-varying reference trajectory. 

\subsection{Time-Varying Quadratic Program} \label{ss:tvqp}
We consider~$N=10$ agents with block sizes~$n_i=2$ for all~$i \in [N]$ that asynchronously track the solution of the following problem for all~$t \in [0,50]$s: 
\begin{equation} \label{eq:simProb}
    \underset{x \in \X \subset \R^{20}}{\min} \ \frac{1}{2} x^T \Q{}{t} x + q(t)^T x,
\end{equation}
where~$\Q{}{t} = Q + I_{20} \cos (\omega t)$,~$q(t) = 100 \sin (2 \omega t) \in \R^{20}$,~$\omega = 0.1$, and~$Q = Q^T \succ 0 
\in \R^{20 \times 20}$ is randomly generated.
We set~$\X = [-100, 100]^{20}$. 
%define the constraint set~$\X = \{ x \in \R^{20} | -100 \leq x_i \leq 100 \ \forall i \in {1,\dots,n}\} \subset \R^n$, where~$n=\sum^N_{i=1} n_i = 20$. 
Agents can sample the objective every~$t_s = 2$ seconds and thus~$\T_{all} \coloneqq \{0,2,4,\dots,50\}$. At 
every~$\ind{}{\ell} \in \T_{all}$, agent~$i$ samples the objective function with probability~$p_{i,s} (k) = 0.5$.
We allow~$\kappa_\ell = 500$ iterations to occur between changes in the objective function, and we set~$\gamma_\ell = 0.001$ and~$B=10$ for all~$\ind{}{\ell} \in \Ts$.
At every iteration~$k$, agent~$i$ computes an update with probability~$p_{i,u} (k) = 0.6$ and communicates its block of decision variables~$\x{i}{i}{}$ with probability~$p_{i,c} (k) =0.6$.
If an agent does not compute or communicate for~$B-1$ timesteps, then 
we make it do so at the~$B^{th}$ timestep to satisfy Assumption~\ref{ass:partialAsynch}.

In Figures~\ref{fig:compare} and~\ref{fig:multipleB}, Algorithm~\ref{alg:myAlg} generates
iterates on the timescale ``Iterations~$(k)$'' and the changes in objective functions occur on the timescale ``Time~$(s)$''. 
\gb{Abrupt increases in cost are due to the changes in the agents' objective function
and are caused by at least one agent taking a sample of the time-varying objective function.
Agents' objective function is constant between these samples, and their performance
in optimizing a static objective is illustrated between samples.}

%Figure~\ref{fig:trackingError} shows the tracking error of the true state of the network~$\x{}{}{k}$ relative to both the minimizer of Problem~\ref{prob:discrete}, namely~$\x{*}{f}{t}$, and a minimizer of Problem~\ref{prob:aggregate}, denoted~$\x{*}{}{t}$. As expected,~$\x{}{}{k}$ tracks~$\x{*}{}{t}$ more closely than~$\x{*}{f}{t}$ and the tracking error tends toward zero for all~$t \in [0,50]$s. A similar behavior holds for the tracking of~$\x{*}{f}{t}$ for all~$t \in [0,50]$s, though~$\| \x{*}{f}{t} - \x{}{}{k} \|$ does not strictly decrease between~$30-32$ seconds. This behavior occurs because agents track~$\x{*}{}{t}$ which may not equal~$\x{*}{f}{t}$.

Figure~\ref{fig:compare} compares Algorithm~\ref{alg:myAlg} with a distributed consensus  optimization algorithm~\cite{nedic2010convergence} on an example with~$N=2$ agents.
%Consensus optimization uses all agents to update all decision variables and average their iterates, while block coordinate descent uses only a single agent to update each decision variable, which makes them natural 
%counterparts.  
%To highlight their differences, we consider 
%that asynchronously track the solution of~\eqref{eq:simProb} for~$t \in [0,20]$s.
Figure~\ref{fig:compare} shows  
%Algorithm~\ref{alg:myAlg} tracks each minimizer more closely than the consensus optimization algorithm, and 
the sub-optimality incurred by Algorithm~\ref{alg:myAlg}
is nearly always less than the sub-optimality incurred by the algorithm from~\cite{nedic2010convergence}. 
\newnew{Table~\ref{tab:comparison} quantifies the differences in the convergence of these algorithms.
``Error'' is the distance between the optimum and each of these algorithms' true states. 
We compare their root-mean-square (RMS) error computed over all~$k \in \{0, \ldots, 5000\}$ 
and the average error each algorithm incurs
at the last tick of~$k$ before each time the objective is sampled. 
Algorithm~\ref{alg:myAlg} incurs~$30.16$\% less RMS error than
the algorithm from~\cite{nedic2010convergence}, and 
on average it incurs~$92.4$\% less error just before each objective
is sampled, which indicates superior performance.
}
Figure~\ref{fig:compare} also shows that Algorithm~\ref{alg:myAlg}'s performance
is not far from its synchronous counterpart. 
%We also compare Algorithm~\ref{alg:myAlg} to its synchronous form, which shows that it performs well even under asynchrony. 

\begin{figure}[!t]
    \centering
    \includegraphics[width = 0.48\textwidth]{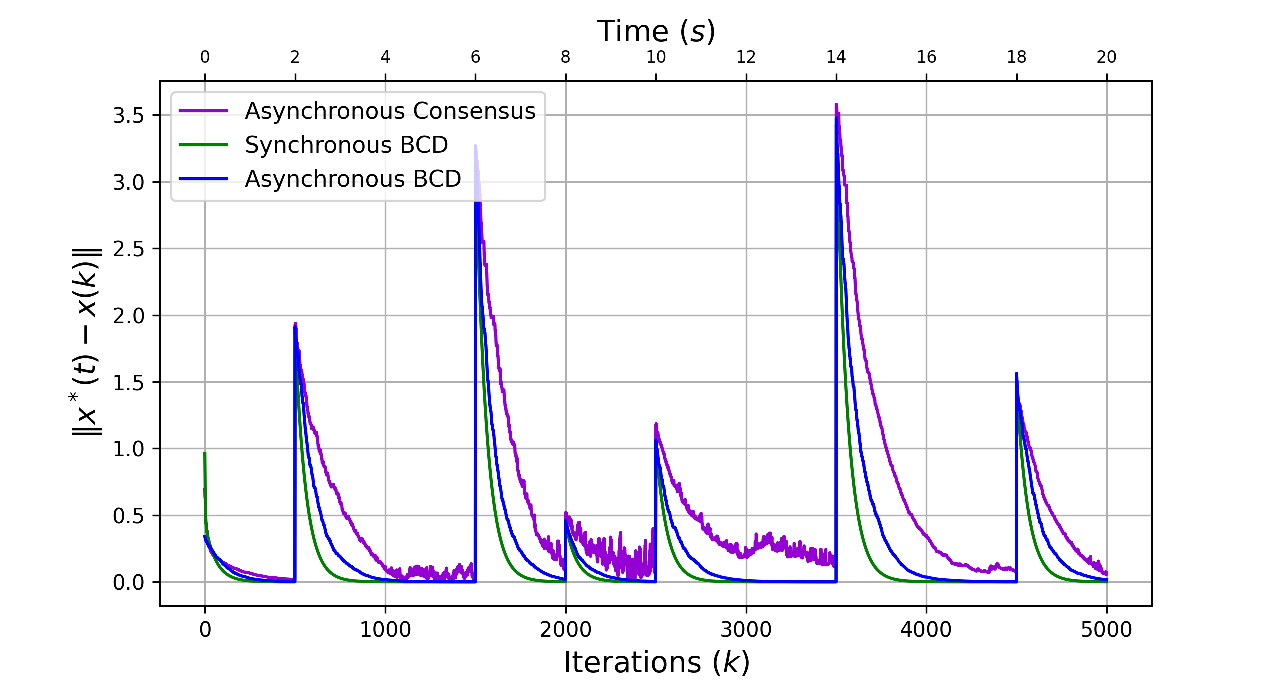}
    \vspace{-3mm}
    \caption{
    Tracking error between the true state of the network and the minimizer of~\eqref{eq:simProb},
    namely~$\| x^*(t) - x(k) \|$, for the asynchronous block coordinate descent (BCD) algorithm in Algorithm~\ref{alg:myAlg}, 
    the consensus optimization algorithm from~\cite{nedic2010convergence}, and synchronous BCD. 
    %Our asynchronous BCD algorithm is able to track the solution more closely than the asynchronous consensus optimization algorithm, and it performs closely to the synchronous BCD algorithm.}    
    }
    \label{fig:compare}
\end{figure}

%\begin{figure}[!t]
%    \centering
%    \includegraphics[width = 0.44 \textwidth]{Figures/trajectoryError.eps}
%    \vspace{-3mm}
%    \caption{Tracking error of agents' positions relative to the discretized reference trajectory~$\| z_i(k) - \bar{r}_d(t) \|$ (red), and asynchronously sampled reference trajectory~$\| z_i(k) - \bar{r}_{i}(t) \|$ (blue) for~$i \in [N]$.  
%    }
%    \label{fig:trajectoryError}
%\end{figure}

\begin{table}[!t]
\centering
\begin{tabular}{|c||c|c|c|} \hline
\diagbox{\newnew{Algorithm}}{\newnew{Type of Error}} & \newnew{RMS}        & \newnew{Avg. before sample}   \\ \hline\hline
\newnew{Algorithm~\ref{alg:myAlg}}        & \newnew{0.558} & \newnew{0.00983} \\ \hline
\newnew{From~\cite{nedic2010convergence}} & \newnew{0.799} & \newnew{0.130} \\ \hline
\end{tabular}
\caption{
\newnew{Comparing Algorithm~\ref{alg:myAlg} to the algorithm from~\cite{nedic2010convergence}.} 
}
\label{tab:comparison}
\end{table}

Lastly, we demonstrate the effect of the delay parameter~$B$ on the convergence of Algorithm~\ref{alg:myAlg}. 
We set~$p_{i,s}=1$ so that all agents sample the objective at each~$t_\ell \in \T_{all}$,
but we only have agents compute and communicate once every~$B$ timesteps. 
We set~$\gamma_\ell = 9 \times 10^{-4}$ for all~$t_\ell \in \Ts$.
Figure~\ref{fig:multipleB} shows~$\at{k}{z}$ for values of~$B=1,3,10,100$. 
As~$B$ increases, agents compute and communicate less, 
which slows the convergence of~$\x{}{}{k}$. 

%We expect that as~$B$ decreases, the value of~$\at{k}{}$ will approach that of the synchronous version of Algorithm~\ref{alg:myAlg}, which we indeed observe in Figure~\ref{fig:multipleB}.

%\begin{figure}[!htbp]
%    \centering
%    \includegraphics[width = 0.44 \textwidth]{Figures/alphaBeta.eps}
%    \vspace{-3mm}
%    \caption{Plot of~$\at{k}{}$ and~$\bt{k}$ 
%    for the problem in Section~\ref{ss:tvqp}. 
%    Abrupt increases occur when agents sample the objective function, and
%    decreases occur as agents execute more iterations of Algorithm~\ref{alg:myAlg}.}
%    \label{fig:alpha}
%\end{figure}

\begin{figure}[!t]
    \centering
    \includegraphics[width = 0.44 \textwidth]{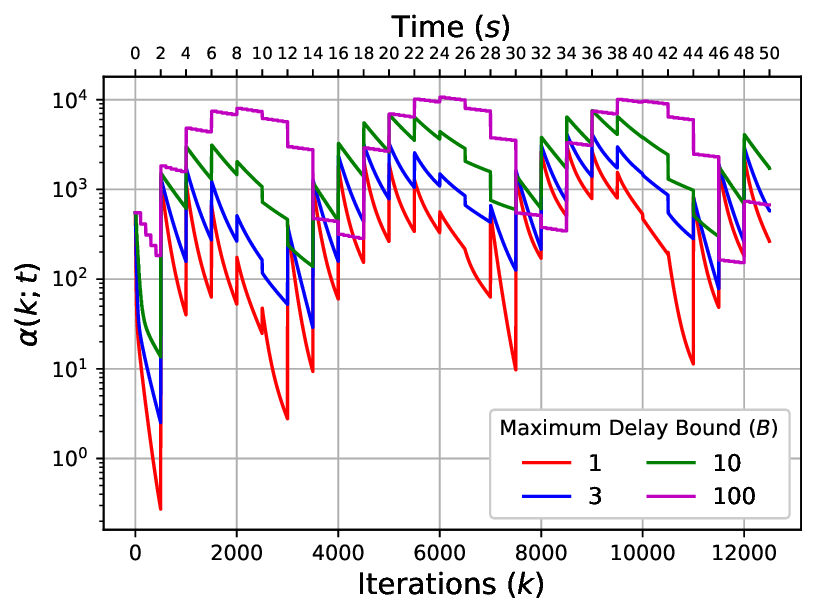}
    \vspace{-3mm}
    \caption{Plot of~$\alpha$ for the problem in Section~\ref{ss:tvqp} for varying~$B$. As~$B$ decreases, agents compute and communicate more often, thus shrinking the suboptimality gap~$\alpha$ and tracking~$x^*(t)$ more closely.}
    \label{fig:multipleB}
\end{figure}

\begin{figure}[!t]
    \centering
    \includegraphics[width = 0.44 \textwidth]{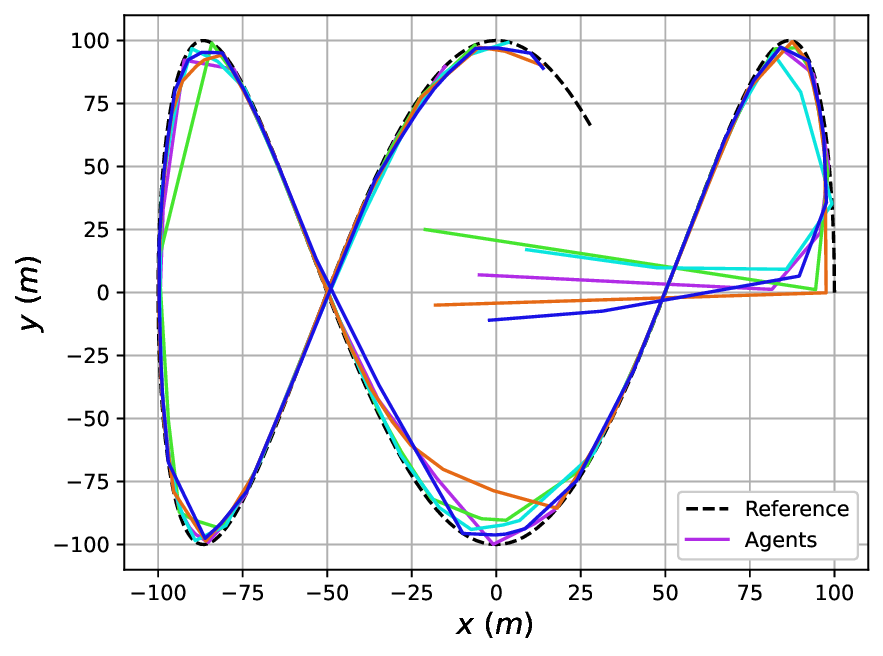}
    \vspace{-3mm}
    \caption{The positions~$z_i(t)$ for the problem in Section~\ref{ss:rt}. 
    %~$N=5$ agents executing Algorithm~\ref{alg:myAlg} and asynchronously sampling the reference trajectory~$\bar{r}(t)$ for all~$t \in [0,500]$s. 
    % The deviations of the agents' trajectories from the desired reference trajectory is due to the asynchrony in agents' computations, communications, and reference sampling. 
    Agents start far from the desired reference but closely track it thereafter. % in computations, communications, and reference sampling.
    Tracking error is due to agents possibly tracking an outdated sample of the reference, as well as asynchrony in their computations and communications.
    % \textcolor{red}{Do more to describe the figure, explain the big deviations from the reference, etc.}
    % \gb{done.}
    }
    \label{fig:trajectories}
\end{figure}

\subsection{Reference Tracking} \label{ss:rt}
In this simulation we consider~$N=5$ agents all tracking an \emph{a priori} unknown time-varying reference trajectory~$\bar{r}(t) \in \R^2$ for~$t \in [0,500]$s. 
% \red{Change~$r(t)$. It's the affine term in Problem~1.}
% \gb{done.}
We set~$ \bar{r}(t) = [ 100 \cos (\omega t), 100 \sin (3 \omega t) ]^T \in \R^2$ meters, 
where~$\omega = 0.01$ Hz. 
Agent~$i$ updates its position~$z_i = [x_i,y_i]^T \in \R^2$ for all~$i \in [N]$, and we concatenate the agents' states into~$z=[z_1^T,\dots,z_N^T]^T \in \R^{10}$.
For~$t \in [0,500]$s, agents track 
\begin{equation}
    \underset{z \in \Z \subset \R^{10}}{\argmin} \ \frac{1}{2} (z- \mathbbm{1}_N \otimes \bar{r}(t))^T Q (z - \mathbbm{1}_N \otimes \bar{r}(t)), 
\end{equation}
where~$Q = 10 \cdot I_{10}$ and~$\Z = [-200,200]^{10}$. %\{ z \in \R^{10} | -200 \leq z_i \leq 200 \ \forall i \in \{1,\dots,n\}\}$. 
Agents are capable of sampling the objective every~$t_s = 1$s, and therefore~$\T_{all} = \{0,1,\dots,500\}$. 
At each~$\ind{}{\ell} \in \T_{all}$, agent~$i$ samples~$\bar{r}(t)$ with probability~$p_{i,s} = 0.1$. 
At each iteration~$k$, agent~$i$ computes an update with probability~$p_{i,u} = 0.5$ and communicates its state 
with probability~$p_{i,c} = 0.5$.
As in Section~\ref{ss:tvqp}, agents that have not computed and/or communicated for~$B-1$ timesteps
are made to do so at the next timestep. 
We set~$\gamma_\ell = 3\times 10^{-3}$,~$\kappa_\ell = 10$ iterations, and~$B=5$ for all~$\ind{}{\ell} \in \Ts$. The agents' initial positions are randomly generated so that~$z_i(0) \in [-30,30] \times [-30,30]$ for all~$i \in [N]$.
Figure~\ref{fig:trajectories} shows the reference trajectory~$\bar{r}(t)$ and the agents' trajectories that result from running Algorithm~\ref{alg:myAlg}. Agents attain close tracking of their references despite asynchrony
in their operations, which illustrates the success of Algorithm~\ref{alg:myAlg} under asynchrony. 
%Figure~\ref{fig:trajectoryError} shows the tracking error between the agents' positions~$z_i (k)$ and (i) the discretized reference trajectory, which obeys~$\bar{r}_d(t_{\ell}) \coloneqq \bar{r}(t_{\ell}) \in \R^2$ for all~$t_{\ell} \in \Ts$, and (ii) the agents' asynchronously sampled values of~$\bar{r}_d(t_{\ell})$, which we denote as~$\bar{r}_i(t_{\ell}) \in \R^2$ for~$i \in [N]$. %  for all~$t \in \T$. 

\section{Conclusion} \label{sec:conclusion}
We presented the first distributed asynchronous time-varying optimization framework 
with asynchronous objective sampling and applied it to quadratic programs. 
\gb{Future extensions include wider classes of objectives and improving the possible conservativeness of the stepsize~$\gammaz$, perhaps
by extending methods for static optimization problems such as~\cite{feyzmahdavian2023asynchronous,ubl2022faster}.}

\appendix
%\begin{appendices}
\subsection{Technical Lemmas}
\begin{lemma}[\!\!{\cite[Lemma 3.1]{luo1992error}}] \label{lem:solSeparation}
    Let~$g$ be as in Problem~\ref{prob:aggregate}. 
    Then for each~$\tz \in \Ts$
    there exists~$\epsilon_z > 0$ such that for~$x^*, y^* \in \X^*(\tz)$ with~$\g{x^*}{\tz} \neq \g{y^*}{\tz}$, we have~$\|x^* - y^*\| \geq \epsilon_z$.
\end{lemma}
%\noindent \emph{Proof.} See~\cite[Lemma 3.1]{luo1992error}. \hfill $\blacksquare$

    Lemma~\ref{lem:solSeparation} shows that points in~$\X^*(\tz)$ 
    with different costs are strictly separated, which
    will be used in our later analyses. 

\begin{lemma} \label{lem:errorBound1}
    Let~$g$ be as in Problem~\ref{prob:aggregate}. For every~$\omega_z > 0$ and $\tz \in \Ts$, there exist~$\delta_z , \lambda_z > 0$ such that for all~$x \in \X$ with~$\g{x}{\tz} \leq \omega_z$ and~${\| x - \Pi_\X \big[ \x{}{}{} - \nabla_x \g{\x{}{}{}}{\tz} \big] \| \leq \delta_z}$, 
    \begin{equation}
        \min_{x^* \in \X^*(\tz)} \| x - x^* \| \leq \lambda_z \| x - \Pi_\X \left[ \x{}{}{} - \nabla_x \g{\x{}{}{}}{\tz} \right] \|. \label{eq:76}
    \end{equation}
\end{lemma}
\noindent \emph{Proof.}
    Problem~\ref{prob:aggregate} is quadratic over a polyhedral constraint set~$\X$, and
    this result holds via~\cite[Theorem 2.3]{luo1992error}.
\hfill $\blacksquare$ 

Lemma~\ref{lem:errorBound1} is the Error Bound Condition, and it will be used to analyze the nonconvex
objectives in Problem~\ref{prob:aggregate}. 

\begin{lemma} \label{lem:noGammaUpdate}
    For all~$\tz \in \Ts$, any~$\gammaz > 0$, and all~$x\in \X$,
    \begin{equation}
        \| x - \Pi_\X \big[ \x{}{}{} -   \nabla_x \g{\x{}{}{}}{\tz} \big] \| 
        \leq \max\{ 1,\gammaz^{-1} \} \| x - \Pi_\X \big[ \x{}{}{} -  \gammaz \nabla_x \g{\x{}{}{}}{\tz} \big] \|. \label{eq:77}
    \end{equation}
\end{lemma}
\noindent \emph{Proof.}
    This follows from~\cite[Lemma 3.1]{tseng1991rate}.
\hfill $\blacksquare$ 

Despite the nonconvexity of Problem~\ref{prob:aggregate}, 
we can use Lemmas~\ref{lem:errorBound1} and~\ref{lem:noGammaUpdate} because it is quadratic over~$\X$ for all~$\tz \in \Ts$. 

\subsection{Convergence Lemmas}
\gb{This section provides lemmas used in Appendices~\ref{app:towardlemma1proof} and~\ref{lem:linearConverge}. 
Lemmas~\ref{lem:descent2}-\ref{lem:outOfDate} are used to prove that the cost decreases over intervals of length~$B$ in Lemmas~\ref{lem:phase2} and~\ref{lem:designGamma}. 
Then Lemmas \ref{lem:update}-\ref{lem:time0converge} are used to derive a linear convergence rate of Algorithm~\ref{alg:myAlg} when solving Problem~\ref{prob:aggregate}
in Lemma~\ref{lem:phase3} and Theorem~\ref{thm:1}. 
%Lastly, Lemmas~\ref{lem:upperConvexity}-\ref{lem:balls} are used to quantify the tracking error between the iterates of Algorithm~\ref{alg:myAlg} and the minimizer of Problem~\ref{prob:discrete} in Theorem~\ref{thm:2}. 
}

%% Technical Lemmas
\begin{lemma} \label{lem:descent2}
If~${h:\R^n \rightarrow \R}$ is continuously differentiable and has 
the property $\Vert \nabla h(x) - \nabla h(y) \Vert \leq K \Vert x-y \Vert$ for every~${x,y \in \R^n}$, then~$h\left(x+y\right) \leq h\left(x\right)+y^{T}\nabla h\left(x\right)+\frac{K}{2}\left\Vert y\right\Vert^{2}$. 
\end{lemma}
\noindent \emph{Proof: } See~\cite[Proposition A.32]{bertsekas1989parallel}. \hfill $\blacksquare$

\begin{lemma} \label{lem:descent}
For all~$i \in [N]$,~$k \geq 0$, and~$\tz \in \Ts$, we have~$\s{i}{k}{\ell}^T \nabla_{\x{}{i}{}} \g{\x{i}{}{k}}{\tz} \leq  -\frac{1}{\gammaz} \|\s{i}{k}{\ell}\|^2.$ 
\end{lemma}
\noindent \emph{Proof: } Follows from~\cite[Sec. 7.5.4, Lemma 5.1]{bertsekas1989parallel}. 
%For~$v \in \X \subset \R^n$, $x \in \R^n \backslash \X$, and~$z=\Pi_\X \left[x\right]$, a property of orthogonal projections is~$0 \geq (z-x)^T(z-v)$. Define~$z \coloneqq \Pi_{\X_i}\big[ \x{i}{i}{k} - \gammaz \nabla_{\x{}{i}{}} \g{\x{i}{}{k}}{\tz} \big]$. Then we find that~$0 \geq \big(z -\big( \x{i}{i}{k} - \gammaz \nabla_{\x{}{i}{}} \g{\x{i}{}{k}}{\tz} \big) \big)^T \big(z-\x{i}{i}{k}\big)$. Expanding the inner product and combining like terms gives us~$0 \geq \|z-\x{i}{i}{k}\|^2 + \big( z - \x{i}{i}{k} \big)^T \gammaz \nabla_{\x{}{i}{}} \g{\x{i}{}{k}}{\tz}$. Using~$\s{i}{k}{\ell} = z - \x{i}{i}{k}$ and rearranging completes the proof.
\hfill $\blacksquare$

\begin{lemma} \label{lem:outOfDate}
For all~$\tz \in \Ts$, for all~$i \in [N]$, and for all~$k \geq 0$, we have~$\| \x{i}{}{k} - \x{}{}{k} \| \leq \sum^{k-1}_{\tau = k - B} \|\s{}{\tau}{\ell}\|$. 
\end{lemma}
\noindent \emph{Proof: } See~\cite[Section 7.5.1]{bertsekas1989parallel}. \hfill $\blacksquare$

\begin{lemma} \label{lem:update}
Define~$\chi(\tz)=2\gammaz  \Lx^2 B N$.
For each~$\tz \in \Ts$,~$\gammaz \in (0,1)$, and for all~$ k \in \{\eta_{z-1},\dots, \eta_{z}\}$,
\begin{equation} 
   \| \x{}{}{k+1} - \x{}{}{k} \|^2 \leq 
   \frac{ 1 + \chi(\tz) }{1-\gammaz }  \sum^{k+B-1}_{\tau=k} \| \s{}{\tau}{\ell} \|^2 
   + \frac{\chi(\tz)}{1-\gammaz}  \sum_{\tau=k-B}^{k-1} \| \s{}{\tau}{\ell}  \| ^2.
\end{equation}
\end{lemma}
\noindent \emph{Proof: } See~\cite[Lemma 4.2]{tseng1991rate}. \hfill $\blacksquare$

\begin{lemma} \label{lem:technical} 
For all~$\tz \in \Ts $ and any~$x \in \R^n$,~$\x{1}{}{},\dots,\x{N}{}{} \in \R^n$, and any~$\bar{x}\in\X$, there holds
\begin{multline} 
    \g{z}{\tz}-\g{\bar{x}}{\tz} \leq 
      \frac{\Lx^2}{2} \Big( N (7 \Lx^2  \gammaz^2 + 6 \Lx  \gammaz + 3 )
      + 3  \big) \sum^N_{i=1} \| x - \x{i}{}{} \|^{2} + \frac{N}{2} \big( 3\Lx^2 + 4  \Big) \|x - \bar{x} \|^{2} \ \\
      +\frac{1}{2\gammaz^2} \big( 3\Lx^2 \gammaz^2  + 4 \gammaz^2 + 6\Lx \gammaz 
      + 3 \big)  \| \bar{z} - x \|^{2}, 
\end{multline}
where~$\bar{z}\coloneqq \Pi_{\X}\big[x - \gammaz \nabla_x \g{x}{\tz} \big]$ and~$z$ is a vector with components~$z_i \coloneqq \Pi_{\X_i} \big[ \x{}{i}{} - \gammaz \nabla_{\x{}{i}{}} \g{\x{i}{}{}}{\tz} \big]$ for all~$i \in [N]$.
\end{lemma}
\noindent \emph{Proof: } See~\cite[Lemma 4.3]{tseng1991rate}. \hfill $\blacksquare$

\begin{lemma} \label{lem:together}
For all~$ \tz \in \Ts$ and~$k \in \{\eta_{z-1} + \kz,\dots,\eta_{z}-B \}$,
\begin{equation}
        \g{\x{}{}{k+B}}{\tz} - \g{\x{*}{k}{\tz}}{\tz} \leq \big( A_1  + A_2  
        + 1 \big) \sum^{k+B-1}_{\tau =k} \| \s{}{\tau}{\ell} \|^2
       + \big( A_1 + A_3 +A_4 \big)  \sum^{k-1}_{\tau =k-B} \| \s{}{\tau}{\ell} \|^2,
\end{equation}
where~$g(\x{*}{k}{\tz}; \tz)$ is 
the same for all~$k$ 
due to Lemma~\ref{lem:kHat}, 
and
\begin{align}
    A_1 &= \frac{N \Lx^2 }{2} \big( 7 \Lx^2 N \gammaz^2 + 6 \Lx N \gammaz + 3 N + 3  \big) \\
    A_2 &= \frac{ 1 }{2\gammaz^2 - 2\gammaz^3} \Big(  B N \big( 6 \Lx^4  + 8 \Lx^2  \big)  \gammaz^3 
    + \big( 12 \Lx^3 B N + 3\Lx^2 + 4  \big) \gammaz^2 \\  
    &\qquad\qquad\qquad+ \big( B \Lx^2 ( 6  N +  N^2  \lambda_z^2 (6 \Lx^2   +8 ) ) + 6\Lx \big) \gammaz + 3  + N \lambda_z^2 \big( 3\Lx^{2} + 4  \big) \Big) \\
    A_3 &= \frac{ 1 }{2\gammaz - 2\gammaz^2} B N \Lx^2 \Big( \big( 6 \Lx^2  + 8 \big) \gammaz^2  +    12 \Lx   \gammaz  + 6  + N \lambda_z^2 (6 \Lx^2   + 8 )  \Big) \\
    A_4 &= \frac{\Lx B N}{2}.
\end{align}
\end{lemma}
\noindent \emph{Proof.} Fix any~$k \in \{\eta_{z-1} + \kz,\dots,\eta_{z}-B \}$. For each~$i$, let~$k^i$ denote the smallest element of~$\K^i$ exceeding~$k$. Then we can write the equation~$\x{i}{i}{k^i+1} = \Pi_{\X_i} \big[ \x{i}{i}{k} - \gammaz \nabla_{\x{}{i}{}} \g{\x{i}{}{k^i}}{\tz}  \big]$.
Next we apply Lemma~\ref{lem:technical} with~$x=\x{}{}{k}$,~$x^i=\x{i}{}{k^i}$ for each~$i \in [N]$,~$\bar{z}(k)=\Pi_{\X} \big[ \x{}{}{k} - \gammaz \nabla_{x} \g{\x{}{}{k}}{\tz}  \big]$, and~$\bar{x}=\x{*}{k}{\tz}$
for any~$\x{*}{k}{\tz} \in \argmin_{x^* \in \X^*(\tz)} \|x(k) - x^*\|$. 
Then
\begin{multline} 
    \g{z}{\tz}-\g{\x{*}{k}{\tz}}{\tz} \leq  \frac{\Lx^2}{2} \big( N ( 7 \Lx^2  \gammaz^2 + 6 \Lx  \gammaz + 3  )  + 3  \big) 
    \sum^N_{i=1} \| \x{}{}{k} - \x{i}{}{k^i} \|^{2} \\ + \frac{1}{2\gammaz^2} \big( (3\Lx^2 
      + 4) \gammaz^2 + 6\Lx \gammaz 
      + 3 \big)  \| \bar{z}(k) - \x{}{}{k} \| ^{2} 
      + \frac{N}{2} \big( 3\Lx^2  
      + 4  \big) \| \x{*}{k}{\tz} - \x{}{}{k} \|^{2},
\end{multline}
where $z \in \R^n$ with~$z_i = \x{i}{i}{k^i+1}$ for all~$i \in [N]$.
Applying Lemma~\ref{lem:noGammaUpdate} to Lemma~\ref{lem:errorBound1} yields
\begin{equation} \label{eq:78}
    \min_{\x{*}{}{} \in \X^*(\tz)} \| \x{}{}{k} - \x{*}{}{} \| 
    \leq  \lambda_z \max\{ 1,\gammaz^{-1} \} \| \x{}{}{k} - \Pi_\X \big[ \x{}{}{k} -  \gammaz \nabla_x \g{\x{}{}{k}}{\tz} \big] \|. 
\end{equation}
Next, applying~\eqref{eq:78}, using~$0 < \gammaz < 1$, and simplifying gives
\begin{multline} \label{eq:23}
    \g{z}{\tz}-\g{\x{*}{k}{\tz}}{\tz} \leq 
      \frac{\Lx^2}{2} \big( N (7 \Lx^2 \gammaz^2 + 6 \Lx \gammaz + 3 ) + 3  \big) \sum^N_{i=1} \| \x{}{}{k} - \x{i}{}{k^i} \|^{2} \\
      + \frac{ 1 }{2\gammaz^2} \Big(   (3\Lx^2 + 4) \gammaz^2 
       + 6\Lx \gammaz + 3 + N \lambda_z^2 \big( 3\Lx^2  + 4  \big) \Big)   \|   \bar{z}(k) - \x{}{}{k} \|^2  .
\end{multline}
We note that~$\| \x{}{}{k} - \x{i}{}{k^i} \|^2 = \sum^N_{j=1} \  \|\x{}{j}{k} - \x{i}{j}{k^i}\|^2$. Assumption~\ref{ass:partialAsynch} gives us~$k-B \leq \tau^i_j (k^i) \leq k + B - 1$ for all~$i,j \in [N]$. Using these facts and~\eqref{eq:1000}, we reach
\begin{multline} 
    \g{z}{\tz}-\g{\x{*}{k}{\tz}}{\tz} \leq 
      \frac{\Lx^2}{2} \big( N( 7 \Lx^2 \gammaz^2 + 6 \Lx  \gammaz + 3  )+ 3  \big) \sum^{k+B-1}_{\tau =k-B} \| \s{}{\tau}{\ell} \|^2 \\
      + \frac{ 1 }{2\gammaz^2} \Big(   (3\Lx^2 + 4 )\gammaz^2 + 6\Lx \gammaz + 3 + \lambda_z^2 \big( 3\Lx^{2} N + 4 N \big) \Big)   \|   \bar{z}(k) - \x{}{}{k} \|^2  . \label{eq:19}
\end{multline}
We also know that~$\x{}{i}{k+B} - \x{}{i}{k^i+1} = \sum^{k+B-1}_{\tau = k^i+1} \s{i}{\tau}{\ell}$. 
By definition of~$z_i=\x{}{i}{k^i+1}$ we have~$\x{}{}{k+B} = z + v$, where we use~$v \coloneqq \big(
        \sum^{k+B-1}_{\tau=k^1 + 1} \s{1}{\tau}{\ell}^T 
        \hdots 
        \sum^{k+B-1}_{\tau=k^N + 1} \s{N}{\tau}{\ell}^T
    \big)^T$.
Then~$\g{\x{}{}{k+B}}{\tz} = \g{z + v}{\tz}$. 
Using Lemma~\ref{lem:descent2}, 
adding~$\nabla_{\x{}{i}{}} \g{\x{i}{}{\tau}}{\tz} - \nabla_{\x{}{i}{}} \g{\x{i}{}{\tau}}{\tz}$, 
and using the Lipschitz property of~$\nabla_{\x{}{}{}} \g{\cdot}{\tz}$ from Section~\ref{ss:p2objectives}, we find
\begin{multline}
    \g{\x{}{}{k+B}}{\tz} \! \leq \! \g{z}{\tz}  
    + \Lx \! \sum^N_{i=1} \! \sum^{k+B-1}_{\tau = k^i + 1} \! \| \s{i}{\tau}{\ell} \| \! \| z - \x{i}{}{\tau} \| \\
    + \sum^N_{i=1} \sum^{k+B-1}_{\tau = k^i + 1} \s{i}{\tau}{\ell}^T \nabla_{\x{}{i}{}} \g{\x{i}{}{\tau}}{\tz}     
    + \frac{\Lx}{2} \| v \|^2 
    . \label{eq:15}
\end{multline}
By partial asynchrony we find that $\| z - \x{i}{}{\tau} \| \leq \sum^N_{j=1} \sum^{k^j}_{\zeta=\tau^i_j(\tau)} \| \s{j}{\zeta}{\ell} \|$. Using this in~\eqref{eq:15}, along with
${\|\s{i}{\tau}{\ell}\| \cdot \|\s{j}{\zeta}{\ell}\|  \leq  \frac{1}{2} \big( \|\s{i}{\tau}{\ell}\|^2 + \|\s{j}{\zeta}{\ell}}\|^2 \big)$, the fact that $k \leq k^i + 1$ for~$i \in [N]$ and
the fact that~$k \leq \tau$, we have
\begin{multline}
    \!\!\!\!\!\g{\x{}{}{k+B}}{\tz} \leq \g{z}{\tz} +  \Lx B N \sum^{k+B-1}_{\tau = k} \|\s{}{\tau}{\ell} \|^2 + \frac{\Lx}{2} \| v \|^2 \\
    + \sum^N_{i=1} \sum^{k+B-1}_{\tau = k^i + 1} \s{i}{\tau}{\ell}^T \nabla_{\x{}{i}{}} \g{\x{i}{}{\tau}}{\tz}  
    + \frac{\Lx B N}{2} \sum^{k+B-1}_{\zeta = k - B}  \|\s{}{\zeta}{\ell}\|^2, 
    \label{eq:16}
\end{multline}
where we have also used~$\sum^N_{i=1} \|\s{i}{k}{\ell}\|^2 = \|\s{}{k}{\ell}\|^2$.
By definition of~$v$ we have~$\| v \|^2 = \sum^N_{j=1} \big\Vert \sum^{k+B-1}_{\tau = k^j+1} \s{j}{\tau}{\ell} \big\Vert^2$.
Using the triangle inequality and~$ \Big( \sum^N_{i=1} z_i \Big)^2 \leq N \sum^N_{i=1} z_i^2$ gives $\big\Vert v \big\Vert^2 \leq \sum^N_{j=1}  \Big( \big(k+B -1 \big) - \big( k^j+1 \big) + 1 \Big)  \sum^{k+B-1}_{\tau = k^j+1} \big\Vert \s{j}{\tau}{\ell} \big\Vert^2$.
By partial asynchrony, for each~$j \in [N]$, we have the inequalities $\big(k+B -1 \big) - \big( k^j+1 \big) + 1 \leq B$ and~$k \leq k^j + 1$. Thus, $\| v \|^2 \leq B \sum^N_{j=1}   \sum^{k+B-1}_{\tau = k} \big\Vert \s{j}{\tau}{\ell} \big\Vert^2$.
Next, using~$\sum^N_{i=1} \|\s{i}{k}{\ell}\|^2 = \|\s{}{k}{\ell}\|^2$ gives~$\| v \|^2 \leq B   \sum^{k+B-1}_{\tau = k} \big\Vert \s{}{\tau}{\ell} \big\Vert^2$.
Using this and Lemma~\ref{lem:descent} in~\eqref{eq:16} gives
\begin{multline}
    \g{\x{}{}{k+B}}{\tz} - \g{z}{\tz} \leq \frac{\Lx B N}{2}  \sum^{k-1}_{\zeta = k - B}  \|\s{}{\zeta}{\ell}\|^2 \\
    - \Bigg[ \frac{1}{\gammaz} - \frac{ \Lx B (3 N + 1) }{2} \Bigg]
    \sum^N_{i=1} \sum^{k+B-1}_{\tau = k^i + 1} \| \s{i}{\tau}{\ell} \|^2 
    + \sum^N_{i=1} \| \s{i}{k^i}{\ell} \|^2. 
\end{multline}
Using~$\gammaz <  \frac{2}{\Lx B (3 N + 1)}$ lets us drop the second
term on the right-hand side, and then adding~\eqref{eq:19}
gives
\begin{multline}
    \g{\x{}{}{k+B}}{\tz} -\g{\x{*}{k}{\tz}}{\tz} \leq 
      \frac{\Lx^2}{2} \big( N (7 \Lx^2 \gammaz^2 + 6 \Lx \gammaz + 3  ) + 3  
        \big) \sum^{k+B-1}_{\tau =k-B} \| \s{}{\tau}{\ell} \|^2  \\
      + \frac{ 1 }{2\gammaz^2} \Big(   (3\Lx^2  
      + 4) \gammaz^2 
      + 6\Lx \gammaz  + 3 + N \lambda_z^2 \big( 3\Lx^{2}  
      + 4 \big) \Big)   \|   \bar{z}(k) - \x{}{}{k} \|^2 \\
      + \frac{\Lx B N}{2}  \sum^{k-1}_{\zeta = k - B}  \|\s{}{\zeta}{\ell}\|^2 
    + \sum^N_{i=1} \| \s{i}{k^i}{\ell} \|^2. \label{eq:24}
\end{multline}
Using~$k \leq k^i \leq k+B-1$ gives 
\begin{equation}
    \sum^N_{i=1} \| \s{i}{k^i}{\ell} \|^2 \leq \sum^N_{i=1}  \sum^{k+B-1}_{\tau=k} \| \s{i}{\tau}{\ell} \|^2 
    \leq \sum^{k+B-1}_{\tau=k} \| \s{}{\tau}{\ell} \|^2. \label{eq:20}
\end{equation}
Using Lemma~\ref{lem:update} and~\eqref{eq:20} in~\eqref{eq:24} completes the proof. 
\hfill $\blacksquare$

\begin{lemma} \label{lem:moreBound}
For each~$\tz \in \Ts$ there exists~$\gammaz < \frac{1}{2}$ such that for all~$k \in \{\eta_{z-1} + \kz,\dots,\eta_{z} - B \}  $, we have~$\at{k+B}{z} \leq \gammaz^{-2}F_z \bt{k+B} + \gammaz^{-1} G_z \bt{k}$,
% \begin{align}
%     \at{k+B}{z} &\leq \gammaz^{-2}F_z \bt{k+B} + \gammaz^{-1} G_z \bt{k} ,
% \end{align}
where~$F_z =  Z_1 +  Z_2 + 1$ and~$G_z =  Z_1 + Z_3 + A_4$, where~$A_4$ is from Lemma~\ref{lem:together} and
\begin{align}
    &Z_1 = \frac{N \Lx^2}{2} \big( 7 \Lx^2 N + 6 \Lx N + 3 N + 3  \big) \\
    &Z_2 = \frac{ 3 }{2 }   \big(  BN ( 6 \Lx^4  + 12 \Lx^3  + 14 \Lx^2 )  + 3\Lx^2   + 6\Lx + 7  +  N \lambda_z^2 (6 \Lx^4 B N + 8 \Lx^2 B N  +  3\Lx^{2}  + 4  ) \big) \\
    &Z_3 = N B \Lx^2 \big( 9 \Lx^2  + 18 \Lx + 21 + N \lambda_z^2 (9 \Lx^2 + 12 )  \big) .
\end{align}
\end{lemma}
\noindent \emph{Proof.}
For~$k \in \{\eta_{z-1} + \hat{k}_z, \dots, \eta_{z} - B\}$ we have 
\begin{equation} 
    \at{k+B}{z}  \leq \big( A_1+A_2 + 1 \big) \bt{k+B} 
    + \big( A_1+A_3+A_4 \big) \bt{k} , \label{eq:28}
\end{equation}
where~$A_1,A_2,A_3,A_4>0$ are from Lemma~\ref{lem:together}.
Since~$\gammaz < 1$, 
\begin{equation} \label{eq:25}
    A_1 \leq \frac{N \Lx^2}{2} \big( 7 \Lx^2 N + 6 \Lx N + 3 N + 3  \big).
\end{equation}
Taking~$\gammaz < \frac{1}{2}$ we have~$\frac{1}{1-\gammaz} < 1+2 \gammaz$. 
Then
%Using this inequality and~$\gammaz < 1$, we find
\begin{equation} \label{eq:26}
    A_2 \leq \gammaz^{-2} \frac{ 3 }{2  }   \big(  BN ( 6 \Lx^4  + 12 \Lx^3  + 14 \Lx^2 )  + 3\Lx^2   + 6\Lx + 7  +  N \lambda_z^2 (6 \Lx^4 B N + 8 \Lx^2 B N  +  3\Lx^{2}  + 4  ) \big) .
\end{equation}
Now we bound~$A_3$.
Using~$\gammaz < \frac{1}{2}$ we again have~$\frac{1}{1-\gammaz} < 1+2 \gammaz$. Applying this and the fact that~$\gammaz < 1$ yields
% \begin{multline} \label{eq:27}
%     A_3 \leq \frac{ 3 }{2\gammaz} \big(  6 \Lx^4 B N + 12 \Lx^3 B N + 14 \Lx^2 B N \\
%     + 6 \Lx^4 B N^2 \lambda_z^2 + 8 \Lx^2 B N^2 \lambda_z^2 \big).
% \end{multline}
\begin{equation} \label{eq:27}
    A_3 \leq \gammaz^{-1} N B \Lx^2 \big( 9 \Lx^2 + 18 \Lx + 21 + N \lambda_z^2 (9 \Lx^2 + 12 ) \big).
\end{equation}
From~\eqref{eq:25}-\eqref{eq:27} we have~$A_1 \leq Z_1$,~$A_2 \leq \gammaz^{-2} Z_2$, and~$A_3 \leq \gammaz^{-1} Z_3$.
Using these in~\eqref{eq:28} with~$\gammaz < 1$ gives the result.
\hfill $\blacksquare$
\begin{lemma} \label{lem:alphaBound}
We have~$\at{\kzero}{0} \leq a_0$ and~$\at{\kzero + B}{0} \leq a_0$, where~$a_0 \geq \Lg{0} \dx$.
% \begin{align}
%     \begin{array}{cc}
%         \at{\kzero}{0} \leq a_0 & \at{\kzero + B}{0} \leq a_0,
%     \end{array}
% \end{align}
% where
% \begin{align}
%     a_0 &\geq \Lg{0} \dx.
% \end{align}
\end{lemma}
\noindent \emph{Proof.}
By definition of~$\at{\kzero}{0}$ in~\eqref{eq:1003} we have
\begin{equation} \label{eq:21}
    \at{\kzero}{0} = \g{\x{}{}{\kzero}}{\ind{}{0}} - \g{\x{*}{\kzero}{\ind{}{0}}}{\ind{}{0}}, 
\end{equation}
where~$g(\x{*}{\kzero}{\ind{}{0}}; t_0)$ is constant for~$k \in \{\kzero, \dots, \eta_1\}$
by Lemma~\ref{lem:kHat}. 
Using this fact and the Lipschitz continuity of~$\g{\cdot}{\ind{}{0}}$ from Section~\ref{ss:p2objectives} we have~$\at{\kzero}{0} \leq \Lg{0} \| \x{}{}{\kzero} - \x{*}{\kzero}{\ind{}{0}} \|$. Then~$\at{\kzero}{0} \leq \Lg{0} \dx$, which follows from~$\| \x{}{}{\kzero} - \x{*}{\kzero}{\ind{}{0}} \| \leq \dx$.
The same argument bounds~$\at{\kzero + B}{0}$ by replacing~$\at{\kzero}{0}$ with~$\at{\kzero + B}{0}$ in~\eqref{eq:21}.
Thus, using %~$\at{\kzero + B}{0} \leq \Lg{0} \dx$.
$a_0 \geq \Lg{0} \dx$ gives the result. \hfill $\blacksquare$
\begin{lemma}\label{lem:betaBound}
For all~$\tz \in \Ts$ and $k \geq 0$, we have~$\bt{k}\leq B \dx^{2}$.
\end{lemma}
\noindent \emph{Proof.}
By definition of $\bt{k}$ in~\eqref{eq:1032} and~$\s{}{\tau}{\ell}$ in~\eqref{eq:1005} we have~$\bt{k} = \sum_{\tau=k-B}^{k-1} \sum_{i=1}^{N} \|x^i_i(\tau+1) - x^i_i(\tau)\|^2$.
Then ${\bt{k} \leq  \sum_{\tau=k-B}^{k-1} \sup_{v \in \X, w \in \X} \|v - w\|^2}$. 
Using $\sup_{v \in \X, w \in \X} \|v - w\| \leq \dx$ we find that~$\bt{k} \leq B \dx^2$.
\hfill $\blacksquare$
% By definition of $\bt{k}$ we have
% \begin{align}
%     \bt{k} & =\sum_{\tau=k-B}^{k-1}\left\Vert \s{}{\tau}{\ell}\right\Vert ^{2}. \label{eq:22}
% \end{align}
% By definition of~$\s{}{\tau}{\ell}$ in Equation~\eqref{eq:1005} we have
% \begin{align}
%     \bt{k} & = \sum^{k - 1}_{\tau = k-B} \left\Vert\begin{bmatrix}
%         \x{1}{1}{\tau + 1}  - \x{1}{1}{\tau} \\ \vdots \\ \x{N}{N}{\tau + 1} - \x{N}{N}{\tau }
%     \end{bmatrix} \right\Vert^{2}.
% \end{align}
% This is equivalent to
% \begin{align}
%     \bt{k} & = \sum_{\tau=k-B}^{k-1} \sum_{i=1}^{N} \|x^i_i(\tau+1) - x^i_i(\tau)\|^2.
% \end{align}
% We can further bound this as
% \begin{align}
%     \bt{k} & \leq \sum_{\tau=k-B}^{k-1} \sum_{i=1}^{N} \sup_{v_i \in \X_i, w_i \in \X_i} \|v_i - w_i\|^2 \\
%     &=  \sum_{\tau=k-B}^{k-1} \sup_{v \in \X, w \in \X} \|v - w\|^2.
% \end{align}
% Due to the fact~$ \sup_{v \in \X, w \in \X} \|v - w\| \leq \dx$ we find
% \begin{align}
%      \bt{k} & \leq B \dx^2,
% \end{align}
% for all~$k \geq 0$.

\begin{lemma} \label{lem:time0converge}
For~$\ind{}{0} \in \Ts$ and~$\gamma_{\max,0} \in (0,1)$ from~\eqref{eq:gammaMax},
% \begin{multline}
%     \gamma_{\max,0} \coloneqq 
%     \min \Bigg\{ \frac{1}{2},\frac{2}{L_0 ( 1 + B  + 2 N B)}, \frac{8 K_g \lambda^2_0}{\epsilon_0^2}, \\
%                 \left( \frac{G_0}{F_0} + \frac{E_0}{D_0} \right)^{-1},\frac{2 \zeta^2}{4 K_g + \zeta^2 L_0 (  1 + B  + 2 N B )}, \\
%                 \frac{ \frac{a_0}{b_0} + 2E_0 + D_0 c_0 - \sqrt{\big( \frac{a_0}{b_0} + 2E_0 + D_0 c_0\big)^2 - 4 D_0 E_0 c_0}}{2 E_0 c_0} ,
%                 \\
%                 \frac{D_0}{E_0},   
%                 \frac{1}{2 c_0},
%                 \frac{D_0}{8 F_0 \Big( \frac{G_0}{F_0} + \frac{E_0}{D_0} \Big) c_0},\frac{2}{3 L_0 B N + L_0 B}
%                  \Bigg\}, 
% \end{multline}
% \begin{align}
% \gamma_{\max,z} \coloneqq 
% \min \bigg\{& \frac{2}{3 L_0 B N + L_0 B}, \frac{2}{L_0 + B L_0 + L_0 N B}, \frac{D_0}{E_0},
% \bigg(\frac{G_0}{F_0} + \frac{E_0}{D_0} \bigg)^{-1},
% \frac{1}{2c_0},\\
% & \frac{D_0}{8 F_0 \Big( \frac{G_0}{F_0} + \frac{E_0}{D_0} \Big) c_0},
% \frac{ \frac{a_0}{b_0} + 2E_0 + D_0 c_0 - \sqrt{\big( \frac{a_0}{b_0} + 2E_0 + D_0 c_0\big)^2 - 4 D_0 E_0 c_0}}{2 E_0 c_0}, 
% \frac{1}{2} \bigg\}, 
% \end{align}
for any~$\gamma_0 \in \big( 0,\gamma_{\max,0} \big)$ and any~$r_0 \in \N_0$, we have 
\begin{align}
    \at{ \kzero + r_0 B}{0} & \leq a_0 \pt{r_0-1}{0} \label{eq:3000}\\
    \bt{ \kzero + r_0 B} & \leq b_0 \pt{r_0-1}{0} \label{eq:3001},
\end{align}
where
% \begin{align}
    % D_0 &= \frac{ 2 - \gamma_0 \big( L_0 + B L_0 + L_0 N B \big) }{2}, \quad 
    % E_0 = \frac{L_0 N B}{2}, \\
    % F_0 &=  \frac{N L_0^2}{2} \big( N ( 7 L_0^2 + 6 L_0  + 3  ) + 3 \big) 
    % + \frac{ 3 }{2}   \big(  BN ( 6 L_0^4  \\
    % &+ 12 L_0^3  + 14 L_0^2 )  + 3L_0^2   + 6L_0 + 7  \\
    % &+  N \lambda_z^2 (6 L_0^4 B N + 8 L_0^2 B N  +  3L_0^{2}  + 4  ) \big) + 1 \\
    % G_0 &=  \frac{N L_0^2}{2} \big( N (7 L_0^2 + 6 L_0  + 3  ) + 3 \big) + \frac{L_0 B N}{2} \\
    % &+  N B L_0^2  \big( 9 L_0^2    + 18 L_0  + 21  + N \lambda_0^2 (9 L_0^2  + 12 ) \big) \\
%     c_0 &= \frac{D_0}{2 F_0 + 2 D_0}, \quad
%     \pt{}{0} = 1 - \gamma_0 c_0 \in (0,1), \label{eq:3010} \\
%     a_0 &= \max \left\{ \Lg{0} \dx, 8 E_0 \Big( \frac{G_0}{F_0} + \frac{E_0}{D_0} \Big) F_0 B \dx^{2} D_0^{-1} \right\} \\
%     b_0 &= \frac{D_0}{8 E_0 \Big( \frac{G_0}{F_0} + \frac{E_0}{D_0} \Big) F_0} a_0  .
% \end{align}
~$D_0,E_0,F_0,G_0, a_0,b_0,c_0>0$ are defined in Section~\ref{ss:phases}.
\end{lemma}
\noindent \emph{Proof.}
For convenience we define~$\s{}{\tau}{0}=0$ for~$\tau<0$.
From Lemma~\ref{lem:phase2} and Lemma~\ref{lem:moreBound} for any~$k$ such that~$ \kzero \leq k \leq \eta_0 - B$,
\begin{align}
    \at{k+B}{0}  \leq & \at{k}{0} \! -  \! \gamma_0^{-1} D_0 \bt{k+B} \! + \! E_0 \bt{k} \label{eq:3002} \\
    \at{k+B}{0} \leq & \gamma_0^{-2}F_0 \bt{k+B} + \gamma_0^{-1} G_0 \bt{k}. \label{eq:3003}
\end{align}
Rearranging~\eqref{eq:3003} gives~$\bt{k+B} \geq \frac{\at{k+B}{0}  - \gamma_0^{-1} G_0 \bt{k}}{\gamma_0^{-2}F_0}$.
Substituting this into~\eqref{eq:3002} and rearranging yields
% \begin{multline}
%     \at{k+B}{0} \leq \at{k}{0} - \gamma_0^{-1}D_0 \bt{k+B} + E_0 \bt{k} \\
%     \leq \at{k}{0} 
%     - \gamma_0^{-1}D_0 \bigg( \frac{\at{k+B}{0}  - \gamma_0^{-1} G_0 \bt{k}}{\gamma_0^{-2}F_0} \bigg)
%     + E_0 \bt{k}.
% \end{multline}
% Rearranging we get
\begin{equation}
    % \at{k+B}{0} \leq \at{k}{0} - \frac{\gamma_0 D_0}{F_0}\at{k+B}{0} + \frac{D_0 G_0}{F_0} \bt{k} \\
    % + E_0 \bt{k} \\
    % \leq \at{k}{0} - \frac{\gamma_0 D_0}{F_0}\at{k+B}{0} + \bigg(  \frac{D_0 G_0}{F_0}
    % + E_0 \bigg) \bt{k} \\
    \bigg[ 1 + \frac{\gamma_0 D_0}{F_0} \bigg] \at{k+B}{0} \leq \at{k}{0} + \bigg[  \frac{D_0 G_0}{F_0}
    + E_0 \bigg] \bt{k}. \label{eq:3004}
\end{equation}
Fix any~$k \in \{\kzero + B, \dots, \eta_0 - B\}$, substitute~$k-B$ for~$k$ in~\eqref{eq:3002} to find
\begin{equation}
    \bt{k} \leq \gamma_0 \frac{ \at{k-B}{0} -\at{k}{0} + E_0 \bt{k-B}}{D_0}. \label{eq:3005}
\end{equation}
Applying ~\eqref{eq:3005} to the right-hand side of~\eqref{eq:3004}, we find 
% \begin{multline}
%     \bigg( 1 + \frac{\gamma_0 D_0}{F_0} \bigg) \at{k+B}{0} \leq \at{k}{0} + \gamma_0 \bigg(  \frac{D_0 G_0}{F_0} \\
%     + E_0 \bigg) \bigg( \frac{ \at{k-B}{0} -\at{k}{0} + E_0 \bt{k-B}}{D_0} \bigg).
% \end{multline}
% Rearranging we get
\begin{multline}
     \at{k+B}{0} \! \leq \! \bigg[ 1 + \frac{\gamma_0 D_0}{F_0} \bigg]^{-1}  \Bigg[ \bigg[1 - \gamma_0 \bigg[ \frac{G_0}{F_0} 
     + \frac{E_0}{D_0} \bigg]\bigg] \at{k}{0} \\
    + \gamma_0 \bigg[ \frac{G_0}{F_0}
    + \frac{E_0}{D_0} \bigg] \Big[ \at{k-B}{0} 
    + E_0 \bt{k-B} \Big] \Bigg] . \label{eq:3006}
\end{multline}
For any integer~$d \geq 2$, we iterate~\eqref{eq:3002} to find
\begin{equation}
    \at{k+dB}{0}  \leq \at{k}{0} - \big( \gamma_0^{-1} D_0 - E_0 \big) \sum^{d-1}_{j=1} \bt{k+jB} 
    - \gamma_0^{-1} D_0 \bt{k+dB} + E_0 \bt{k}.
\end{equation}
Using~$\gamma_0 < D_0/E_0$ and~$\bt{\tau} \geq 0$ for all~$\tau>0$ gives 
\begin{equation}
{\at{k+dB}{0} \leq \at{k}{0} - \big( \gamma_0^{-1} D_0 - E_0 \big) \bt{k+B} + E_0 \bt{k}}.
\end{equation}
% Taking~$\gamma_0 < D_0 / E_0$, we obtain from the nonnegativity of~$\bt{\tau}$, for all~$\tau$, that
Since~$\at{k+dB}{0} \geq 0$,  we then have the bounds 
\begin{equation}
{0 \leq \at{k}{0} - \big( \gamma_0^{-1} D_0 - E_0 \big) \bt{k+B} + E_0 \bt{k}}
\end{equation}
and 
\begin{equation}
   \bt{k+B} \leq \frac{\gamma_0}{ D_0 - \gamma_0 E_0} \big( \at{k}{0} + E_0 \bt{k} \big). \label{eq:3007}
\end{equation}

Next, we will use~\eqref{eq:3006} and~\eqref{eq:3007} to show the linear convergence of Algorithm~\ref{alg:myAlg} for~$t_0$. 
Select~$a_0 \geq \Lg{0} \dx$ and~${b_0 \geq B \dx^{2}}$ according to Lemma~\ref{lem:alphaBound} and Lemma~\ref{lem:betaBound}. Then~$\at{\kzero}{0} \leq a_0$,~$\at{\kzero + B}{0} \leq a_0$,~$\bt{ \kzero} \leq b_0$, and~$\bt{\kzero + B} \leq b_0$.
% \begin{align}
%     \begin{array}{cc}
%          \at{\kzero}{0} \leq a_0 & \at{\kzero + B}{0} \leq a_0 \\
%          \bt{ \kzero} \leq b_0  & \bt{\kzero + B} \leq b_0. \label{eq:3008}
%     \end{array}
% \end{align}
Then~\eqref{eq:3000} and~\eqref{eq:3001} hold for~$r_0=0,1$. Next, we will use induction to show that~\eqref{eq:3000} and~\eqref{eq:3001} hold for all~$r_0 \in \N_0$. 
For the inductive hypothesis suppose~\eqref{eq:3000} and~\eqref{eq:3001} hold for all~$r_0$ up to~$d \geq 1$. 
From~$\gamma_0 < 1 \big/ \big( \frac{G_0}{F_0} + \frac{E_0}{D_0} \big)$ and~\eqref{eq:3006} we find 
\begin{multline} \label{eq:29}
     \at{\kzero+dB+B}{0} \leq \frac{1}{1 + \gamma_0 \frac{D_0}{F_0}}  \Bigg( \bigg(1 - \gamma_0 \bigg( \frac{G_0}{F_0} 
     + \frac{E_0}{D_0} \bigg)\bigg) \at{\kzero+dB}{0} \\
    + \gamma_0 \bigg( \frac{G_0}{F_0} 
    + \frac{E_0}{D_0} \bigg) \Big( \at{\kzero+dB-B}{0} 
    + E_0 \bt{\kzero+dB-B} \Big) \Bigg).
\end{multline}
For~$\gamma_0 \leq \frac{1}{2 c_0}$ we have~$\pt{-1}{0} \leq 1 + 2c_0\gamma_0$. Using this bound and the inductive hypothesis, we find
\begin{multline}
     \at{\kzero+dB+B}{0} \leq \frac{1}{1 + \gamma_0 \frac{D_0}{F_0}}  \Bigg( 1 + 2 \gamma_0^2 c_0 \bigg( \frac{G_0}{F_0} + \frac{E_0}{D_0} \bigg) \\
     + \gamma_0 \bigg( \frac{G_0}{F_0} + \frac{E_0}{D_0} \bigg) \big( 1 + 2c_0 \gamma_0 \big) E_0 \frac{b_0}{a_0} \Bigg) a_0 \pt{d-1}{0}. 
\end{multline}
Then, using~$\frac{b_0}{a_0} = \frac{D_0}{8 E_0 \big( \frac{G_0}{F_0} + \frac{E_0}{D_0} \big) F_0}$,~$\gamma_0 \leq \frac{1}{2 c_0}$, and~$\gamma_0 \leq \frac{D_0}{8 F_0 \Big( \frac{G_0}{F_0} + \frac{E_0}{D_0} \Big) c_0}$ and simplifying gives
the upper bound
%\begin{equation}
     $\at{\kzero+dB+B}{0}  \leq  
     % \bigg( \frac{1}{1 + \gamma_0 \frac{D_0}{F_0}} + \frac{\gamma_0 D_0}{2 F_0 + 2 \gamma_0 D_0} \bigg) a_0 \pt{d-1}{0} \\
     % = \bigg( \frac{2 F_0}{2 F_0 + 2\gamma_0 D_0} + \frac{\gamma_0 D_0}{2 F_0 + 2 \gamma_0 D_0} \bigg) a_0 \pt{d-1}{0} \\
     % =
      \frac{2 F_0 + \gamma_0 D_0}{2 F_0 + 2 \gamma_0 D_0}  a_0 \pt{d-1}{0}$.
%\end{equation}
Using~$1 - \gamma_0 \frac{ D_0}{2F_0 + 2 \gamma_0 D_0} = \frac{2F_0 + \gamma_0 D_0}{2F_0 + 2 \gamma_0 D_0}$ and~$\gamma_0 < 1$, we reach
% \begin{equation}
%      \at{\kzero+dB+B}{0} \leq \bigg( 1 - \gamma_0 \frac{ D_0}{2F_0 + 2 D_0 \gamma_0 } \bigg) a_0 \pt{d-1}{0}.
% \end{equation}
% Taking~$\gamma_0 < 1$ we get
% \begin{equation}
%      \at{\kzero+dB+B}{0} \leq \bigg( 1 - \gamma_0 \frac{ D_0}{2F_0 + 2 D_0} \bigg) a_0 \pt{d-1}{0}.
% \end{equation}
% Recall that~$c_0 = \frac{D_0}{2 F_0 + 2 D_0}$. Then from the definition of~$\pt{}{0} = 1 - \gamma_0 c_0$ we get
% \begin{align}
%      \at{\kzero+dB+B}{0} &\leq \big( 1 - \gamma_0 c_0 \big) a_0 \pt{d-1}{0} \\
%      &\leq a_0 \pt{}{0} \pt{d-1}{0}.
% \end{align}
% Therefore,
\begin{equation}
     \at{\kzero+dB+B}{0} \leq a_0 \pt{d}{0}, \label{eq:3009}
\end{equation}
and this completes the induction on~\eqref{eq:3000}.
% since~\eqref{eq:3009} holds for~$d \geq 2$ and~\eqref{eq:3008} holds for~$d=0,1$. 

To complete the inductive argument we make a similar argument for~\eqref{eq:3001}. From~\eqref{eq:3007} and the inductive hypothesis we have the bound~$\bt{\kzero +dB+B} \leq \frac{\gamma_0 \big( \frac{a_0}{b_0} + E_0 \big)}{ D_0 - \gamma_0 E_0} b_0 \pt{d-1}{0}$.
% \begin{multline}
%    \bt{\kzero + dB+B} \leq \frac{\gamma_0}{ D_0 - \gamma_0 E_0} \Big( \at{\kzero + dB}{0} \\
%    + E_0 \bt{\kzero +dB} \Big). 
% \end{multline}
% By the inductive hypothesis we have
% \begin{align}
%    \bt{\kzero +dB+B} &\leq \frac{\gamma_0}{ D_0 - \gamma_0 E_0} \big( a_0 \pt{d-1}{0} + E_0 b_0 \pt{d-1}{0} \big). 
% \end{align}
% Factoring out~$ b_0 \pt{d-1}{0}$ we get
% \begin{align}
%    \bt{\kzero +dB+B} &\leq \frac{\gamma_0 \big( \frac{a_0}{b_0} + E_0 \big)}{ D_0 - \gamma_0 E_0} b_0 \pt{d-1}{0}.
% \end{align}
Here~$\gamma_0 < \gamma_{\max,0}$ ensures~$\gamma_0 \Big( \frac{a_0}{b_0} + E_0 \Big) \leq \big(  D_0 - \gamma_0 E_0 \big) \big( 1 - \gamma_0 c_0 \big)$, and we obtain~$\bt{\kzero +dB+B} \leq \big( 1 - \gamma_0 c_0 \big) b_0 \pt{d-1}{0}$.
The definition of~$\pt{}{0}$ gives~$\bt{\kzero +dB+B} \leq b_0 \pt{d}{0}$. 
%Thus,~\eqref{eq:3000} and~\eqref{eq:3001} hold for all~$r_0 \in \N_0$.
%
%
% We design~$\gamma_0 < \frac{ \frac{a_0}{b_0} + 2E_0 + D_0 c_0 - \sqrt{\big( \frac{a_0}{b_0} + 2E_0 + D_0 c_0\big)^2 - 4 D_0 E_0 c_0}}{2 E_0 c_0}$ to ensure that~$\gamma_0 \Big( \frac{a_0}{b_0} + E_0 \Big) \leq \big(  D_0 - \gamma_0 E_0 \big) \big( 1 - \gamma_0 c_0 \big)$ so that we obtain 
% \begin{align}
%    \bt{\kzero +dB+B} &\leq \big( 1 - \gamma_0 c_0 \big) b_0 \pt{d-1}{0}. \\
% \end{align}
% Then from the definition of~$\pt{}{0}=1 - \gamma_0 c_0$ in~\eqref{eq:3010} we get
% \begin{align}
%    \bt{\kzero +dB+B} &\leq b_0 \pt{d}{0}.
% \end{align}
% Therefore~\eqref{eq:3000} and~\eqref{eq:3001} hold for all~$r_0 \in \N_0$.
\hfill $\blacksquare$

%\input{Lemmas/0-upperConvexityBound}

%% Appendices
\subsection{Lemmas in Support of Proving Lemma~\ref{lem:phase3}} \label{app:towardlemma1proof}
Toward proving Lemma~\ref{lem:phase3}, we have several preliminary lemmas. 

\begin{lemma} \label{lem:phase1}
    For all~$\tz \in \Ts$, we have the upper bound
    \begin{equation}
        \g{\x{}{}{\eta_{z-1}+B}}{\tz} \leq   \g{\x{}{}{\eta_{z-1}}}{\tz} +
        B \dx  M_{x,z}
        + \frac{\Lx B^2 \dx^2}{2}.
    \end{equation}
\end{lemma}
 \begin{proof}
For all~$\tz \in \Ts$, we have
% \begin{equation}
    ~$\g{\x{}{}{\eta_{z-1}+ B} }{\tz} = \g{\x{}{}{\eta_{z-1}} + \sum^{\eta_{z-1}+B-1}_{\tau = \eta_{z-1}} \s{}{\tau}{\tau}}{\tz} $.
% \end{equation}
From Lemma~\ref{lem:descent2} we have
\begin{equation}
    \g{\x{}{}{\eta_{z-1}+B}}{\tz} \leq \g{\x{}{}{\eta_{z-1}}}{\tz}  
    + \left( \sum^{\eta_{z-1}+B-1}_{\tau = \eta_{z-1}} \!\!\!\! \s{}{\tau}{\tau} \! \right)^T \!\!\! \nabla_x \g{\x{}{}{\eta_{z-1}}}{\tz} 
    + \frac{\Lx}{2} \left\| \sum^{\eta_{z-1}+B-1}_{\tau = \eta_{z-1}} \!\!\!\!\s{}{\tau}{\tau} \right\|^2\!\!.
\end{equation}
For the first sum, we apply the Cauchy-Schwarz inequality and
the triangle inequality, 
and for the second sum we 
use the fact that~$ \big\Vert \sum^N_{i=1} z \big\Vert^2 \leq N \sum^N_{i=1} \Vert z \Vert^2$ to find
%\begin{multline}
%    \g{\x{}{}{\eta_{z-1}+B}}{\tz} \leq \g{\x{}{}{\eta_{z-1}}}{\tz}  \\
%    + \left( \sum^{\eta_{z-1}+B-1}_{\tau = \eta_{z-1}} \s{}{\tau}{\tau} \right)^T \nabla_x \g{\x{}{}{\eta_{z-1}}}{\tz} \\
%    + \frac{\Lx B}{2} \sum^{\eta_{z-1}+B-1}_{\tau = \eta_{z-1}} \| \s{}{\tau}{\tau} \|^2 .
%\end{multline}
%Applying the Cauchy-Schwarz inequality yields
%\begin{multline}
%    \g{\x{}{}{\eta_{z-1}+B}}{\tz} \leq \g{\x{}{}{\eta_{z-1}}}{\tz}  \\
%    + \left\| \sum^{\eta_{z-1}+B-1}_{\tau = \eta_{z-1}} \s{}{\tau}{\tau} \right\| 
%    \| \nabla_x \g{\x{}{}{\eta_{z-1}}}{\tz} \| \\
%    + \frac{\Lx B}{2} \sum^{\eta_{z-1}+B-1}_{\tau = \eta_{z-1}} \| \s{}{\tau}{\tau} \|^2 .
%\end{multline}
%Applying the triangle inequality gives us
\begin{multline}
    \g{\x{}{}{\eta_{z-1}+B}}{\tz} \leq  \frac{\Lx B}{2} \sum^{\eta_{z-1}+B-1}_{\tau = \eta_{z-1}} \| \s{}{\tau}{\tau} \|^2 \\
    + \g{\x{}{}{\eta_{z-1}}}{\tz}  + \| \nabla_x \g{\x{}{}{\eta_{z-1}}}{\tz} \|  \sum^{\eta_{z-1}+B-1}_{\tau = \eta_{z-1}} \| \s{}{\tau}{\tau} \| .
\end{multline}
The result follows from 
Lemma~\ref{lem:betaBound}, 
the boundedness of~$\nabla_x \g{\cdot}{\tz}$ from Section~\ref{ss:p2objectives}, 
and the fact that~$\Vert \s{}{\cdot}{z+1} \Vert \leq \dx$.
\end{proof}
%i.e.,~$ \sum^{\eta_z + B - 1}_{\tau = \eta_z} \big\Vert \s{}{\tau}{z+1} \big\Vert^2 \leq B \dx^{2}$, we obtain
%\begin{multline}
%    \g{\x{}{}{\eta_{z-1}+B}}{\tz} \leq \g{\x{}{}{\eta_{z-1}}}{\tz}  \\
%    + \left( \sum^{\eta_{z-1}+B-1}_{\tau = \eta_{z-1}} \| \s{}{\tau}{\tau} \| \right)
%    \| \nabla_x \g{\x{}{}{\eta_{z-1}}}{\tz} \| 
%    + \frac{\Lx B^2 \dx^2}{2} .
%\end{multline}
%Due to the boundedness of~$\nabla_x \g{\cdot}{\tz}$ in Equation~\eqref{eq:1006} we have
%\begin{multline}
%    \g{\x{}{}{\eta_{z-1}+B}}{\tz} \leq \g{\x{}{}{\eta_{z-1}}}{\tz}  \\
%    + \left( \sum^{\eta_{z-1}+B-1}_{\tau = \eta_{z-1}} \| \s{}{\tau}{\tau} \| \right)
%    M_{x,z} 
%    + \frac{\Lx B^2 \dx^2}{2} .
%\end{multline}
%Due to the fact~$\Vert \s{}{\cdot}{z+1} \Vert \leq \dx$ yields the result
%\begin{multline}
%    \g{\x{}{}{\eta_{z-1}+B}}{\tz} \leq \g{\x{}{}{\eta_{z-1}}}{\tz}  
%    + B \dx  M_{x,z} \\
%    + \frac{\Lx B^2 \dx^2}{2} . \label{eq:71}
%\end{multline}

% \end{appendice}

\begin{lemma} \label{lem:phase2}
    For all~$\tz \in \Ts$ and~$k \in \{ \eta_{z-1}+B, \dots, \eta_z - B \}$, 
    \begin{equation} \label{eq:l5result}
        \g{\x{}{}{k+B}}{\tz} - \g{\x{}{}{k}}{\tz}  \leq  \frac{\Lx N B}{2}  \bt{k} 
        -  \frac{ 2 - \gammaz \Lx ( 1 + B + N B ) }{2 \gammaz}  \bt{k+B}.
    \end{equation}
\end{lemma}

% For all~$\tz \in \Ts$ and~$k \in \{ \eta_{z-1}+B, \dots, \eta_z - B \}$, we have
%     \begin{multline} 
%         \g{\x{}{}{k+B}}{\tz} - \g{\x{}{}{k}}{\tz}  \leq   \\
%         - \Bigg( \frac{ 2 - \gammaz \big( \Lx + B \Lx + \Lx N B \big) }{2 \gammaz} \Bigg) \sum^{k+B-1}_{\tau = k} \| \s{}{\tau}{\ell} \|^2 \\
%         + \frac{\Lx N B}{2}  \sum^{k-1}_{\tau = k - B} \| \s{}{\tau}{\ell} \|^2. \label{eq:73}
%     \end{multline}

 \begin{proof}
  We begin by bounding the difference between~$\g{\x{}{}{k}}{\tz}$ and~$\g{\x{}{}{k+1}}{\tz}$. From the definition of~$\x{}{}{k+1}$ we have
%\begin{align}
    $\g{\x{}{}{k+1}}{\tz} = \g{\x{}{}{k} + \s{}{k}{\ell}}{\tz}$.
%\end{align}
From Lemma~\ref{lem:descent2}, adding~$\nabla_{\x{}{i}{}} \g{\x{i}{}{k}}{\tz} - \nabla_{\x{}{i}{}} \g{\x{i}{}{k}}{\tz}$
we find
% \begin{multline}
%     \g{\x{}{}{k+1}}{\tz}  \leq \g{\x{}{}{k}}{\tz} 
%     + \s{}{k}{\ell}^T \nabla_{\x{}{}{}} \g{\x{}{}{k}}{\tz} \\
%     + \frac{\Lx}{2} \|\s{}{k}{\ell}\|^2.
% \end{multline}
% Expressing the inner product as a sum across all agents then gives
% \begin{multline}
%     \g{\x{}{}{k+1}}{\tz}  \leq \g{\x{}{}{k}}{\tz} 
%     + \sum^N_{i=1}\s{i}{k}{\ell}^T \nabla_{\x{}{i}{}} \g{\x{}{}{k}}{\tz} \\
%     + \frac{\Lx}{2} \|\s{}{k}{\ell}\|^2.
% \end{multline}
% Adding zero as~$\nabla_{\x{}{i}{}} \g{\x{i}{}{k}}{\tz} - \nabla_{\x{}{i}{}} \g{\x{i}{}{k}}{\tz}$ and rearranging, we find
\begin{multline}
    \g{\x{}{}{k+1}}{\tz}  \leq \g{\x{}{}{k}}{\tz} 
    + \sum^N_{i=1}\s{i}{k}{\ell}^T  \nabla_{\x{}{i}{}} \g{\x{i}{}{k}}{\tz} \\
    + \sum^N_{i=1}\s{i}{k}{\ell}^T \! \big( \nabla_{\x{}{i}{}} \g{\x{}{}{k}}{\tz} - \nabla_{\x{}{i}{}} \g{\x{i}{}{k}}{\tz}\!\big) 
    + \frac{\Lx}{2} \|\s{}{k}{\ell} \|^2\!.
\end{multline}
Applying the Cauchy-Schwarz inequality, the Lipschitz property of~$\nabla_{\x{}{}{}} \g{\cdot}{\tz}$, Lemma~\ref{lem:outOfDate}, and ${\big\Vert \s{i}{k}{\ell} \big\Vert \cdot \big\Vert \s{}{\tau}{\ell} \big\Vert  \leq  \frac{1}{2} \big( \big\Vert \s{i}{k}{\ell} \big\Vert^2 + \| \s{}{\tau}{\ell} \|^2}\big)$ leads to
%\begin{multline}
%    \g{\x{}{}{k+1}}{\tz} \leq \g{\x{}{}{k}}{\tz} 
%    + \sum^N_{i=1}\s{i}{k}{\ell}^T  \nabla_{\x{}{i}{}} \g{\x{i}{}{k}}{\tz} \\
%    + \Lx \sum^N_{i=1} \big\Vert \s{i}{k}{\ell} \big\Vert \big\Vert \x{}{}{k} - \x{i}{}{k} \big\Vert 
%    + \frac{\Lx}{2} \|\s{}{k}{\ell}\|^2.
%\end{multline}
%Next, applying the result of Lemma~\ref{lem:outOfDate} and~${\big\Vert \s{i}{k}{\ell} \big\Vert \cdot \big\Vert \s{}{\tau}{\ell} \big\Vert  \leq  \frac{1}{2} \big( \big\Vert \s{i}{k}{\ell} \big\Vert^2 + \| \s{}{\tau}{\ell} \|^2}$, leads to
% \begin{multline}
%     \g{\x{}{}{k+1}}{\tz}  \leq \g{\x{}{}{k}}{\tz}  
%     + \sum^N_{i=1}\s{i}{k}{\ell}^T  \nabla_{\x{}{i}{}} \g{\x{i}{}{k}}{\tz} \\
%     + \Lx \sum^N_{i=1} \big\Vert \s{i}{k}{\ell} \big\Vert \sum^{k-1}_{\tau = k - B} \big\Vert \s{}{\tau}{\ell} \big\Vert 
%     + \frac{\Lx}{2} \|\s{}{k}{\ell}\|^2.
% \end{multline}
% Using the inequality~$ab \leq \frac{1}{2} (a^2 + b^2)$ we have ${\big\Vert \s{i}{k}{\ell} \big\Vert \cdot \big\Vert \s{}{\tau}{\ell} \big\Vert  \leq  \frac{1}{2} \big( \big\Vert \s{i}{k}{\ell} \big\Vert^2 + \| \s{}{\tau}{\ell} \|^2} \big)$ and thus
% \begin{multline}
%     \g{\x{}{}{k+1}}{\tz} \leq \g{\x{}{}{k}}{\tz} 
%     + \sum^N_{i=1}\s{i}{k}{\ell}^T  \nabla_{\x{}{i}{}} \g{\x{i}{}{k}}{\tz} \\
%     + \frac{\Lx}{2} \sum^N_{i=1} \sum^{k-1}_{\tau = k - B} \big( \big\Vert \s{i}{k}{\ell} \big\Vert^2 + \| \s{}{\tau}{\ell} \|^2 \big) 
%     + \frac{\Lx}{2} \|\s{}{k}{\ell}\|^2.
% \end{multline}
% Simplifying the double sum gives
\begin{multline}
    \g{\x{}{}{k+1}}{\tz}  \leq \g{\x{}{}{k}}{\tz} 
    + \sum^N_{i=1}\s{i}{k}{\ell}^T  \nabla_{\x{}{i}{}} \g{\x{i}{}{k}}{\tz} \\
    + \frac{\Lx}{2} \Biggl( \! B \sum^N_{i=1} \| \s{i}{k}{\ell} \|^2
    + N \! \!\!\!\sum^{k-1}_{\tau = k - B} \| \s{}{\tau}{\ell} \|^2 \! \Biggr) 
    + \frac{\Lx}{2} \|\s{}{k}{\ell}\|^2.
\end{multline}
Using~$\sum^N_{i=1} \big\Vert \s{i}{k}{\ell} \big\Vert^2 = \|\s{}{k}{\ell}\|^2$ and Lemma~\ref{lem:descent} gives 
% \begin{multline} 
%     \g{\x{}{}{k+1}}{\tz}  \leq \g{\x{}{}{k}}{\tz} 
%     + \sum^N_{i=1}\s{i}{k}{\ell}^T  \nabla_{\x{}{i}{}} \g{\x{i}{}{k}}{\tz} \\
%     + \frac{\Lx}{2} \Bigg( B \|\s{}{k}{\ell}\|^2 
%     + N \sum^{k-1}_{\tau = k - B} \| \s{}{\tau}{\ell} \|^2 \Bigg) 
%     + \frac{\Lx}{2} \|\s{}{k}{\ell}\|^2.
% \end{multline}
% Applying the result of Lemma~\ref{lem:descent} gives
% \begin{multline} 
%     \g{\x{}{}{k+1}}{\tz}  \leq   \g{\x{}{}{k}}{\tz} 
%     -\frac{1}{\gammaz} \sum^N_{i=1}  \| \s{i}{k}{\ell} \|^2 \\
%     + \frac{\Lx}{2} \Bigg( B \|\s{}{k}{\ell}\|^2 
%     + N \sum^{k-1}_{\tau = k - B} \| \s{}{\tau}{\ell} \|^2 \Bigg) 
%     + \frac{\Lx}{2} \|\s{}{k}{\ell}\|^2.
% \end{multline}
% From the fact that~$\sum^N_{i=1} \big\Vert \s{i}{k}{\ell} \big\Vert^2 = \|\s{}{k}{\ell}\|^2 $, we next have
% \begin{multline} 
%     \g{\x{}{}{k+1}}{\tz}  \leq   \g{\x{}{}{k}}{\tz} 
%     -\frac{1}{\gammaz}  \| \s{}{k}{\ell} \|^2 \\
%     + \frac{\Lx}{2} \Bigg( B \|\s{}{k}{\ell}\|^2 
%     + N \sum^{k-1}_{\tau = k - B} \| \s{}{\tau}{\ell} \|^2 \Bigg) 
%     + \frac{\Lx}{2} \|\s{}{k}{\ell}\|^2.
% \end{multline}
% Combining like terms and subtracting~$\g{\x{}{}{k}}{\tz}$ from both sides now gives
\begin{equation} 
    \g{\x{}{}{k+1}}{\tz} - \g{\x{}{}{k}}{\tz} \leq  \frac{N \Lx}{2}  \sum^{k-1}_{\tau = k - B} \| \s{}{\tau}{\ell} \|^2 
    + \Bigg( -\frac{1}{\gammaz} + \frac{\Lx}{2} + \frac{B \Lx}{2} \Bigg) \| \s{}{k}{\ell} \|^2. \label{eq:61}
\end{equation}
Applying this successively to~$k, k+1, \dots, k+B - 1$ and summing the resulting inequalities gives us the result. 
% \begin{multline} 
%     \g{\x{}{}{k+B}}{\tz} - \g{\x{}{}{k}}{\tz} \leq \\
%     \Bigg( -\frac{1}{\gammaz} + \frac{\Lx}{2} + \frac{B \Lx}{2} \Bigg)\sum^{k+B-1}_{\tau = k} \| \s{}{\tau}{\ell} \|^2 \\
%     + \frac{\Lx}{2} N B \sum^{k+B-1}_{\tau = k - B} \| \s{}{\tau}{\ell} \|^2 .
% \end{multline}
% Splitting the last sum gives
% \begin{multline} 
%     \g{\x{}{}{k+B}}{\tz} - \g{\x{}{}{k}}{\tz} \leq   \\
%     \Bigg( -\frac{1}{\gammaz} + \frac{\Lx}{2} + \frac{B \Lx}{2} \Bigg)\sum^{k+B-1}_{\tau = k} \| \s{}{\tau}{\ell} \|^2 \\
%     + \frac{\Lx}{2} N B \sum^{k+B-1}_{\tau = k} \| \s{}{\tau}{\ell} \|^2
%     + \frac{\Lx}{2} N B \sum^{k-1}_{\tau = k - B} \| \s{}{\tau}{\ell} \|^2.
% \end{multline}
% Combining like terms we obtain
% \begin{multline} 
%     \g{\x{}{}{k+B}}{\tz} - \g{\x{}{}{k}}{\tz} \leq \\
%     \Bigg( -\frac{1}{\gammaz} + \frac{\Lx}{2} + \frac{ B \Lx }{2} + \frac{ \Lx N B}{2} \Bigg) 
%     \sum^{k+B-1}_{\tau = k} \| \s{}{\tau}{\ell} \|^2 \\
%     + \frac{\Lx N B}{2}  \sum^{k-1}_{\tau = k - B} \| \s{}{\tau}{\ell} \|^2.
% \end{multline}
% Simplifying we obtain the result
% \begin{multline} 
%     \g{\x{}{}{k+B}}{\tz} - \g{\x{}{}{k}}{\tz}  \leq   \\
%     \Bigg( - \frac{ 2 - \gammaz \big( \Lx + B \Lx + \Lx N B \big) }{2 \gammaz} \Bigg) \sum^{k+B-1}_{\tau = k} \| \s{}{\tau}{\ell} \|^2 \\
%     + \frac{\Lx N B}{2}  \sum^{k-1}_{\tau = k - B} \| \s{}{\tau}{\ell} \|^2.
% \end{multline}
    
% \end{appendice}
\end{proof}

%\subsection{Proof of Lemma~\ref{lem:designGamma}} \label{lem:designGammaProof}
\begin{lemma} \label{lem:designGamma}
    We have~$\g{\x{}{}{k+B}}{\tz} - \g{\x{}{}{k}}{\tz}  < 0$ 
    for all~$\tz \in \Ts$, times~$k \in \{ \eta_{z-1}+B, \dots, \eta_z - B \}$, and stepsizes $\gammaz \in \left(0, \frac{2}{\Lx (1 + B  + 2  N B)} \right)$. 
    %, we have 
    %\begin{equation}
    %    \g{\x{}{}{k+B}}{\tz} - \g{\x{}{}{k}}{\tz}  < 0.
    %\end{equation}
\end{lemma}

    % For all~$\tz \in \Ts$,~$k \in \{ \eta_{z-1}+B, \dots, \eta_z - B \}$, and~$\gammaz \in \left(0, \frac{2}{\Lx + B \Lx + 2 \Lx N B} \right)$, it holds,
    % \begin{align}
    %     \g{\x{}{}{k+B}}{\tz} - \g{\x{}{}{k}}{\tz}  &< 0.
    % \end{align}

\begin{proof}
    % By Assumption~\ref{ass:partialAsynch} for all~$k \in \{ \eta_{z-1}+B, \dots, \eta_{z} \}$ all information in the network has been computed with the current objective, i.e.,~$\g{\cdot}{\tz}$. Then, by
Define $q_1 \coloneqq \frac{2 \gammaz}{ 2 - \gammaz \Lx ( 1 + B  +  N B ) } $.
By Lemma~\ref{lem:phase2} for all~$\tz \in \Ts$ and~$k \in \{ \eta_{z-1}+B, \dots, \eta_{z} - B\}$ we have
\begin{equation}
    \g{\x{}{}{k\!+\!B}}{\tz} - \g{\x{}{}{k}}{\tz} \! \leq \! \frac{\Lx N B}{2}  \bt{k} 
                -  \frac{ 1 }{q_1}  \bt{k \!+\! B}. \label{eq:66}
\end{equation}
% Next, we design~$\gammaz$ such that the negative term dominates on the RHS which will guarantee~$\g{\x{}{}{k}}{\tz}$ will decrease over intervals of length~$B$ for~$k \in \{ \eta_{z-1}+B,\dots,\eta_{z} - B \}$.
Then we would like to design~$\gammaz$ such that
\begin{equation}
    % \frac{ 2 - \gammaz \big( \Lx + B \Lx + \Lx N B \big) }{2 \gammaz}  \bt{k+B}
    \frac{ 1 }{q_1}  \bt{k+B}
    \geq 
    \frac{\Lx N B}{2} \bt{k}. \label{eq:89}
\end{equation}
% There are two cases we need to consider. 
First, suppose~$\bt{k+B} \geq \bt{k}$. Then it suffices to show 
%\begin{equation}
     % \frac{ 2 - \gammaz \big( \Lx + B \Lx + \Lx N B \big) }{2 \gammaz}  \bt{k+B}
     $\frac{ 1 }{q_1}  \bt{k+B}
    \geq 
    \frac{\Lx N B}{2} \bt{k+B}$.
%\end{equation}
% This is equivalent to
% \begin{multline}
%     \Bigg( \frac{ 1 }{ \gammaz} - \frac{  \Lx + B \Lx + \Lx N B  }{2} \Bigg) \bt{k+B}
%     \geq \\
%     \frac{\Lx N B}{2} \bt{k+B}.
% \end{multline}
Dividing by~$\bt{k+B}$ yields
%\begin{align}
    $\frac{ 1 }{ \gammaz} - \frac{ \Lx (1 + B  + N B ) }{2}
    \geq \frac{\Lx N B}{2}$.
%\end{align}
Solving for~$\gammaz$ yields
%\begin{align}
    % \frac{ 1 }{ \gammaz}
    % &\geq \frac{  \Lx + B \Lx + 2 \Lx N B }{2} \\
    $\gammaz \leq \frac{ 2 }{ \Lx (1 + B  + 2 N B)}$.
%\end{align}
Now, suppose~$\bt{k+B} \leq \bt{k}$. 
%To enforce~\eqref{eq:89} it suffices to show
%\begin{equation}
%    % \frac{ 2 - \gammaz \big( \Lx + B \Lx + \Lx N B \big) }{2 \gammaz} \bt{k+B}
%    \frac{ 1 }{q_1} \bt{k+B}
%    \geq 
%    \frac{\Lx N B}{2} \bt{k}.
%\end{equation}
Multiplying~\eqref{eq:89} by~$q_1$ gives 
%\begin{align}
    $\bt{k+B} \geq \frac{ \Lx N B \gammaz}{ 2 -  \gammaz \Lx ( 1 + B  +  N B ) } \bt{k}$.
%\end{align}
Since~$\bt{k+B} \leq \bt{k}$ in this case, it must be the case that
%\begin{align}
    $\frac{ \Lx N B \gammaz}{ 2 -  \gammaz \Lx ( 1 + B + N B ) }
    \leq 1$,
%\end{align}
%Multiplying both sides by the denominator of the left-hand side and 
which holds for~$\gammaz$ from above. 
%Solving for~$\gammaz$ again gives
% \begin{align}
%     2  \Lx N B \gammaz 
%     &\leq  4 - 2 \gammaz \big( \Lx + B \Lx + \Lx N B \big) \\
%       \Lx N B \gammaz 
%     &\leq  2 - \gammaz \big( \Lx + B \Lx + \Lx N B \big).
% \end{align}
% Rearranging this equation yields
% \begin{align}
%     \Lx N B \gammaz + \gammaz \big( \Lx + B \Lx + \Lx N B \big)
%     &\leq  2 \\ 
%     \gammaz \big(\Lx + B \Lx + 2 \Lx N B \big)
%     &\leq  2.
% \end{align}
% Solving for~$\gammaz$ yields
%\begin{equation}
%$\gammaz \leq  \frac{2}{\Lx + B \Lx + 2 \Lx N B}$.
%\end{equation}
%Therefore, designing~$\gammaz < \frac{2}{\Lx + B \Lx + 2 \Lx N B} $ ensures that
% \begin{multline} 
%         \g{\x{}{}{k+B}}{\tz} - \g{\x{}{}{k}}{\tz}  \leq  
%         \\
%          - \Bigg( \frac{ 2 - \gammaz \big( \Lx + B \Lx + \Lx N B \big) }{2 \gammaz} \Bigg) \bt{k+B} \\
%         + \frac{\Lx N B}{2}  \bt{k} 
%         < 0,
% \end{multline}
% \begin{equation}
%    ~$\g{\x{}{}{k+B}}{\tz} - \g{\x{}{}{k}}{\tz} < 0$,
% \end{equation}
%for all~$\tz \in \Ts$ and~$k \in \{\eta_{z-1}+B,\dots,\eta_{z}-B  \}$.
%\hfill $\blacksquare$
\end{proof}
    
% \end{appendice}
%\subsection{Proof of Lemma~\ref{lem:kHat}} \label{lem:kHatProof}
\begin{lemma} \label{lem:kHat}
    For~$\tz \in \Ts$, let the stepsize be
   %\begin{equation}
       $\gammaz \in \left(0, \min \left\{ 1, \frac{2}{ \Lx ( 1+ B + 2 N B) } \right\}\right)$. 
       Then there exists a finite time~$\hat{k}_z$ such
    %that
   %\begin{equation}
    $g(\x{*}{k}{\tz}) = g(\x{*}{k+1}{\tz})$
   %\end{equation}
   for~$k \in \{ \eta_{z-1} + \kz, \dots, \eta_{z}  \}$. 
   %and 
   %if~$2\epsilon_z^2 > 4L_{g,z}M_zB(1 + \lambda_z)^2$ 
   %then we may take 
   % \begin{align}
   % \kz \! &= \! \frac{B}{\log(q_2)} \log \left(\frac{ \epsilon_z^2 \big( 2 \gammaz - \Lx \big( 1 + B + B N \big) \gammaz^2 \big) }{ 16 K_g ( \gammaz^2 + 2 \lambda_z \gammaz + \lambda^2_z ) }\right) \! + 2B,  \\
    % &\x{*}{k}{\tz} = \argmin_{\x{*}{}{} \in \X^*(\tz)} \| \x{*}{}{} - \x{}{}{k} \|, \\
   % &\gmax = \g{\x{}{}{\eta_{z-1}}}{\tz} + B \dx  \Mx + \frac{\Lx B^2 \dx^2}{2} ,\\
   % &K_g = \gmax + \frac{\Lx B^2 N \dx^{2} }{2}, \quad \zeta = \frac{\epsilon_z}{2(\lambda_z \gammaz^{-1} + 1)}, \\
   % &q_2 = \frac{\Lx N B \gammaz}{ 2 - \gammaz \Lx ( 1 + B + N B  )} \in (0,1),
   % \end{align}
    %where $\epsilon_z>0$ is from Lemma~\ref{lem:solSeparation} and~$\lambda_z>0$ is from Lemma~\ref{lem:errorBound1}. 
\end{lemma}

\begin{proof}
Set~$\omega_z = g\big(x(\eta_{z-1} + B), t_z\big)$. 
By Lemma~\ref{lem:errorBound1} there exist~$\delta_z$ and~$\lambda_z$ such that,
for~$k \in \{\eta_{z-1} + B, \ldots, \eta_z\}$,
if~$x(k)$ satisfies
\begin{equation} \label{eq:residual_condition}
\|x(k) - \Pi_{\mathcal{X}}[x(k) - \nabla_xg\big(x(k); t_z\big)]\| \leq \delta_z,
\end{equation}
then
\begin{equation}
\min_{x^* \in \mathcal{X}^*(t_z)} \|x(k) - x^*\| \leq \lambda_z \|x(k) - \Pi_{\mathcal{X}}[x(k) - \nabla_xg\big(x(k); t_z\big)]\|.
\end{equation}
For~$x(k)$ satisfying~\eqref{eq:residual_condition}, applying Lemma~\ref{lem:noGammaUpdate} then gives
\begin{equation} \label{eq:l7main}
\min_{x^* \in \mathcal{X}^*(t_z)} \|x(k) - x^*\| \leq \lambda_z\min\{1, \gammaz^{-1}\} \|x(k) - \Pi_{\mathcal{X}}[x(k) - \gammaz\nabla_xg\big(x(k); t_z\big)]\|. 
\end{equation}

Next, using the triangle inequality we see that
for any~$\x{*}{k}{\tz} \in \argmin_{x^* \in \X} \|x(k) - x^*\|$
and any~$\x{*}{k+1}{\tz} \in \argmin_{x^* \in \X} \|x(k+1) - x^*\|$ we have 
\begin{multline} \label{eq:l7tri}
\|x^*_k(t_z) - x^*_{k+1}(t_z)\| \leq \|x^*_k(t_z) - x(k)\| + \|x(k) - x^*_{k+1}(t_z)\| \\
\leq 
\|x^*_k(t_z) - x(k)\| + \|x(k) - x(k + 1)\| + \|x(k + 1) - x^*_{k+1}(t_z)\|. 
\end{multline}
Applying~\eqref{eq:l7main} to~\eqref{eq:l7tri} and using the definition
of~$s(k)$ we find
\begin{align} \label{eq:l7callback}
\|x^*_k(t_z) - x^*_{k+1}(t_z)\| &\leq \lambda_z\min\{1, \gammaz^{-1}\}(\|s(k)\| + \|s(k+1)\|) + \|s(k)\| \\
&\leq (1 + \lambda_z\min\{1, \gammaz^{-1}\})(\|s(k)\| + \|s(k+1)\|). 
\end{align}

Next, summing~\eqref{eq:l5result} 
across~$k, k+B, \ldots, k+dB$ gives
\begin{equation}
g\big(x(k + dB); t_z) - g\big(x(k); t_z\big) \leq \frac{L_zNB}{2}\beta(k) 
- \frac{2 - \gammaz L_z(1 + B + 2NB)}{2\gammaz} \sum_{\ell=1}^{d} \beta(k + \ell B)
\end{equation}
for~$d \in \N$ such that~$k + dB \leq \eta_z$. Rearranging gives
\begin{equation} \label{eq:l7almost}
\sum_{\ell=1}^{d} \beta(k + \ell B) \leq \frac{\gammaz L_zNB}{2 - \gammaz L_z(1 + B + 2NB)}Bd_{\mathcal{X}}^2  
+ \frac{2\gammaz}{2 - \gammaz L_z(1 + B + 2NB)}g_{max,z},
\end{equation}
where we have used Lemma~\ref{lem:betaBound} 
and defined~$g_{max,z} = \max_{x \in \mathcal{X}} g(x; t_z)$. 
The right-hand side of~\eqref{eq:l7almost} is constant, and therefore
we must have~$s(k) \to 0$ as~$d$ grows on the left-hand side. In particular,
there exists some time~$\hat{k}_z$ after which
\begin{equation} \label{eq:l7skbound}
\|s(k)\| < \min\left\{\frac{\epsilon_z}{2(1 + \lambda_z\min\{1,\gammaz^{-1}\})}, \gammaz\delta_z\right\}.
\end{equation}
Then for~$k \geq \hat{k}_z$, we see that~\eqref{eq:residual_condition} holds by Lemma~\ref{lem:noGammaUpdate}, and therefore~\eqref{eq:l7callback} holds
as well. Substituting~\eqref{eq:l7skbound} into~\eqref{eq:l7callback} gives
%\begin{equation}
$\|x^*_k(t_z) - x^*_{k+1}(t_z)\| < \epsilon_z$. 
%\end{equation}
By Lemma~\ref{lem:solSeparation}, distinct elements of~$\mathcal{X}^*(t_z)$
are separated by at least~$\epsilon_z$ and hence~$g(x^*_k(t_z)) = g(x^*_{k+1}(t_z))$
for~$k \geq \hat{k}_z$. 
\end{proof}

\subsection{Proof of Lemma~\ref{lem:phase3}} \label{lem:linearConverge}
We first show that for any~$\gamma_1 \in \left( 0, \gamma_{\max,1} \right)$ and~$r_1 \in \N_0$ 
\begin{align}
    \at{\eta_{0} + \kone + r_1 B}{1} & \leq a_1 \pt{r_1-1}{1} \label{eq:3011} \\
    \bt{\eta_{0} + \kone +r_1 B} & \leq b_1 \pt{r_1-1}{1} \label{eq:3012}.
\end{align}
Following the same steps to reach~\eqref{eq:3006} we find
\begin{multline}
     \at{k+B}{1} \! \leq \! \bigg[ 1 + \frac{\gamma_1 D_1}{F_1} \bigg]^{-1}  \Bigg[ \bigg[ 1 - \gamma_1 \bigg[ \frac{G_1}{F_1} + \frac{E_1}{D_1} \bigg]\bigg] \at{k}{1} \\
    + \gamma_1 \bigg[ \frac{G_1}{F_1}
    + \frac{E_1}{D_1} \bigg] \Big[ \at{k-B}{1} 
    + E_1 \bt{k-B} \Big] \Bigg], \label{eq:3017}
\end{multline}
and following the same steps to reach~\eqref{eq:3007}, we find
% For any integer~$d \geq 2$, we have from repeated application of ~\eqref{eq:3013}
% \begin{multline}
%     \at{k+dB}{1} \leq \at{k}{1} - \gamma_1^{-1} D_1 \sum^d_{j=1} \bt{k+jB} \\
%     + E_1 \sum^{d-1}_{j=0} \bt{k+jB}.
% \end{multline}
% This expression is equivalent to
% \begin{multline}
%     \at{k+dB}{1} \leq \at{k}{1} - \gamma_1^{-1} D_1 \sum^{d-1}_{j=1} \bt{k+jB} \\
%     - \gamma_1^{-1} D_1 \bt{k+dB} + E_1 \sum^{d-1}_{j=1} \bt{k+jB} + E_1 \bt{k}\\ 
%     \leq \at{k}{1} - \big( \gamma_1^{-1} D_1 - E_1 \big) \sum^{d-1}_{j=1} \bt{k+jB} - \gamma_1^{-1} D_1 \bt{k+dB} \\
%     + E_1 \bt{k}.
% \end{multline}
% Taking~$\gamma_1 < D_1 / E_1$, we obtain from the nonnegativity of~$\bt{\tau}$, for all~$\tau$, that
% \begin{align}
%     \at{k+dB}{1} &\leq \at{k}{1} - \big( \gamma_1^{-1} D_1 - E_1 \big) \bt{k+B} + E_1 \bt{k}.
% \end{align}
% Due to the nonnegativity of~$\at{k+dB}{z}$ we can write
% \begin{align}
%     0 &\leq \at{k}{1} - \big( \gamma_1^{-1} D_1 - E_1 \big) \bt{k+B} + E_1 \bt{k}.
% \end{align}
% Rearranging gives us
\begin{align}
   \bt{k+B} &\leq \frac{\gamma_1}{ D_1 - \gamma_1 E_1} \big( \at{k}{1} + E_1 \bt{k} \big). \label{eq:3018}
\end{align}
We will use~\eqref{eq:3017} and~\eqref{eq:3018} to 
show that~\eqref{eq:3011} and~\eqref{eq:3012} hold for any~$r_1 \in \N_0$.
First, we show that~\eqref{eq:3011} and~\eqref{eq:3012} hold for~$r_1 = 0,1$. Specifically, we select~$a_1 > 0$ and~$b_1 > 0$ such that
\begin{align}
    \begin{array}{cc}
         \at{\eta_0 + \kone }{1} \leq a_1, & \at{\eta_0 + \kone + B}{1} \leq a_1,  \label{eq:3019}
    \end{array} \\
    \begin{array}{cc}
    \bt{\eta_0 + \kone} \leq b_1,  & \bt{\eta_0 + \kone + B} \leq b_1 . \label{eq:3020}
    \end{array}
\end{align}
By definition of~$\at{\eta_0 + \kone}{1}$ we have
\begin{equation} \label{eq:30}
    \at{\eta_0 + \kone}{1} = \g{\x{}{}{\eta_0 + \kone}}{\ind{}{1}} - \g{\x{*}{\kone}{\ind{}{1}}}{\ind{}{1}}. 
\end{equation}
%where~$\x{*}{\kone}{\ind{}{1}}$ for all~$k \in \{ \eta_0 + \kone, \dots, \eta_1 \}$.
By Lemma~\ref{lem:designGamma} we have~$\g{\x{}{}{\eta_0 + \kone }}{\ind{}{1}} < \g{\x{}{}{\eta_0 + B}}{\ind{}{1}}$, and thus
\begin{equation}
    \at{\eta_0 + \kone}{1} \leq \g{\x{}{}{\eta_0 + B}}{\ind{}{1}} - \g{\x{*}{\kone}{\ind{}{1}}}{\ind{}{1}}.
\end{equation}
Using Lemma~\ref{lem:phase1} and Lemma~\ref{lem:jump} we obtain
% \begin{multline}
%     \at{\eta_0 + \kone}{1} \leq \g{\x{}{}{\eta_0}}{\ind{}{1}} + B \dx  M_{x,1} \\
%     + \frac{L_1 B^2 \dx^2}{2}  - \g{\x{*}{\kone}{\ind{}{1}}}{\ind{}{1}}.
% \end{multline}
% By Lemma~\ref{lem:jump} we have~$\g{\x{}{}{\eta_0}}{\ind{}{1}} \leq \g{\x{}{}{\eta_0}}{t_0} + L_t \Delta$. Therefore
%\begin{multline}
%    \at{\eta_0 + \kone}{1} \leq \g{\x{}{}{\eta_0}}{t_0} + L_t \Delta + B \dx  M_{x,1} \\
%    + \frac{L_1 B^2 \dx^2}{2}  - \g{\x{*}{\kone}{\ind{}{1}}}{\ind{}{1}}.
%\end{multline}
%and thus
%Adding zero as~$\g{\x{*}{\eta_0}{t_0}}{t_0} - \g{\x{*}{\eta_0}{t_0}}{t_0}$ and~$\g{\x{*}{\eta_0}{t_0}}{t_1} - \g{\x{*}{\eta_0}{t_0}}{t_1}$ yields
\begin{multline}
    \at{\eta_0 + \kone}{1} \leq \g{\x{}{}{\eta_0}}{t_0} - \g{\x{*}{\kone}{\ind{}{1}}}{t_0} + L_t \Delta 
    + B \dx  M_{1} + \frac{L_1 B^2 \dx^2}{2} \\
    + \g{\x{*}{\kone}{\ind{}{1}}}{t_0} - \g{\x{*}{\eta_0}{t_0}}{t_1} 
    + \g{\x{*}{\eta_0}{t_0}}{t_1} - \g{\x{*}{\kone}{\ind{}{1}}}{\ind{}{1}}.
\end{multline}
Taking the absolute value gives 
\begin{multline}
    \at{\eta_0 + \kone}{1} \leq  L_t \Delta + B \dx  M_{1} + \frac{L_1 B^2 \dx^2}{2} 
    + | \g{\x{}{}{\eta_0}}{t_0} - \g{\x{*}{\eta_0}{t_0}}{t_0} | \\
    + | \g{\x{*}{\eta_0}{t_0}}{t_0} 
    - \g{\x{*}{\eta_0}{t_0}}{t_1} |
    + | \g{\x{*}{\eta_0}{t_0}}{t_1} - \g{\x{*}{\kone}{\ind{}{1}}}{\ind{}{1}} |.
\end{multline}
Rearranging Lemma~\ref{lem:jump} and taking the absolute value of both sides gives us~$| \g{x}{t_z} - \g{x}{t_{z+1}} | \leq L_t \Delta$ for all~$x\in \X$ and~$\tz \in \Ts$. Applying this bound, the Lipschitz continuity of~$\g{\cdot}{\ind{}{1}}$ from
Section~\ref{ss:p2objectives}, and Assumption~\ref{ass:sigma} we have
% \begin{multline}
%     \at{\eta_0 + \kone}{1} \leq  2 L_t \Delta + B \dx  M_{x,1} + \frac{L_1 B^2 \dx^2}{2} \\
%     + | \g{\x{}{}{\eta_0}}{t_0} - \g{\x{*}{\eta_0}{t_0}}{t_0} | \\
%     + | \g{\x{*}{\eta_0}{t_0}}{t_1} - \g{\x{*}{\kone}{\ind{}{1}}}{\ind{}{1}} |.
% \end{multline}
% By the Lipschitz continuity of~$\g{\cdot}{\ind{}{1}}$ in~\eqref{eq:1004} we have
% \begin{multline}
%     \at{\eta_0 + \kone}{1} \leq  2 L_t \Delta + B \dx  M_{x,1} + \frac{L_1 B^2 \dx^2}{2} \\
%     + | \g{\x{}{}{\eta_0}}{t_0} - \g{\x{*}{\eta_0}{t_0}}{t_0} | 
%     + \Lg{1} \| \x{*}{\eta_0}{t_0} - \x{*}{\kone}{\ind{}{1}} \|.
% \end{multline}
% From the bounded minimizer movement in Assumption~\ref{ass:time} we find
\begin{equation}
    \at{\eta_0 + \kone}{1} \leq  2 L_t \Delta + B \dx  M_{1} + \frac{L_1 B^2 \dx^2}{2} + \Lg{1} \sigma_1 
    + | \g{\x{}{}{\eta_0}}{t_0} - \g{\x{*}{\eta_0}{t_0}}{t_0} | .
\end{equation}
Using the definition of~$\at{\eta_0}{0}$ and the fact that~$\at{\eta_0}{0} \geq 0$, we have
%\begin{multline}
    $\at{\eta_0 + \kone}{1} \leq  2 L_t \Delta + B \dx  M_{1} + \frac{L_1 B^2 \dx^2}{2} + \Lg{1} \sigma_1 
    + \at{\eta_0}{0}$.
%\end{multline}
We have~$\eta_0 = \kzero + r_0 B$, and therefore we can apply the result of Lemma~\ref{lem:time0converge} to obtain
\begin{equation}
    \at{\eta_0 + \kone}{1} \leq a_0 \pt{r_0 -1}{0} +  2 L_t \Delta + B \dx  M_{1}  + \Lg{1} \sigma_1 
    + \frac{L_1 B^2 \dx^2}{2} . \label{eq:3032}
\end{equation}
Next, we modify the right-hand side of~\eqref{eq:3032} to design~$\frac{b_1}{a_1}$. 
We observe that~$8E_1\left(\frac{G_1}{F_1} + \frac{E_1}{D_1}\right)F_1 \geq 1$.
Indeed, we have the bound~$8E_1\left(\frac{G_1}{F_1} + \frac{E_1}{D_1}\right)F_1 = 8E_1G_1 + 8\frac{E_1^2F_1}{D_1} \geq 1$,
% \begin{align}
%     8E_1\left(\frac{G_1}{F_1} + \frac{E_1}{D_1}\right)F_1 = 8E_1G_1 + 8\frac{E_1^2F_1}{D_1} \geq 1,
% \end{align}
which follows by inspection of~$8E_1G_1$.
For all~$\gamma_1 \in \left( 0 ,  \frac{2}{ L_1 (1 + B + BN)} \right)$ we have~$D_1 \leq 1$. Then we have~$\frac{D_1}{8E_1\left(\frac{G_1}{F_1} + \frac{E_1}{D_1}\right)F_1} \leq 1$ as well.
% \begin{align}
%     \frac{D_1}{8E_1\left(\frac{G_1}{F_1} + \frac{E_1}{D_1}\right)F_1} \leq 1.
% \end{align}
Multiplying $\frac{8E_1\left(\frac{G_1}{F_1} + \frac{E_1}{D_1}\right)F_1}{D_1} \geq 1$ with the last term in~\eqref{eq:3032} gives
\begin{equation}
    \at{\eta_0 + \kone}{1} 
    \leq  2 L_t \Delta + B \dx  M_{1}  + \Lg{1} \sigma_1
    + a_0 \pt{r_0 -1}{0} 
    + 4 L_1 B^2 \dx^2 E_1\left(\frac{G_1}{F_1} + \frac{E_1}{D_1}\right)F_1 D_1^{-1} . 
\end{equation}
Lemma~\ref{lem:designGamma} gives~$\g{\x{}{}{\eta_0 + \kone + B}}{\ind{}{1}} \leq \g{\x{}{}{\eta_0 + \kone}}{\ind{}{1}}$ and thus~$\at{\eta_0 + \kone + B}{1} \leq \at{\eta_0 + \kone}{1}$.
Thus, for
\begin{equation} \label{eq:31}
    a_1 = a_0 \pt{r_0 -1}{0} + 2 L_t \Delta + B \dx  M_{1}  + \Lg{1} \sigma_1 
    +  4 L_1 B^2 \dx^2 E_1\left(\frac{G_1}{F_1} + \frac{E_1}{D_1}\right)F_1 D_1^{-1}
\end{equation}
we satisfy ~\eqref{eq:3019}. Furthermore if we select~$b_1 = B \dx^2$ according to Lemma~\ref{lem:betaBound},
then we satisfy ~\eqref{eq:3020}.
From these selections of~$a_1$ and~$b_1$ we have~$\frac{b_1}{a_1} \leq \frac{D_1}{  8E_1\left(\frac{G_1}{F_1} + \frac{E_1}{D_1}\right)F_1}$.
% \begin{multline}
%     \!\!\!\!\!\! \frac{b_1}{a_1} = \frac{B \dx^2}{a_0 \pt{r_0 -1}{0} + 2 L_t \Delta + B \dx  M_{x,1}  + \Lg{1} \sigma_1  +  \frac{L_1 B^2 \dx^2 8E_1\left(\frac{G_1}{F_1} + \frac{E_1}{D_1}\right)F_1}{2 D_1}} \\
%     \leq \frac{B \dx^2}{ \frac{L_1 B^2 \dx^2 8E_1\left(\frac{G_1}{F_1} + \frac{E_1}{D_1}\right)F_1}{2 D_1}}, \\
%     \leq   \frac{D_1}{  8E_1\left(\frac{G_1}{F_1} + \frac{E_1}{D_1}\right)F_1}
%     \frac{2}{L_1 B}, \\
%     \leq   \frac{D_1}{  8E_1\left(\frac{G_1}{F_1} + \frac{E_1}{D_1}\right)F_1},
% \end{multline}
% where the last line holds by hypothesis.
Then these choices of~$a_1,b_1$ satisfy~\eqref{eq:3011} and~\eqref{eq:3012} for~$r_1 = 0,1$.

Next we prove that ~\eqref{eq:3011} and~\eqref{eq:3012} hold for all~$r_1 \in \N_0$ by induction. 
For the inductive hypothesis suppose~\eqref{eq:3011} and~\eqref{eq:3012} hold for all~$r_1$ up to~$d \geq 1$.
Using the same steps from~\eqref{eq:29} to~\eqref{eq:3009}, we have~$\at{\eta_0+\kone+dB+B}{1} \leq a_1 \pt{d}{1}$, and this completes induction on~\eqref{eq:3011}.
To complete the inductive argument for~$\beta$ we make a similar argument to reach~\eqref{eq:3012}. 
From~\eqref{eq:3018} we have
\begin{equation} \label{eq:t2_beta1}
   \bt{\eta_0 + \kone+dB\!+\!B} \leq \frac{\gamma_1}{ D_1 - \gamma_1 E_1} \big( \at{\eta_0 + \kone + dB}{1} 
   + E_1 \bt{\eta_0 +\kone+dB} \big). 
\end{equation}
% By the inductive hypothesis we have
% \begin{align}
%    \bt{\eta_0 +\kone+dB+B} &\leq \frac{\gamma_1}{ D_1 - \gamma_1 E_1} \big( a_1 \pt{d-1}{1} + E_1 b_1 \pt{d-1}{1} \big). 
% \end{align}
% Factoring out~$ b_1 \pt{d-1}{1}$ we get
% \begin{align}
%    \bt{\eta_0 +\kone+dB+B} &\leq \frac{\gamma_1 \big( \frac{a_1}{b_1} + E_1 \big)}{ D_1 - \gamma_1 E_1} b_1 \pt{d-1}{1}.
% \end{align}
Next, we have~$\bt{\eta_0 +\kone+dB+B} \leq \frac{\gamma_1 \big( \frac{a_1}{b_1} + E_1 \big)}{ D_1 - \gamma_1 E_1} b_1 \pt{d-1}{1}$ by the inductive hypothesis.
We then design~$\gamma_1 < \frac{ \frac{a_1}{b_1} + 2E_1 + D_1 c_1 - \sqrt{\big( \frac{a_1}{b_1} + 2E_1 + D_1 c_1\big)^2 - 4 D_1 E_1 c_1}}{2 E_1 c_1}$ to ensure that~$\gamma_1 \Big( \frac{a_1}{b_1} + E_1 \Big) \leq \big(  D_1 - \gamma_1 E_1 \big) \big( 1 - \gamma_1 c_1 \big)$ so that we obtain~$\bt{\kone+dB+B} \leq \big( 1 - \gamma_1 c_1 \big) b_1 \pt{d-1}{1}$.
% \begin{align}
%    \bt{\kone+dB+B} &\leq \big( 1 - \gamma_1 c_1 \big) b_1 \pt{d-1}{1}. \\
% \end{align}
% Then from the definition of~$\pt{}{1}=1 - \gamma_1 c_1$ in ~\eqref{eq:34} we get
Therefore
\begin{equation} \label{eq:t2_beta2}
   \bt{\kone+dB+B} \leq b_1 \pt{d}{1}.
\end{equation}
Thus,~\eqref{eq:3012} holds for all~$r_1 \in \N_0$.
% Therefore ~\eqref{eq:3011} and~\eqref{eq:3012} hold for all~$r_1 \in \N_0$.

The preceding establishes the base case for the next inductive argument in which we show that if~$\at{\eta_{z-1} + \kz + \rz B}{z}  \leq \az \pt{\rz-1}{z}$ and~$\bt{\eta_{z-1} + \kz + \rz B} \leq \bz \pt{\rz-1}{z}$ hold for a fixed~$\tz$, then
% Next we show that if
% \begin{align}
%     \at{\eta_{z-1} + \kz + \rz B}{z} & \leq \az \pt{\rz-1}{z} \label{eq:3022} \\
%     \bt{\eta_{z-1} + \kz + \rz B} & \leq \bz \pt{\rz-1}{z} 
% \end{align}
% hold for a fixed~$\tz$, then 
\begin{align}
    \at{\eta_{z} + \kzz + r_{z+1} B}{z + 1} & \leq \azz \pt{\rzz-1}{z + 1} \label{eq:3023}\\
    \bt{\eta_{z} + \kzz + r_{z+1} B} & \leq \bzz \pt{\rzz-1}{z + 1} \label{eq:3024}
\end{align}
also hold. 
This will complete our proof of Lemma~\ref{lem:phase3}.
From Lemma~\ref{lem:phase2} and Lemma~\ref{lem:moreBound} for any~$k \in \{ \eta_z + \kz, \dots, \eta_{z+1} - B\}$,
\begin{align}
    \at{k+B}{z+1}  &\leq \at{k}{z+1} - \gammazz^{-1} \Dzz \bt{k+B} + \Ezz \bt{k}, \label{eq:3025} \\
    \at{k + B}{z+1} &\leq \gammazz^{-2}\Fzz \bt{k + B} + \gammazz^{-1} \Gzz \bt{k}. \ \ \ \label{eq:3026}
\end{align}
We rearrange~\eqref{eq:3026} %, we obtain
to lower bound~$\bt{k+B}$, and using that bound in
%\begin{align}
%     \bt{k+B} &\geq \frac{\at{k+B}{z+1}  - \gammazz^{-1} \Gzz \bt{k}}{\gammazz^{-2}\Fzz}.
%\end{align}
%Applying this to 
\eqref{eq:3025} gives
% \begin{multline}
%     \at{k+B}{z+1} \leq \at{k}{z+1} - \gammazz^{-1}\Dzz \bt{k+B} + \Ezz \bt{k} \\
%     \leq \at{k}{z+1} 
%     - \gammazz^{-1}\Dzz \bigg( \frac{\at{k+B}{z+1}  - \gammazz^{-1} \Gzz \bt{k}}{\gammazz^{-2}\Fzz} \bigg) \\
%     + \Ezz \bt{k}.
% \end{multline}
% Rearranging we get
\begin{equation}
    % \at{k+B}{z+1} \leq \at{k}{z+1} - \frac{\gammazz \Dzz}{\Fzz}\at{k+B}{z+1} \\
    % + \frac{\Dzz \Gzz}{\Fzz} \bt{k} + \Ezz \bt{k} \\
    % \leq \at{k}{z+1} - \frac{\gammazz \Dzz}{\Fzz}\at{k+B}{z+1} + \bigg(  \frac{\Dzz \Gzz}{\Fzz} \\
    % + \Ezz \bigg) \bt{k} \\
    \bigg( 1 + \frac{\gammazz \Dzz}{\Fzz} \bigg) \at{k+B}{z+1} \leq \at{k}{z+1} 
    + \bigg(  \frac{\Dzz \Gzz}{\Fzz} 
    + \Ezz \bigg) \bt{k}. \label{eq:3027}
\end{equation}
Fix any~$k \in \{\eta_{z-1} + \kz + B,\dots,  \eta_z - B \}$ and 
% such that~$\eta_{z-1} + \kz + B \leq k \leq \eta_z - B$
replace~$k$ by~$k-B$ in~\eqref{eq:3025}. Then~$\bt{k} \leq  \frac{\gammazz}{\Dzz} (\at{k-B}{z} -\at{k}{z+1} + \Ezz \bt{k-B})$.
% \begin{align}
%     \bt{k} &\leq \gammazz \frac{ \at{k-B}{z} -\at{k}{z+1} + \Ezz \bt{k-B}}{\Dzz}. \label{eq:3028}
% \end{align}
Applying this to~\eqref{eq:3027} gives
% \begin{multline}
%     \bigg( 1 + \frac{\gammazz \Dzz}{\Fzz} \bigg) \at{k+B}{z+1} \leq \at{k}{z+1} \\
%     + \gammazz \bigg(  \frac{\Dzz \Gzz}{\Fzz}
%     + \Ezz \bigg) \\
%     \bigg( \frac{ \at{k-B}{z} -\at{k}{z+1} + \Ezz \bt{k-B}}{\Dzz} \bigg).
% \end{multline}
% Rearranging we get
\begin{multline}
     \at{k+B}{z+1} \leq \bigg[ 1 + \frac{\gammazz \Dzz}{\Fzz} \bigg]^{-1}  \Bigg( \bigg[1 - \gammazz \bigg[ \frac{\Gzz}{\Fzz} 
     + \frac{\Ezz}{\Dzz} \bigg]\bigg] \at{k}{z+1} \\
    + \gammazz \bigg[ \frac{\Gzz}{\Fzz}
    + \frac{\Ezz}{\Dzz} \bigg] \cdot 
    \Big[ \at{k-B}{z} + \Ezz \bt{k-B} \Big] \Bigg). \label{eq:3029}
\end{multline}
For any integer~$d \geq 2$, we iterate~\eqref{eq:3025} to find
% \begin{multline}
%     \at{k+dB}{z} \leq \at{k}{z+1} - \gammazz^{-1} \Dzz \sum^d_{j=1} \bt{k+jB} \\
%     + \Ezz \sum^{d-1}_{j=0} \bt{k+jB}.
% \end{multline}
% This expression is equivalent to
\begin{multline}
    \at{k+dB}{z+1} 
    % \leq \at{k}{z+1} - \gammazz^{-1} \Dzz \sum^{d-1}_{j=1} \bt{k+jB} \\
    % - \gammazz^{-1} \Dzz \bt{k+dB} + \Ezz \sum^{d-1}_{j=1} \bt{k+jB} + \Ezz \bt{k}\\ 
     \leq \at{k}{z+1} - \gammazz^{-1} \Dzz \bt{k+dB} \\
     + \Ezz \bt{k}
     - \big( \gammazz^{-1} \Dzz - \Ezz \big) \sum^{d-1}_{j=1} \bt{k+jB} 
     .
\end{multline}
Since~$\gammazz < \Dzz / \Ezz$ and~$\bt{\tau} \geq 0$ for all~$\tau>0$, 
\begin{equation}
    \at{k+dB}{z+1} \leq \at{k}{z+1} + \Ezz \bt{k} - \big( \gammazz^{-1} \Dzz - \Ezz \big) \bt{k+B}.
\end{equation}
Using~$\at{k+dB}{z} \geq 0$ we rearrange to find
% \begin{align}
%     0 &\leq \at{k}{z+1} - \big( \gammazz^{-1} \Dzz - \Ezz \big) \bt{k+B} + \Ezz \bt{k}.
% \end{align}
% Rearranging gives us
\begin{equation}
   \bt{k+B} \! \leq \! \frac{\gammazz}{ \Dzz \! - \! \gammazz \Ezz} \! \big( \at{k}{z+1} \! + \! \Ezz \bt{k} \big). \label{eq:3030}
\end{equation}

We will use~\eqref{eq:3029} and~\eqref{eq:3030} to show that~\eqref{eq:3023} and~\eqref{eq:3024} hold for any~$\rzz \in \N_0$. 
First we show that~\eqref{eq:3025} and~\eqref{eq:3026} hold for~$\rzz = 0,1$. Specifically, we select~$\azz > 0$ and~$\bzz > 0$ such that
% \blue{Can't figure out why (61) overlaps equation}
\begin{align}
      \alpha (\eta_z + \kzz;\tzz) \leq \azz, \,\,\,\,\,\, \at{\eta_z + \kzz + B}{z+1} \leq \azz \ \ \label{eq:3034} \\
    \beta (\eta_z + \kzz )\leq \bzz,  \ \ \  \bt{\eta_z + \kzz + B} \leq \bzz . \ \ \ \label{eq:3035}
\end{align}
% By definition of~$\at{\eta_z + \kzz}{z+1}$ we have
% \begin{multline}
%     \at{\eta_z + \kzz}{z+1} = \g{\x{}{}{\eta_z + \kzz}}{\tzz} \\
%     - \g{\x{*}{\kzz}{\tzz}}{\tzz},
% \end{multline}
% where~$\x{*}{\kzz}{\tzz} = \underset{\x{*}{}{} \in \X^*(\tzz)}{\argmin} \| \x{}{}{\kzz} - \x{*}{}{} \|$ is the unique minimizer for all~$k \in \{ \eta_z + \kzz, \dots, \eta_{z+1} \}$.
% \gb{Note to self: Check out this sentence.}
We repeat the steps from~\eqref{eq:30} to~\eqref{eq:31} to find that if we select
\begin{multline}
    \azz = \az \pt{r_z -1}{z} + 2 L_t \Delta + B \dx  \Mxx  + \Lg{z+1} \sigma_{z+1} \\
    +  4 \Lxx B^2 \dx^2 \Ezz\left(\frac{\Gzz}{\Fzz} + \frac{\Ezz}{\Dzz}\right)\Fzz \Dzz^{-1}, 
\end{multline}
then~\eqref{eq:3034} holds. Selecting~$\bzz = B \dx^2$ satisfies~\eqref{eq:3035} due to Lemma~\ref{lem:betaBound}.
From these selections of~$\azz$ and~$\bzz$ we have the bound~$\frac{\bzz}{\azz} \leq   \frac{\Dzz}{  8\Ezz\left(\frac{\Gzz}{\Fzz} + \frac{\Ezz}{\Dzz}\right)\Fzz}$.
% \begin{align}
%     \frac{\bzz}{\azz} &= \frac{B \dx^2}{\az \pt{r_z -1}{z} + 2 L_t \Delta + B \dx  \Mxx  + \Lg{z+1} \sigma_{z+1}  +  \frac{\Lxx B^2 \dx^2 8\Ezz\left(\frac{\Gzz}{\Fzz} + \frac{\Ezz}{\Dzz}\right)\Fzz}{2 \Dzz}} \\
%     &\leq \frac{B \dx^2}{ \frac{\Lxx B^2 \dx^2 8\Ezz\left(\frac{\Gzz}{\Fzz} + \frac{\Ezz}{\Dzz}\right)\Fzz}{2 \Dzz}}, \\
%     &\leq   \frac{\Dzz}{  8\Ezz\left(\frac{\Gzz}{\Fzz} + \frac{\Ezz}{\Dzz}\right)\Fzz}
%     \frac{2}{\Lxx B}, \\
%     &\leq   \frac{\Dzz}{  8\Ezz\left(\frac{\Gzz}{\Fzz} + \frac{\Ezz}{\Dzz}\right)\Fzz},
% \end{align}
% where the last line holds by hypothesis.
These selections of~$\azz,\bzz$ satisfy ~\eqref{eq:3023} and~\eqref{eq:3024} for~$\rz = 0,1$. 
Next we prove that ~\eqref{eq:3025} and~\eqref{eq:3026} hold for all~$\rzz \in \N_0$ by induction. For the inductive hypothesis suppose~\eqref{eq:3025} and~\eqref{eq:3026} hold for all~$\rzz$ up to~$d \geq 1$.
Using the same steps used from~\eqref{eq:29} to~\eqref{eq:3009}
we obtain~$\at{\eta_z+\kzz+dB+B}{z+1} \leq \azz \pt{d}{z+1}$, and this completes the induction on~\eqref{eq:3023}.

To complete the inductive argument for~$\beta$ we make a similar argument to reach~\eqref{eq:3024}.
Following the same steps used to go from~\eqref{eq:t2_beta1} to~\eqref{eq:t2_beta2}, 
we find~$\bt{\eta_z + \kzz+dB+B} \leq \bzz \pt{d}{z+1}$.
Therefore,~\eqref{eq:3024} holds for~$z+1$.
\hfill $\blacksquare$

\begin{lemma} \label{lem:jump}
    For all~$x \in \X$ and~$\tz \in \Ts$, we have the bound 
    %\begin{align}
        $\g{x}{\tzz} \leq \g{x}{\tz} + L_t \Delta$,
    %\end{align}
    where we have defined  %we have used the constants
    %\begin{align}
        $L_t \coloneqq \frac{1}{2} \max_{x \in \X} \| x \|^2 \sum^N_{i=1} L_Q^{[i]} 
    + \max_{x \in \X} \| x \| \sum^N_{i=1} L_r^{[i]}$, and $\Delta>0$ is from Assumption~\ref{ass:sample}.
    %\end{align}
\end{lemma}

\begin{proof}
From Problem~\ref{prob:aggregate} we have
\begin{equation}
    \g{x}{\tzz} - \g{x}{\tz}  = \frac{1}{2} x^T (\Upsilon (Q(\tzz)) - \Upsilon (Q(\tz)) ) x 
    + (\Xi (r(\tzz)) - \Xi (r(\tz)))^T x.
\end{equation}
    Applying the Cauchy-Schwarz inequality and the fact that 
    %matrix norm property~$\| Qx \| \leq \| Q \| \| x \|$, and using 
    $\| x^{[i]} \| \leq \| x \| \leq \max_{x \in \X} \| x \|$ for all~$i\in [N]$ gives 
\begin{multline}
    \g{x}{\tzz} - \g{x}{\tz} \leq 
    \frac{1}{2} \max_{x \in \X} \| x \|^2 \sum^N_{i=1} \| \Q{i}{\theta_i \left( \tzz \right)} - \Q{i}{\theta_i \left( \tz \right)} \| \\
    + \max_{x \in \X} \| x \| \sum^N_{i=1} \| \rr{i}{\theta_i \left( \tzz \right)} - \rr{i}{\theta_i \left( \tz \right)} \|.
\end{multline}
Applying Assumptions~\ref{ass:lipTime} and~\ref{ass:sample} yields the result.
\end{proof}

%\input{Lemmas/1-jorgeBound}
%\input{Lemmas/15-relatingProblems}
%\subsection{Proof of Corollary~\ref{cor:synch1}} \label{lem:synch1}
%From~\eqref{eq:1020} we have
%\begin{equation}
%    | \g{\x{}{}{\eta_z}}{\tz} - f^*(\tz) |
%    \leq \az{} \rhoz^{\rz-1}  + | \g{\x{*}{k}{\tz}}{\tz} - f^*(\tz) |. 
%%\end{equation}
%since~$\g{\cdot}{\tz} = \f{\cdot}{\tz}$ for~$\Ts = \T$, we have~$\g{\x{*}{k}{\tz}}{\tz} - f^*(\tz) = 0$ for all~$\tz \in \Ts$, which gives the result. 
%\hfill $\blacksquare$
%\red{Put this proof in the main body.}

% \end{appendice}
%\subsection{Proof of Corollary~\ref{cor:synch2}} \label{lem:synch2}
%For all~$\tz \in \Ts$ and~$k \in \{ \eta_{z-1},\dots,\eta_{z} \}$, we have~$\g{\cdot}{\tz} = \f{\cdot}{\tz} \in \mathcal{S}$. Then, for all~$k \in \{ \eta_{z-1},\dots,\eta_{z}-1  \}$, we have~$\x{*}{}{k} - \x{*}{}{k+1} = 0$ due to the strong convexity of~$\g{\cdot}{\tz}$. Then~$\kz = 0$ for all~$\tz \in \Ts$. Using that fact and~$B=1$ 
%(due to synchrony), Lemma~\ref{lem:phase3} gives
%    %\begin{align}
%        $\at{ \eta_{z-1} + \rz }{z} \leq \az \pt{\rz-1}{z}$,
%    %\end{align}
%    for any~$\rz \in \{0,\dots, \kappa_z \}$ and for all~$\tz \in \Ts$. 
%    Setting~$\rz = \kappa_z $ yields the result. \hfill $\blacksquare$
%\red{Put this proof in the main body as well.}

% \end{appendice}
%\end{appendices}

\bibliographystyle{IEEEtran}{}
\bibliography{references}

\end{document}